%% file: main.tex
\documentclass[reqno]{amsart}

\input{preamble}

\begin{document}

\title{Nonproduct minimal submanifolds \\ with cylindrical tangent cones}

\author{Rafe Mazzeo}
\address{Department of Mathematics, Stanford University}
\email{rmazzeo@stanford.edu}

\author{Gregory J. Parker}
\address{Department of Mathematics, University of California, Berkeley}
\email{gjparker@berkeley.edu}

\begin{abstract}
This article constructs, via perturbation methods, infinite-dimensional families of non-product minimal 
submanifolds which have singular sets locally modeled on cylindrical tangent cones. This answers
a question raised by Leon Simon in the 1980s, and generalizes work Caffarelli-Hardt-Simon, who treated 
the case of isolated singularities.  This setting requires substantially more work, however, since the 
natural obstruction space for the problem is infinite dimensional and the map which is meant to 
compensate for this obstruction using deformations of the singular set exhibits a loss of regularity. This necessitates working in the category of tame Frechet spaces and 
using the Nash-Moser Implicit Function Theorem.
\end{abstract}

\maketitle
\tableofcontents

%%
%% Body: each \input pulls in one section file.
%% To work on a single section, open its file directly.
%% Labels and \ref{} cross-references work globally because
%% all files are compiled together through this master document.
%%

\input{sections/sec_introduction}

%\partbreaktoc{Part I: The Jacobi Operator}
\input{sections/sec_chs_deformations}

\input{sections/sec_edge_calculus}
\input{sections/sec_infinite_obstruction}

%\input{sections/sec_function_spaces_streamlined}
%\input{sections/sec_linear_edge_theory}
%\input{sections/sec_poisson_operator}
%\input{sections/sec_boundary_deformation}

%\partbreaktoc{Part II: Nash-Moser Theory}

\input{sections/PartII_intro}
\input{sections/sec_nash_moser}

\input{sections/sec_perturbations}

\input{sections/sec_deformation_operator}

\input{sections/sec_perturbed_obstructions_quicker}

\input{sections/sec_tame_estimates}
\input{sections/sec_analyticspaces}

\input{sections/sec_ellipticanalytic}

\input{sections/sec_arclengthgauge}
\input{sections/sec_analyticbootstrapping}

%\partbreaktoc{Part III: Elliptic Edge Theory on Analytic Spaces}

%\input{sections/sec_deformations_singsets}
%\input{sections/sec_deformations_minimal}
%\input{sections/sec_uniformization_gauge}
%\input{sections/sec_perturbed_obstruction}

%%
%% References
%%
{\small
\medskip
\bibliographystyle{amsalpha}
{\small 
\bibliography{Bibliography_CylindricalTC}}
}

\end{document}

%% file: preamble.tex
\usepackage{amsthm}   
\usepackage{amsfonts, amsmath, amscd}
\usepackage{amssymb, latexsym}
\usepackage[all,cmtip]{xy}
\usepackage{enumerate} 
\usepackage{color}
\usepackage[usenames,dvipsnames]{xcolor}
\usepackage{verbatim}
\usepackage{upgreek}
\usepackage{hyperref}
\hypersetup{colorlinks=false}
\usepackage{cite}
\usepackage{enumitem}

\usepackage[OT2,T1]{fontenc}

\usepackage{graphicx, psfrag}
\usepackage{pstool}

\usepackage{xspace}
\usepackage{tikz-cd}
\usepackage{tikz}
\usetikzlibrary{calc}
\usetikzlibrary{arrows,shapes,decorations.pathmorphing}
\usetikzlibrary{decorations.pathreplacing}

\usepackage[titletoc]{appendix}
\usepackage[normalem]{ulem}
\usepackage{amsmath}
\usepackage{amscd}
\usepackage[latin2]{inputenc}
\usepackage{t1enc}
\usepackage[mathscr]{eucal}
\usepackage{indentfirst}
\usepackage{graphicx}
\usepackage{graphics}
\usepackage{pict2e}
\usepackage{epic}
\usepackage{wrapfig}
\usepackage{subeqnarray}
\usepackage{yfonts}
\usepackage[makeroom]{cancel}
\usepackage{trimclip}

\newif\iflclip
\newif\ifbclip
\newif\ifrclip
\newif\iftclip
\def\CLIP{\dimexpr\fboxrule+.2pt\relax}
\def\nulclip{0pt}
\newcommand\partbox[2]{%
  \lclipfalse\bclipfalse\rclipfalse\tclipfalse%
  \let\lkern\relax\let\rkern\relax%
  \let\lclip\nulclip\let\bclip\nulclip\let\rclip\nulclip\let\tclip\nulclip%
  \parseclip#1\relax\relax%
  \iflclip\def\lkern{\kern\CLIP}\def\lclip{\CLIP}\fi
  \ifbclip\def\bclip{\CLIP}\fi
  \ifrclip\def\rkern{\kern\CLIP}\def\rclip{\CLIP}\fi
  \iftclip\def\tclip{\CLIP}\fi
  \lkern\clipbox{\lclip{} \bclip{} \rclip{} \tclip}{\fbox{#2}}\rkern%
}
\def\parseclip#1#2\relax{%
  \ifx l#1\lcliptrue\else
  \ifx b#1\bcliptrue\else
  \ifx r#1\rcliptrue\else
  \ifx t#1\tcliptrue\else
  \fi\fi\fi\fi
  \ifx\relax#2\relax\else\parseclip#2\relax\fi
}
\usepackage{slashed} 

\usepackage{mathtools}
\usepackage{colonequals}
\usepackage{booktabs}
\usepackage[numbers]{natbib}
\usepackage{fancyhdr}
\usepackage{multirow}
\usepackage{braket}
\usepackage[margin=2.9cm]{geometry}
\usepackage{mathabx}
\usepackage{esint}
\usepackage{mathrsfs}
\usepackage{blkarray}
\usepackage{pinlabel}
\usepackage{scalerel,stackengine}
\stackMath
\newcommand\reallywidehat[1]{%
\savestack{\tmpbox}{\stretchto{%
  \scaleto{%
    \scalerel*[\widthof{\ensuremath{#1}}]{\kern-.6pt\bigwedge\kern-.6pt}%
    {\rule[-\textheight/2]{1ex}{\textheight}}%
  }{\textheight}% 
}{0.5ex}}%
\stackon[1pt]{#1}{\tmpbox}%
}
\newcommand{\R}{\mathbb{R}}

\newcommand{\Z}{\mathbb{Z}}

\newcommand{\N}{\mathbb{N}}

\newcommand{\C}{\mathbb{C}}

\newcommand{\br}{\langle}
\newcommand{\kt}{\rangle}

\theoremstyle{plain}
\newtheorem{thm}{Theorem}[section]
\newtheorem{prop}[thm]{Proposition}
\newtheorem{lm}[thm]{Lemma}
\newtheorem{cor}[thm]{Corollary}
\newtheorem{conj}[thm]{Conjecture}

\theoremstyle{definition}

\newtheorem{defn}[thm]{Definition}

\theoremstyle{remark}

\newtheorem{rem}[thm]{Remark}

\newtheorem{notation}[thm]{Notation}

\newtheorem*{hyp1*}{Hypothesis (I)}
\newtheorem*{hyp2*}{Hypothesis (II)}
\newtheorem*{hyp3*}{Hypothesis (III)}

\newtheoremstyle{exercise}{}{}{\itshape}{}{\bfseries}{:}{.5em}{\thmname{#1} \thmnumber{#2}\thmnote{(#3)}}
\theoremstyle{exercise}

\usepackage{amsfonts}
\DeclareFontFamily{U}{wncy}{}
\DeclareFontShape{U}{wncy}{m}{n}{<->wncyr10}{}
\DeclareSymbolFont{mcy}{U}{wncy}{m}{n}
\DeclareMathSymbol{\sha}{\mathord}{mcy}{"58} 

\numberwithin{equation}{section}

\renewcommand{\d}[2]{\frac{d#1}{d#2}}

\newcommand{\p}[2]{\frac{\partial#1}{\partial#2}}

\newcommand{\nc}{\newcommand}
\nc {\Prod}{(\underset{I}\phi\, r_I^{n_I})}
\nc {\wlw}{\wedge\ldots\wedge}
\nc {\olo}{\otimes\ldots\otimes}
\nc{\bea}{\begin{eqnarray*}}
\nc{\eea}{\end{eqnarray*}}
\nc {\e}{\varepsilon}
\nc{\de}{\delta}
\nc{\A}{\hat a}
\nc{\Ad}{\hat a^\dag}
\nc{\grad}{\nabla}
\nc{\nlim}{\underset{n\to\infty}\lim}
\nc{\ilim}{\underset{i\to\infty}\lim}
\nc{\isum}{\sum\limits_{i=1}^N}
\nc{\xx}{\underset{x\to x_0}\lim}
\nc{\Bnrm}{\Big |\Big|}
\nc{\sym}{\text{Sym}}
\nc{\inte}{\overset{\circ}}
\nc{\del}{\partial}
\nc{\plp}{+\ldots+}
\nc{\lre}{\longrightarrow}
\nc{\be}{\begin{equation}}
\nc{\ee}{\end{equation}}
\nc{\CP}{\mathbb{CP}}
\nc{\End}{\text{End}}
\nc{\mf}{\mathfrak}
\nc{\Hom}{\text{Hom}}
\nc{\spec}{\text{Spec}}
\nc{\sub}{\subseteq}
\nc{\weakto}{\rightharpoonup}
\nc{\ph}{\varphi}
\nc{\leqc}{\lesssim}
\nc{\delbar}{\overline \del}
\nc{\Tr}{\text{Tr}}
\nc{\Lhe}{\mathcal L_{(\phii^{h_\e}, A^{h_\e})}}

\newcommand{\CC}{\mathbb C}

\newcommand{\RR}{\mathbb R}

\newcommand{\calB}{{\mathcal B}} 
\newcommand{\calC}{{\mathcal C}}

\newcommand{\calO}{{\mathcal O}}

\newcommand{\calU}{{\mathcal U}}
\newcommand{\calV}{{\mathcal V}}

\newcommand{\wtX}{{X}^\circ}

\newcommand{\high}{\mathrm{high}}

\usepackage{calligra}
\DeclareMathAlphabet{\mathcurs}{OT1}{pzc}{m}{it}
\newcommand{\curs}[1]{\mathcurs{#1}}

\newcommand{\FrDmmu}{\mathbb X^{\mu,\infty}}
\newcommand{\FrRgmu}{\mathbb Y^{\mu,\infty}}
\newcommand{\FrDms}{\mathbb X^{\mu,s}}
\newcommand{\FrRgs}{\mathbb Y^{\mu,s}}
\newcommand{\Fdiffeos}{\mathbb F}
\newcommand{\Gdiffeos}{\mathbb G}

\newcommand{\Dir}{D}

\newcommand{\GrT} {\mathcal G}
\newcommand{\GrTparam}{\mathcal F}

\newcommand{\expmultnegEll}{\underline e^{-\sigma_r D_t}}
\newcommand{\expmultposEll}{\underline e^{\sigma_r D_t}}

\newcommand{\expmultneg}{\underline e^{-\sigma_r D_t}}
\newcommand{\expmultnegtau}{\underline e^{-\sigma_{r} D_\tau}}
\newcommand{\expmultpos}{\underline e^{\sigma_r D_t}}

%% file: sections/sec_introduction.tex
%%
%% sec_introduction.tex
%%

\section{Introduction}

%{\color{blue} PREVIOUS DRAFT
%\cite{CHS} Caffarelli-Hardt-Simon
%\cite{smale1991equivariant} N. Smale Equivariant 
%\cite{smale1989minimal} N. Smale many isolated singularities
%\cite{szekelyhidi2021minimal} Szekelyhidi Existence point singularity w/ cylindrical
%\cite{szekelyhidi2026uniqueness,collinsli2023uniqueness,firester2026uniqueness} Some uniqueness results
%\cite{szekelyhidi2023recent} Szekelyhidi Survey
%\cite{simon2023stable} Simon Fractal like singular sets
%\cite{Allard1972, Federer1970,Almgren2000,DeLellisMinterSkorobogatova2023III,KrummelWickramasekera} Various Regularity Results
%\cite{simon1983asymptotics} Simon 1983}

The extensive development of geometric measure theory (GMT) in the 1960s and 1970s led to a remarkable general existence theory for many 
variational problems, resulting in particular in a robust existence theory for volume minimizing submanifolds in Euclidean space and other 
Riemannian manifolds \cite{simon1984lectures,DeLellisSurvey}. The solutions obtained in this way are {\it a priori} unconstrained in their topology and regularity. These methods 
therefore decouple the existence problem from questions about regularity of solutions.  It was already understood at that time 
that in certain settings, one should not expect volume minimizers to be smooth at all points. For example, complex analytic subvarieties 
of K\"ahler manifolds are always volume minimizing, yet may have very intricate singular sets. The task of understanding the precise
nature of singular sets has been a key focus in GMT and geometric analysis in the decades since.

 Even under the strongest reasonable hypotheses (that the submanifold has codimension 1 and is an absolute minimizer of the volume) 
singular sets are unavoidable. There was originally a hope that in this more constrained situation, absolute minimizers would necessarily be smooth, and indeed this is the case in low-dimensional settings. The techniques developed by Almgren, Federer \cite{federer1960normal,federer2014geometric,almgren2000almgren} and others linked the regularity theory of general minimizers to that of their tangent 
cones, which reduces the question to whether there exist nontrivial volume-minimizing cones. The fact that these do not exist
in dimensions less than $7$ thus leads to thus leads to strong regularity results of more general minimizing solutions in these
low dimensions. However, the Simons cone, i.e. the cone over $S^3\times S^3$ in $\RR^8$, was famously shown to be minimizing \cite{SimonsCone,bombieri1969minimal}. This  means that singular sets of codimension at least $7$ are in general unavoidable for absolutely minimizing hypersurfaces. The regularity for (generalized) submanifolds of arbitrary codimension which are only stationary points of the volume functional should be even worse. There are now many other known examples 
of singular minimizers, even in the presence of extra structure such as a calibration \cite{liuconjecture}.

\begin{comment}
Nonetheless, there was reasonable hope that at least in certain situations, minimizers should be 
smooth, and indeed this is true in a number of low-dimensional settings.
% as far back as the solution of the classical Plateau problem by Douglas and Rado in the 1930's.  
The techniques developed by Almgren, Federer and others linked the regularity theory of general minimizers to that of their tangent 
cones, so the key question is whether there exist nontrivial volume-minimizing cones. The fact that these do not exist
in dimensions less than $7$ thus leads to thus leads to strong regularity results of more general minimizing solutions in these
low dimensions. However, Simons discovered famously that the cone over $S^3\times S^3$ in $\RR^8$ is minimizing. 
This shows that singularities in the minimizers of the area functional are unavoidable; regularity for more general (generalized)
submanifolds which are only stationary points of the volume functional should be even worse. There are now many other known examples 
of singular minimizers, even in the presence of extra structure such as a calibration.
\end{comment}

In all these settings, the study of the singular structure of minimal submanifold proceeds from multiple perspectives.  On the one hand, one can take the
``bottom up'' approach, which starts from weak GMT solutions (which are minimizing or stable) and then attempts to constrain the
size and structure of their singular sets, for example limiting their Hausdorff dimension, showing that they are (mostly) rectifiable, etc.
We mention
%This approach is primarily the domain of geometric measure theory:  singularity models are classified by their tangent cones, and much recent work prove strong structural constraints, including various strong rectifiability result for the singular setsD
\cite{Allard1972, Federer1970,Almgren2000,DeLellisMinterSkorobogatova2023III,KrummelWickramasekera} 
amongst many other contributions from this perspective.
On the other hand, one may also approach this problem in a ``top down'' way by explicit construction of examples,
typically by PDE methods, with the  goal of showing that these regularity thresholds are sharp and studying the 
stability of singular sets under perturbations of geometric data.
%well-behaved examples and studying their stability and deformations. This latter approach is naturally the domain of singular elliptic PDE methods.

As already mentioned, there are ``easy'' examples
produced via complex analysis where results from that area show that the singular sets of these complex varieties
are stratified, i.e., decompose into finite unions of smooth (in fact, analytic) submanifolds of varying dimension.   The real difficulty thus
resides in cases where one does not have such a rigid external structure.  The first systematic study of singularities of minimal submanifolds 
from a PDE perspective was in the work of Caffarelli-Hardt-Simon \cite{CHS} in the early 1980s.  They developed a perturbation theory
for the mean curvature operator on a minimal cone to construct large families of minimal (and minimizing) submanifolds with boundary 
 admitting isolated conical singularities in the interior. Subsequent work of Hardt-Simon \cite{HS} and Simon \cite{simon1983asymptotics} 
 culminated in Simon's celebrated theorem about uniqueness of tangent cones at isolated singularities, under certain conditions. 

There are obvious examples of minimal or minimizing submanifolds with non-isolated singularities produced by taking the product 
of any one of these solutions with isolated singularities with  a flat Euclidean space or torus.  Somewhat more general
are the equivariant constructions by N. Smale \cite{smale1991equivariant}.  However, it has been a long-standing problem, already raised  
by Leon Simon in the 1980s, to carry out a similar perturbation theory as in \cite{CHS} to construct families of non-product 
minimal submanifolds with boundary, with non-isolated interior singularities, and more specifically where the tangent 
cone is `cylindrical', i.e., the product of a Euclidean space with a minimal cone with isolated singularities. This seemingly 
modest task has thus far resisted efforts. We settle this problem here by constructing a broad class of non-product minimal submanifolds with this property.

\smallskip 

In addition to the existence question, there is the question of the regularity of these singular sets themselves as subsets of the ambient space. 
There has been progress showing that singularities of minimizers can be quite bad in certain cases. In  \cite{szekelyhidi2021minimal},
Sz\'ekelyhidi constructed a minimizing submanifold with isolated singularity but a cylindrical tangent cone at that point. More recently, 
\cite{simon2023stable} showed that, provided the ambient smooth metric is allowed to vary, there exist minimal submanifolds of $\RR^m \times \RR^n$ for which the singular 
set is an arbitrary closed set of $\{0\} \times \RR^n$. Finally, Halavati and Spolaor \cite{HaSpo} recently constructed minimal hypersurfaces with a singularity 
at the origin with cylindrical tangent cone that is an accumulation point of a sequence of isolated singularities. Here too, the background metric 
must be altered. These examples, and a number of other related ones in recent years, leave open the basic question about 
the regularity of the singular set when the ambient metric is Euclidean, or at least real analytic.

We actually prove a stronger regularity result, that the singular sets of the submanifolds constructed 
here are themselves analytic submanifolds of the ambient Euclidean space.  The proof of this fact is 
similar to the proof of existence in the sense that we carry out the same main steps in the construction, 
but in function spaces which capture analytic regularity.   

%In a forthcoming companion to this paper \cite{MP2} we show that for these examples, the singular sets are actually real analytic submanifolds.   The proof of that result is very similar in spirit to the one here, but involves working in more delicate function spaces.   

This and the results above suggest that there may be a regularity ``threshold'' phenomenon for the regularity of the singular set. That is to say, 
while there exist examples whose singular sets are more or less as rough as the GMT allows, there is a finite regularity such that the following holds: 
if the singular set of a minimizer is a submanifold with regularity greater 
than that threshold, and if the ambient metric is analytic, then that singular set is automatically analytic as well. This phenomenon would mirror known regularity results in free-boundary problems, and by way of analogy, these singular sets should most fruitfully be regarded as higher codimension analogues of free boundaries. {During the writing of this manuscript, the authors learned that forthcoming work of Halavati and Spolaor makes substantial progress in this direction \cite{HaSpForthcoming}}.

\subsection{Main Results}  Fixing integers $3\leq n-1 \leq N-2,$ let $(\Sigma^{n-1}, g_\Sigma)\subseteq (S^{N-1}, g_{0})$ be a smooth, 
closed minimal submanifold in the round unit sphere of codimension $N-n$. The truncated cone 
\[
C_1(\Sigma):=\Big \{rx \ | \ x\in \Sigma  \ , \ 0\leq r \leq 1\Big\}\subseteq B^{N}
\]
is also minimal; it has boundary $\Sigma$ at $r=1$ in the unit sphere and an isolated conical singularity at the origin. 

The specific deformation problem studied by \cite{CHS} may be stated as follows. We wish to solve the minimal surface equation 
for sections of the normal bundle over this entire truncated cone, where the boundary value of this section is an arbitrary section 
of the normal bundle of $\Sigma$ in the sphere.  We seek minimal deformations, if they can be found, which retain an isolated conical singularity at
the origin. They proved that this problem has a finite-dimensional obstruction defined in terms of the spectral data 
of the Jacobi operator of $\Sigma$ in the sphere.  This obstruction also depends on the choice of a positive real parameter 
$\mu' > 1 + \epsilon$ which dictates the rate of decay of the perturbation to the exact cone. (The lower bound being greater than
$1$ ensures that the tangent cone of the perturbation at the origin is the original cone.)  
% an order of vanishing $\mu \in \R^+$ of the perturbation
%to the exact cone). 
%This weight is chosen to dictate the asymptotics of a perturbation to the cone as $r\to 0,$ 
%and the lower bound $\mu+1>\tfrac{n}{2}$ ensures the perturbation is $O(r^{1+\epsilon}),$ in which case it preserves both the singularity and the tangent cone model at the origin. 
%This vanishing condition, however, is too stringent to allow an arbitrary perturbation of the boundary data to be corrected back to minimal with an interior perturbation, and this is the source of the obstruction. It is defined by the spectral data on $\Sigma$ as follows: 
To state this more carefully, denote by $L_\Sigma$ this Jacobi operator. This is self-adjoint, with eigendata $\{(\phi_j, \lambda_j)\}_{j=0}^\infty$.
Given any $\mu \in \R$, define
\be  
J(\mu) = \left \{ \ \lambda_j \in \R  \ \ \Big | \ \  \sqrt{\frac{(n-2)^2}{4}+\lambda_j}  \ <  \ \mu \ + \ 1 \ \  \right\} 
\qquad \mathcal B_0({\mu}):= \bigoplus_{\lambda_j \in J(\mu)} E_{\lambda_j},
\label{weightconstraint}
\ee
where $E_{\lambda_j}$ is the eigenspace corresponding to the eigenvalue $\lambda_j$. (The value $\mu$ appearing here satisfies 
$\mu' + n/2 = \mu$, where $\mu'$ is the vanishing rate above. This shift is simply because $\mu$ is the $L^2$ rather than
the $L^\infty$ vanishing order.)
We call $J(\mu)$ the 
set of {\it low lying eigenvalues}; these are counted with multiplicity. It is shown in \cite{CHS} that that when $\mu+1 > \frac{n}{2}$, 
then the obstruction to this deformation problem with arbitrary boundary values corresponds precisely to the space $\mathcal B_0(\mu)$,
and thus in particular has finite dimension $J:=\dim J(\mu)$. We explain this carefully in Section \ref{CHSSection} below.

{ The space $\mathcal B_0(\mu)$ contains two distinguished eigenspaces with particularly transparent geometric meaning. These correspond to `trivial' deformations of the cone obtained by ambient Euclidean isometries,}
namely translations (which move the singular set) and rotations (for which the corresponding normal section vanishes
exactly linearly, not like $r^{1+\epsilon}$). These correspond to the eigenvalues $0$ and $(n+1)$ respectively. In other words,  there is a map
\be
\R^N \oplus \mathfrak{so}(N) \to \mathcal B_0(\mu) 
\label{isometryrestriction}
\ee 
\noindent which in fact is injective on the first factor.

The deformation result in the cylindrical setting below is significantly more rigid in many respects. In particular, we need to 
impose the restriction that the Jacobi fields for $\Sigma$ corresponding to Euclidean isometries in the conic (i.e. $\R^N$) factor
constitute the full set of eigenvectors with low-lying eigenvalues.
\begin{comment}There are two types of `trivial' deformations which do not quite fit into the scheme above, namely those obtained by ambient Euclidean 
isometries, namely translations (which move the singular set) and rotations (for which the corresponding normal section vanishes
exactly linearly, not like $r^{1+\epsilon}$). The infinitesimal variations of boundary values associated to translations 
correspond to the eigenvalues $\lambda_j$, $j = 0, \ldots, N-1$.  Assuming that $\Sigma$ is in general position (i.e., does
not lie in any equatorial subsphere), the eigenvalues associated to infinitesimal rotations are indexed by
$j = N, \ldots N + K-1$, where $K = N(N-1)/2$.

In other words, there is an injective map
\be
\R^N \oplus \frak{so}(N) \to \mathcal B_0(\mu) 
\label{isometryrestriction}
\ee
\end{comment}
\begin{defn} 
We say that a smooth minimal submanifold $\Sigma^{n-1} \subseteq S^{N-1}$ is {\bf spectrally rigid} if there exists a 
$\mu+1>\frac{n}{2}$ such that the restriction map \eqref{isometryrestriction} is surjective. In this case, 
\[
J(\mu)=\Big\{ 0 \ , \  n-1 \Big\},
\]
{with multiplicities $\dim E_0=N,$ and $\dim E_{n-1}:= K \leq \frac{N(N-1)}{2}$ respectively. Thus in this case, $J=N+ K.$}
\label{spectralassumptions}
\end{defn}

\noindent There are examples of spectral rigid minimal submanifolds. Recent work of Rauchenberger-Semmelmann 
\cite{rauchenberger2026index} shows that the Veronese embeddings 
\be 
\C \mathbb P^2 \subseteq S^7 \hspace{1cm}\mathbb H \mathbb P^2 \subseteq S^{13}\hspace{1cm}\mathbb O 
\mathbb P^2 \subseteq S^{25} \label{embeddingsKPN}
\ee
are all spectrally rigid. These have $K=20, 70, 273$ in the three cases respectively.  On the other hand, 
if $\Sigma$ is a hypersurface in the sphere, then the lowest eigenvalue corresponds to a strictly positive eigenfunction,
and by the work of Hardt-Simon, under suitable conditions this corresponds to a family of {\it smooth} minimal submanifolds
which lie in one component of $\R^N \setminus C_1(\Sigma)$.  Thus hypersurfaces are never spectrally rigid, and
our construction below does not apply in that case (cf. Remark \ref{deformationremark} below).    

\subsubsection{Submanifolds with Cylindrical Tangent Cones} To generalize \cite{CHS} to the setting of submanifolds with 
cylindrical or {\bf edge} singularities, consider the product 
\be 
X_0 := S^1 \times C_1(\Sigma) \ \subseteq S^1\times B^N.
\ee
In fact, we might equally consider the product of $C_1(\Sigma)$ with $\R$, but taking the product with the circle is the
same as considering solutions periodic in the $\R$ direction; this avoids a number of mostly superfluous  technical complications.
In any case, this product is again minimal, now with an `edge', which we think of as a family of conical singularities along the 
smooth submanifold $S^1\times \{0\}$.  %In directions normal to this singular set, these singularities n the cylindrical tangent cone
%$\R \times C(\Sigma)$; n fact $X_0$ is equal to the truncation of this cylindrical tangent cone modulo the action of translations $t\mapsto t+2k\pi$ on $\R$. 
We write $\text{reg}(X_0)$ and $\text{sing}(X_0)$ for the smooth and singular loci, respectively. 

The perturbations of this product have an additional degree of freedom, as compared to the conic case. 
%Unlike in the conical case, the singular sets of minimal submanifolds with edge singularities are no longer necessarily 
%equivalent by rigid isometries. 
Indeed, one expects (or at least hopes to construct) minimal submanifolds where the singular set is a graph
\be 
\Gamma_\xi:= \Big \{ (t,\xi(t)) \in S^1 \times \R^N \ \  \Big | \ \ \xi: S^1 \to \R^N \Big \}\label{Gammaxi}
\ee
for $\xi: S^1 \to \R^N$ some sufficiently small function. 
%\noindent of a sufficiently small function $\xi: S^1\to \R^N$. These graphs are distinguished by the extrinsic curvature of the singular set, provided they are sufficiently smooth. 
Our main result 
%is the following generalization of the main theorem of Caffarrelli-Hardt-Simon \cite{CHS} to the cylindrical setting. It 
shows that if $\Sigma$ is spectrally rigid, then $X_0$ admits minimal perturbations where the singular set is deformed to a 
nearby smooth curve (which depends on the perturbation of boundary data).

\begin{thm} \label{maina}
 Suppose that $\Sigma\subseteq S^{N-1}$ is a spectrally rigid minimal submanifold of $\text{dim}(\Sigma)\geq 2,$ and let $\mu+1 > \frac{n}{2}$ be the corresponding weight.  There exists an open neighborhood of the origin, $\mathcal V_1 \subseteq  \mathcal B_0(\mu)^\perp\subseteq C^\infty(S^1\times \Sigma; N\Sigma)$ such that for each choice of boundary values 
\begin{align} \varrho \in \mathcal V_1
\end{align}
\noindent there exists a submanifold with edge singularities $X_{\xi,\varrho}\subseteq S^1 \times \R^N$ 
with the following properties: 
\begin{enumerate}
\item[(I)] $X_{\xi,\varrho}$ is minimal with respect to the Euclidean product metric.
\medskip 
\item[(II)] The regular locus $\text{reg}\, (X_{\xi,\varrho})$ is diffeomorphic to $X_0\setminus \text{sing}(X_0).$
\medskip 
\item[(III)] The singular set $\text{sing}\, (X_{\xi,\varrho})=\Gamma_\xi$ is the graph of a smooth section $\xi \in C^\infty(S^1;\R^N)$, 
and there is a smooth map $\Xi: \mathcal V_1\to C^\infty(S^1;\R^N)$ that defines $\xi=\Xi(\varrho)$ implicitly as a function of $\varrho.$
\medskip 
\item[(IV)] After an appropriate identification of the boundaries $\del X_{\xi,\varrho}\simeq \del X_0,$ the `component' of $\del X_{\varrho, \xi}$ determined by $\mathcal B_0(\mu)^\perp$ equals the graph of $\varrho$ over 
$S^1\times \Sigma$, i.e.,
\smallskip
\[
 \text{exp}^{-1}_{S^1\times N\Sigma}\left(\del X_{\xi,\varrho}\right)= \varrho\  \mod \mathcal B_0(\mu) 
\]
\noindent The component of $\del X_{\xi,\varrho}$ in the span of the constant (over $S^1$) eigensections in $\mathcal B_0(\mu)$ is determined implicitly as a smooth function of $\varrho \in \mathcal V_1.$ 
\smallskip 
\end{enumerate}
\end{thm}
\noindent We emphasize that all the singular sets arising in Theorem \ref{maina} are {\it smooth}. 

As explained earlier, for isolated conical singularities we may as well keep the singular point fixed since a minimal submanifold with an 
isolated singularity can be translated to keep the singularity at the origin. Likewise, if the tangent cone is a rotated copy of the same
cone, then we can simply rotate back to ensure that the tangent cones remain the same. In the cylindrical setting, however, the space 
of possible singular sets is infinite-dimensional, and the submanifolds corresponding to different functions $\xi$ are not isometrically 
equivalent except in very rare circumstances.

There is a much more significant difference, which is of an analytic nature.   In the isolated singularity case, \cite{CHS} apply the 
standard Implicit Function theorem in Banach spaces, which reduces the problem to the linear analysis of the Jacobi operator.
%the new features there were primarily linear. 
We review (and provide
a much simpler recasting of) their result in Section \ref{CHSSection}.  In that setting, the obstruction to the linearized 
problem is finite-dimensional and the point is that by allowing some freedom in the boundary values, it can be canceled.
In this setting, however, both the obstruction and the degrees of freedom in the singular set are infinite-dimensional, and 
variations in the singular set are the new feature used to cancel the obstruction.   Unfortunately the ``cancellation operator'' 
here fails to have closed range between any fixed pair of Banach spaces, and this forces us to work in the category of tame 
Fr\'echet spaces and to use the Nash-Moser Implicit Function Theorem. 
%This need for Nash-Moser theory and more careful analysis of the cancellation operator 
This is what accounts for the relative complexity of the proof of Theorem \ref{maina} as compared to \cite{CHS}. 

To be a bit more specific, the main content of the proof involves showing that the deformations of order zero (translations) and 
order one (rotations) along the singular set provide sufficiently many extra parameters to cancel the now infinite-dimensional 
obstruction. These deformations of the tangent cone along the singular set yield degrees of freedom parameterized by 
$L^2(S^1; \R^{N}\oplus \mathfrak{so}(N)).$ Meanwhile, the obstruction to deforming the cone with fixed singular set is parameterized by 
$L^2(S^1; \R^{J(\mu)})$. The spectral rigidity assumption ensures \eqref{isometryrestriction} is surjective, and this at least suggests
that these parametrized normal isometries are sufficient to cancel the full obstruction. Indeed, we show that when the spectral 
rigidity condition holds, the cancellation is an elliptic pseudodifferential operator with symbol given by multiples of 
\eqref{isometryrestriction} modulo its kernel. This suggests that the spectral rigidity condition is not only sufficient but also 
necessary, for a result like the one here to be valid. In particular, in the absence of spectral rigidity, it seems likely that perturbations preserving the edge singularity should be extremely `rare' (i.e. infinite-codimension), with most perturbations destroying the manifold structure of the singular set (see Remark \ref{deformationremark} below).
% as sufficient, for there to be a finite-codimension space of boundary values for which the corresponding perturbations preserve a singular set diffeomorphic to $S^1.$

\smallskip 

\begin{rem} The submanifolds $X_{\xi,\varrho}$ in Theorem \ref{maina} have boundary along the set $\{\text{exp}_{N\Gamma_\xi}(y) \ | \ |y|=1\},$ i.e. at unit length in geodesic normal coordinates around the singular sets $\Gamma_{\xi}.$ Once could instead ask for submanifolds whose boundary lies instead along $r=1$ in global polar coordinates on $\R^N.$ A small variation of the proof allows the construction of minimal submanifolds in this case as well, at the cost of augmenting $\mathcal B_0(\mu)$ in (IV) with an additional finite-dimensional subspace. 
\end{rem}

\smallskip 

\begin{rem} \label{S1used} The fact that we consider the product with $S^1$,  rather than a higher dimensional torus, 
in Theorem \ref{maina} is largely for convenience. 
%in particular we use the stability of a well-enumerated spectrum under perturbations of the metric. 
%This assumption can likely be removed with fairly standard techniques from spectral theory and harmonic analysis, 
We expect that this extension should be possible in a fairly straightforward way, but has been omitted here to
reduce the technicalities. %which would yield the analogue of Theorem \ref{maina} where $\text{sing}(X_{\xi,\varrho})\simeq T^k.$
\end{rem}

\smallskip 

\begin{rem}
As explained above, the proof of Theorem \ref{maina} suggests that the spectral rigidity condition is necessary in order to prescribe 
the boundary values within a finite codimensional subspace.  For more general minimal submanifolds $\Sigma,$ it seems reasonable 
to expect that the space of perturbations of $S^1\times \Sigma$ which bound a minimal submanifold with smooth singular set has
{\it infinite codimension}, and is naturally complemented by $L^2(S^1; \R^M)$ where $\R^{M}$ is spanned by those eigenspaces 
in $\mathcal B_0(\mu) / \text{Im}(\R^{N+K})$ via \eqref{isometryrestriction}. It would be interesting to extend our results to this 
more general setting (e.g. by including additional parameters corresponding to the Hardt-Simon de-singularizations discussed after \eqref{embeddingsKPN}).
% by incorporating ($S^1$ families of) an additional $M$-dimensional desingularization parameter for the cone $C(\Sigma).$ 
We remark that analogous spectral constraints were found by Y. Wang \cite{wang2022atiyah} in the case of deforming singular 
Yang-Mills connections. 
\label{deformationremark}
\end{rem}

\medskip 

It is of course highly interesting to know if there are any global constraints or characterizations of those functions $\xi$ whose graphs actually arise as the perturbed singular sets of near-product minimal submanifolds. 
%It is an interesting question to ask for a more precise characterization of {\it which} functions $\xi$ arise in the collection of singular sets of the submanifolds $X_{\xi,\varrho}$ constructed 
Theorem \ref{maina} shows that these functions (and singular sets) are necessarily smooth. In fact, a stronger statement is true: the singular sets $\text{sing}(X_{\xi,\varrho})$ are actually {\it analytic} submanifolds (of the ambient space). 
%a much stronger result on the regularity of 
% is true, whose proof will appear in a follow-up to this article \cite{AnalyticBootstrap}.  
To state this precisely, recall that by the classical Paley-Weiner Theorem, a function $\xi: S^1\to \R^N$ 
admits a holomorphic extension to the strip $S^1\times [-\sigma, \sigma]\subseteq T^*S^1$ if and only if its Fourier 
modes $\xi_\ell$ decay exponentially like $e^{-\sigma |\ell|}.$ Let 
\[ 
e^{-\sigma\sqrt{\Delta_t}} L^2(S^1;\R^N)\subsetneq C^\omega(S^1;\R^N)
\]
denote the Banach space of analytic curves admitting such a holormophic extension for some given $\sigma>0.$ The main 
result of the second half this article is the following ``analytic bootstrap'' for the singular set. 
\begin{thm} \label{mainb} For the submanifolds $X_{\xi,\varrho}$ constructed in Theorem \ref{maina}, the singular sets $\Gamma_\xi \subseteq S^1\times \R^N$ are analytic submanifolds. More specifically, for every $\e>0$ and $\sigma=1-\e,$ there is an open ball $\mathcal V_1(\e)\subseteq \mathcal V_1$ such that for $\varrho \in \mathcal V_1(\epsilon),$
\[\xi(\varrho) \in e^{-\sigma \sqrt{\Delta_t}}L^2(S^1;\R^N).\]
\noindent There is a $p>2$ such that the radius of $\mathcal V_1(\epsilon)$ measured in the $H^{s_0}$-norm is $O(\epsilon^p).$
\end{thm}
\noindent More generally, if the cone $C(\Sigma)$ is truncated at outer radius $R_0,$ then the singular sets must lie in the 
space of functions admitting holomorphic extensions to the strip of width $\sigma=R_0-\epsilon.$

The relationship between the singular set $\Gamma_{\Xi(\varrho)}$ and the boundary value $\varrho$ remains unclear. We had at first 
believed that these could be prescribed in some sense independently of one another, but our attempts to prove this led to 
very delicate analytic problem with borderline behavior.  We expect that it should be possible to say much more about the 
relationship between $\varrho$ and $\Xi(\varrho)$:

%There is of course, the possibility that $\Gamma_{\xi(\varrho)}=S^1 \times \{0\}$ is simply the product singular set for {\it every} choice of $\varrho,$ in which case Theorem \ref{mainb} would be trivial. This reveals a fundamental question about the behavior of minimal submanifolds with such singularities. The two pieces of data $(\xi,\varrho)$ determine the behavior of the submanifold as $r\to 0$ and $r\to 1$ respectively, but it remains unclear if these are independent data, or whether prescribing one of them necessarily implicitly prescribes the other. The authors at first believed these two pieces of data could effectively be prescribed independently, but the analogous setup to Theorem \ref{maina} in that case leads to a fundamentally borderline analytic problem. Instead, we now believe that there should be a precise characterization of the sense in which $\varrho$ determines $\xi,$ and vice-versa, with simplest possible characterization being that the decoupled situation (i.e. $\Gamma_\xi$ is always a product) does not occur. 

%Let $X_{\xi(\varrho),\varrho}$ denote the submanifolds in the conclusion of Theorem \ref{maina}, and $\Gamma_{\xi(\varrho)}$ their singular sets.  
\begin{conj} The implicitly defined map 
\be
\varrho \mapsto \Xi(\varrho) \in e^{-\sigma \sqrt{\Delta_t}}L^2(S^1;\R)
\label{locallysurjective}
\ee
is locally surjective at $\xi \equiv 0$ for every $\sigma>1.$ 
\label{conj1}
\end{conj}
This would preclude the potential `trivial' case, where $\Xi(\varrho) \equiv 0$ for every $\varrho$ (corresponding to the
singular set remaining as the round $S^1$). In fact, it may even be the case that the set of boundary perturbations $\Xi^{-1}(0)$ 
which keeps the singular curve fixed has the following more precise characterization. With $\mu$ as above, let 
$L^2(\Sigma; N\Sigma)=\mathcal B_0 \oplus \mathcal B_0^\perp$ denote the $L^2-$orthogonal splitting into the $J(\mu)$ 
low-lying eigenmodes and their complement. We may write 
\[
\varrho=(\varrho^\text{low}, \varrho^\text{high}) \in C^\infty_0(S^1;\mathcal B_0) \oplus C^\infty(S^1;\mathcal B_0^\perp),
\]
where in the first factor, $C^\infty_0$ denotes the $L^2$-orthogonal complement of the constant Fourier mode in the $S^1$-direction. 
We expect that, in addition to \eqref{locallysurjective} being locally surjective, 
\[
\xi^{-1}(0)=\Big \{  \ \big (\zeta(\varrho^\text{high})  \ , \ \varrho^{\text{high}}\big) \ \big | \ \varrho^\text{high}\in 
C^\infty(S^1;\mathcal B_0) \Big \},
\]
i.e., the condition corresponding to the product singular set remaining fixed defines the low eigenmode content of the 
boundary values implicitly as a function of the high modes.

Observe here that $\varrho^\text{low}\in C^\infty_0(S^1;\mathcal B_0)$ is, up to a regularity shift, naturally identified with 
the space of possible singular sets up to ambient isometries. If the derivative of \eqref{locallysurjective} were given by this 
identification, up to appropriate isomorphisms, this characterization of the zero set $\Xi^{-1}(0)$ is what would arise 
naturally from the Implicit Function Theorem. 
%It remains far from clear though, whether such an approach is feasible. 
Notice also that the conjecture asserts surjectivity for regularities $\sigma>1$ above the borderline. The question of what 
occurs {\it at} the borderline regularity $\sigma=1$ may be far more difficult. 

\subsection{Outline} The proof of Theorem \ref{maina} occupies the next eight sections of the article, with the final four sections devoted to the proof of Theorem \ref{mainb}. 

These first eight section proceed 
roughly as follows. In Section \ref{CHSSection}, we review the result of Caffarelli-Hardt-Simon, offering a succinct proof 
to presage the later proof of Theorem \ref{maina}. In Section \ref{PreliminariesSection} we begin the analysis of the cylindrical 
setting in earnest, starting with the study of the linearized Jacobi operator. Here, the Jacobi operator is a degenerate 
{\it elliptic edge operator}, and we first review some small part of the theory of this class of operators that suffices 
for our purposes. One important conclusion of this section is that for the relevant weights $\mu,$ the Jacobi operator is 
injective with an infinite-dimensional cokernel, i.e. an obstruction. Section \ref{section4} gives a precise characterization 
of this infinite-dimensional obstruction, which yields a natural identification with the space of $L^2$-sections of a 
vector bundle on the edge $S^1.$ These first four sections complete the linear analysis of Jacobi operator. 

The following five sections set up the Nash-Moser Implicit Function Theorem used to prove Theorem \ref{maina}. Section \ref{FrechetSection} gives a brief review of the tame Fr\'echet category and the relevant version of the Implicit Function Theorem. Section \ref{deformationssection} then carefully sets up a space of deformations of the product $X_0$ consisting of submanifolds with edge singularities along nearby curves, and Section \ref{deformationoperatorsection} shows the freedom to deform the singular set results in a perfect pairing with the infinite-dimensional obstruction. Section \ref{obstructionperturbations} extends the discussion of the obstruction to define an obstruction bundle for the family of linearized operators over the space of submanifolds with edge singularities, which is necessary due to the stronger hypotheses of Nash-Moser compared to the standard Implicit Function Theorem. Finally, Section \ref{tameestimate} verifies the hypotheses of Nash-Moser to conclude the proof of Theorem \ref{maina}. 

Sections \ref{analyticspacessection}--\ref{analytictame} then prove Theorem \ref{mainb}. The proof of this stronger regularity result is accomplished by repeating, in essence, the exact same argument as Theorem \ref{maina}, only in more delicate function spaces imposing analytic regularity. Section \ref{analyticspacessection} introduces these function spaces, and Section \ref{LinearanalyticSection} extends the main aspects of the elliptic edge theory to this setting. The main new ingredient in this case, appearing in Section \ref{arclengthgaugesection}, is the construction of an `arclength gauge' which implicitly defines a Fr\'echet manifold of (not necessarily minimal) candidate submanifolds with sufficient analytic regularity as measured in their own arclength coordinates. Finally, Section \ref{analytictame} concludes the proof of Theorem \ref{mainb}. 

\subsection*{Acknowledgements}
The first author had worked on this problem many years ago with Nathan Smale, and wishes to acknowledge their many fruitful conversations at that time.
The authors also wish to thank Otis Chodosh, Davide Parise, and Yichen Zhang for helpful conversations, and Zhihan Wang for pointing out the work of Rauchenberger-Semmelmann \cite{rauchenberger2026index}. G.P. is partially supported by an NSF Mathematical Sciences Postdoctoral Research 
Fellowship (Award No. 2303102).

\subsection*{Tool Usage} Large Language Models (LLMs) were used in the development of this article to aid in surveying existing results in the literature and to suggest streamlined approaches to a few technical lemmas during the editing process. No portion of this article was written directly by LLMs, nor is any portion of its core mathematical content attributed thereto.

%% file: sections/sec_chs_deformations.tex
%%
%% sec_chs_deformations.tex
%%

\section{CHS Deformations of Isolated Singularities}
\label{CHSSection}
Before embarking on our construction of minimal submanifolds with edge singularities and cylindrical tangent cones, we review the analogous
problem studied originally by Caffarelli-Hardt-Simon \cite{CHS}, which finds minimal deformations of a minimal cone with isolated singularity. 
A loose statement of their main result appears already in the introduction, but we briefly present in this section a particular formulation of 
the proof as a model and guide for the much more difficult arguments in the rest of this paper.  We note that yet another formulation of
the proof using the modern theory of conic operators was given in the recent paper \cite{FMM26}, but the one described here charts
a strategy adapted for handling the more challengng result in the edge setting.

Let $\Sigma^{n-1}$ be a minimal submanifold in the round unit sphere $S^{N-1}$, with Jacobi operator
\[
L_\Sigma = \Delta_{N\Sigma} - h h^\star;
\]
the first term on the right is the connection Laplacian on the normal bundle of $\Sigma$ (normalized to have nonnegative spectrum) and
$h$ is the second fundamental form of $\Sigma$ in the sphere, with $hh^\star$ acting by $V \mapsto \sum h^{ij} \langle h_{ij}, V \rangle$. 
The Jacobi operator on $C_1(\Sigma)$ then takes the form
\be 
L_{C(\Sigma)} = - \left(\frac{\del^2\,}{\del r^2} + \frac{(n-1)}{r} \frac{\del\,}{\del r}\right) + \frac{L_\Sigma}{r^2} \label{Lconical}
\ee
where $r$ is the radial variable on the cone.  This is an example of an elliptic conic operator; in slightly different language, $r^2 L_{C(\Sigma)}$
is called an elliptic $b$-operator because it is an elliptic combination of compositions of the basic vector fields $r\del_r$ and
$\del_{x_j}$, where $(x_1,..,x_{n-1})$ constitute a local coordinate system on $\Sigma$. There is a robust theory of differential and pseudodifferential
$b$-operators which includes, in particular, a framework for the construction of parametrices for elliptic $b$-operators.
These parametrices then determine sharp mapping and regularity properties \cite{Melrosebcalculus, grieser2001basics,Mazzeo1991EdgeOperators}. %,melrose1981transformation}.

Despite this well-developed theory, many of the results have elementary proofs in this setting. Below we opt towards these elementary proofs, and only borrow the two following standard definitions from the $b$-operator framework. Because the operator applies combinations of the vector fields $r\del_r, \del_{x_i},$ it is naturally bounded on Sobolev spaces formed with these derivatives. Given a vector bundle $E\to C_1(\Sigma)$ with smooth connection $\nabla$ (taken to be the Levi-Civita connection for bundles of tensors) we denote the corresponding covariant derivative by 

\[ \nabla^b = dr\otimes r\nabla_r \ + \ dx_i \otimes \nabla_{x_i}. \]

\begin{defn} \label{bspaces}

For $\mu \in \R$, $s\in \N$, and a vector bundle $E\to C_1(\Sigma)$ define the ${b}$-{\bf Sobolev spaces}
\smallskip
\begin{align*}
\hspace{.5cm}r^\mu H^s_b(C_1(\Sigma);E) &:= \Big\{u \in L^2 \ \Big | \ \|u\|_{r^\mu H^s_b}<\infty \Big\},
\intertext{where}
\|u\|_{r^\mu H^s_b}&:= \Bigg(\sum_{|\alpha|\leq s}\int_{C_1(\Sigma)} |(\nabla^b)^\alpha u|^2 \ r^{-2\mu} \ r^{n-1}dr dV_\Sigma \Bigg)^{1/2}
\end{align*}
\noindent and $\alpha$ denotes a multi-index ranging over $r,i=1,...,n-1$. For $s=0$, these spaces are denoted by $r^\mu L^2.$ All are Hilbert spaces under the inner-product arising from the norm's polarization. 

\smallskip 
 \label{bSobolev}
\end{defn}

Notice the weights imply that $f \in L^2$ iff $r^\mu f\in r^\mu L^2,$ and likewise for higher regularity. It is immediate that \be L_\Sigma: r^{\mu+2}H^{s+2}_b(C_1(\Sigma), NC_1(\Sigma))\lre r^{\mu}H^s_b(C_1(\Sigma), NC_1(\Sigma))\label{conicalJacobimap}\ee

\noindent is bounded. At $r=1,$ the standard boundary trace map $\text{Tr}(u)=u|_{r=1}$ is bounded 
\[\text{Tr}\, :r^{\mu+2}H^{s}_b(C_1(\Sigma; E)\to H^{s-1/2}(\Sigma;E).\] 
We write 
 \[r^\mu H^s_{b,\Dir}(C_1(\Sigma); E):=\ker(\text{Tr})\]

 \noindent to denote the space of sections with vanishing Dirichlet boundary values.

\medskip

Given the eigendata $(\phi_j,\lambda_j)$ 
for $L_\Sigma$, then to each eigensection $\phi_j$ one associates two solutions to $L_{C(\Sigma)}u=0$ in the corresponding eigenspace, namely
\be
u_j^\pm(r,x) = r^{\gamma_j^\pm} \phi_j(x),\label{conicalsolutions}
\ee
where $\gamma_j^\pm \in \C$. Any suitably convergent linear combination of the $u_j^\pm$ is likewise a solution to $L_{C(\Sigma)} u = 0$, and conversely, any homogeneous solution
of this equation must be a linear combination of these basic functions. Direct application of \eqref{Lconical} shows $\gamma_j^\pm$ obey a quadratic equation that is easily solved. 

\begin{defn} The set of {\bf indicial roots} $\gamma_j^\pm$ is determined by the eigenvalues $\lambda_j$ is
\be 
\text{Ind}(L_{C(\Sigma)})=\left\{ \gamma_{j}^\pm = \frac{2-n}{2} \pm \sqrt{ \frac{(n-2)^2}{4} + \lambda_j}\right\}.
\label{indicialroots}\ee

\noindent We tacitly assume that the indicial roots $\gamma_j^\pm\in \R$ are {\it real}, though the general case requires only cosmetic modifications. 
\end{defn}

\noindent The mapping properties of $L_{C(\Sigma)}$ now reduce essentially to arithmetic with the weights and indicial roots. In particular, one of the basic solutions \eqref{conicalsolutions} lies in a particular weighted space if and only if the function $\smash{r^{\gamma_j^\pm}}$ is integrable there. 

We are primarily interested graphical submanifolds over $C_1(\Sigma)$ defined by normal vector fields which decay like $r^{1+\delta}$ for some $\delta > 0.$ The vanishing condition ensures that these submanifolds all share both the vertex at $\{0\}$ and the accompanying tangent cone model $C(\Sigma)$ there with $C_1(\Sigma)$ itself. Because we work with Sobolev (rather than H\"older) spaces, the volume form shifts the decay by $n/2.$ Thus we assume that the domain weight in \eqref{conicalJacobimap} obeys $\mu+2>n/2 +1+\delta,$ i.e. 

\be \mu+1 > \frac{n}{2} \hspace{2cm} \mu+2- \frac{n}{2}\notin \text{Ind}(L_{C(\Sigma)}). \label{weightconstraintsbcase}\ee

\noindent The second of these conditions excludes a discrete collection of problematic weights. Indeed, a standard argument shows that if the solutions \eqref{conicalsolutions} are exactly borderline for integrability, multiplying by cut-off functions produces a Reisz sequence which prevents \eqref{conicalJacobimap} from having closed range. We call the weights $\mu_j^\pm=\gamma_j^\pm +n/2-2$ violating the second condition in \eqref{weightconstraintsbcase} the {\it indicial weights}. 

For a suitable choice of $\mu$ obeying \eqref{weightconstraintsbcase}, there are a finite number of `$+$' indicial roots $\gamma_j^+$ such that the corresponding solution \eqref{conicalsolutions} fails to be integrable. These correspond to those eigenvalues lying in the set
\be  J(\mu) = \left \{ \ \lambda_j \in \R  \ \ \Big | \ \  \sqrt{\frac{(n-2)^2}{4}+\lambda_j}  \ <  \ \mu \ + \ 1 \ \  \right\} \label{Jmuagain}  \ee

\noindent that appeared in \eqref{weightconstraint}. We denote $J=\#J(\mu)$ counted with multiplicity, once a choice of $\mu$ is specified. Finally, as in the introduction, we let 
\[
L^2(\Sigma; N\Sigma)=\mathcal B_0 \oplus \mathcal B_0^\perp,
\] 

\noindent denote the span of the corresponding low-lying eigenspaces and their complement, and $\pi^\text{low}, \pi^\text{high}$ the corresponding orthogonal projections.

With this preparation, we can now read off the mapping properties of $L_{C(\Sigma)}$ directly from the solvability of ODEs. {The arithmetic in (\refeq{weightconstraint}) is arranged precisely so that integrability constraints on the solutions (\refeq{conicalsolutions}) yield the following: 

\begin{prop} The map 
\[
L_{C(\Sigma)}: r^{\mu+2} H^{s+2}_{b,\Dir}(C_1(\Sigma), N C_1(\Sigma)) \longrightarrow r^{\mu} H^s_{b}(C_1(\Sigma), N C_1(\Sigma)) 
\]
is Fredholm if and only if $\mu$ is non-indicial. 

Moreover, assuming $\mu$ is non-indicial, it is 
\begin{enumerate}
\item[(A)] injective if and only if $\mu>\mu_0^-,$ in which case $\text{Ind}(L_{C(\Sigma)})=-J(\mu)$, and
\item[(B)] surjective if and only if $\mu<\mu_0^+,$ in which case $\text{Ind}(L_{C(\Sigma)})=J(-\mu-2)$.
\end{enumerate}
\noindent Hence in particular, it is an isomorphism for $\mu\in (\mu_0^-, \mu_0^+).$, i.e. $\mu+2-n/2 \in (\gamma_0^-,\gamma_0^+)$.
\label{mapb}
\end{prop}
\begin{proof} To show the index is finite, it suffices to calculate in individual eigenmodes. Since $\smash{u_j^\pm(1)\neq 0},$ any linear combination obeying Dirichlet boundary conditions must be a linear combination with non-zero coefficients for both solutions. Since $\gamma_j^-<\gamma_j^+$ is the more singular, injectivity occurs if and only if none of these are integrable, which is precisely where 
\[\mu+2 - n/2>\gamma_0^-,\]

\noindent i.e. $\mu >\mu_0^-.$
 Because $\smash{L_{C(\Sigma)}}$ is symmetric with respect to the unweighted volume pairing, the adjoint operator is again $\smash{L_{C(\Sigma)}},$ now viewed with the dual weight $-\mu.$ The linear combination of $\smash{u_j^\pm}$ obeying Dirichlet boundary conditions is integrable precisely for those eigenvalues lying in  \eqref{Jmuagain}. The surjectivity is simply injectivity for this dual weight. Closed range (hence Fredholmness) follows from a simple integration by parts argument.  
\noindent 

\end{proof}

For $\mu$ obeying \eqref{weightconstraintsbcase}, we are necessarily in the injective case (A) in Proposition \ref{mapb}. For such a $\mu$, the arithmetic shows the following for the solvability of the homogeneous problem.

\begin{lm}
If $\mu$ satisfies \eqref{weightconstraintsbcase}, then for any $\varrho \in \mathcal B_0^\perp$, there exists a unique solution $u$ to the boundary value problem 
\[
L_{C(\Sigma)} u=0, \qquad \text{Tr}(u)=\varrho
\] 
with $u \in r^{\mu+2}L^2$. However, if $\varrho$ contains any component in the $\calB_0$ summand, then there are no $r^{\mu+2}L^2$-bounded solutions to 
this boundary value problem.
\label{linmapconic}\qed
\end{lm}

\medskip

We now state and (re)prove the theorem of Caffarelli-Hardt-Simon. Their result is ultimately an application of the Implicit Function Theorem using the properties of Proposition \ref{mapb} for the linearization. Write $\mathcal B_0^{\perp,s} = \calB_0^\perp \cap H^{s}(\Sigma; N\Sigma)$ for $s>0$; no 
regularity filtration is needed for the finite-dimensional complement $\calB_0.$  Finally, any $\Phi \in r^{\mu+2}  H^{s+2}_b(X_0, N X_0)$ with sufficiently 
small norm parametrizes a normal graph over $C_1(\Sigma)$, which we denote by 
\[
X_{\Phi}:=\text{exp}_{NC_1(\Sigma)}(\Phi).
\]
\noindent This is an embedded  submanifold away from the origin, where it has an isolated conical singularity. 

\begin{thm}[Caffarelli-Hardt-Simon, \cite{CHS}]\label{CHS}
Fix $s > n/2$ and $\mu$ as in \eqref{weightconstraintsbcase}. Then there exists a neighborhood of the origin $\mathcal V\subseteq 
\mathcal B_0^{\perp,s+3/2}$ and a smooth map $F: \calV \to r^{\mu+2} H^{s+2}_b(X_0;NX_0) $ such that if $F(\varrho) = \Phi$ 
then the corresponding conical submanifold $X_\Phi$ obeys
\begin{enumerate}
\item[(I)] $X_{\Phi}$ is minimal in the Euclidean metric, 
\item[(II)] $X_{\Phi}$ has an isolated conical singularity at $\{0\}\in \R^N$, and
\item[(III)] The boundary values are given by 
\[\Phi|_{r=1}=(\zeta(\varrho),\varrho) \in \mathcal B_0\oplus \mathcal B_0^\perp,\]
\noindent where $\zeta: \mathcal B_0^\perp \to \mathcal B_0$ is smooth. 
\end{enumerate}
\noindent That is to say, $X_{\Phi}$ has prescribed boundary data $\varrho$ in $\mathcal B_0^\perp$, and boundary values in $\mathcal B_0$ which are
determined implicitly as a function of $\varrho$. 
\end{thm}
\begin{proof}
If $\Phi \in r^{\mu+2}H^{s+2}_b(X_0; NX_0)$ is sufficiently small, the normal graph $X_\Phi$ is an embedded conic submanifold. This uses 
$\mu+2>n/2+1$ to ensures that $\Phi$ vanishes faster than linearly as $r \to 0$, while  $s>n/2$ ensures that elements of $H^s_b$
may be multiplied, cf.\ Section \ref{SobolevSpaces} . Using these, the mean curvature defines a smooth non-linear map
\[
\mathcal Q: r^{\mu+2}H^{s+2}_b(X_0;NX_0)\to r^{\mu} H^s_b(X_0;NX_0).
\]
We are thus searching for a solution to the boundary problem 
\be 
\mathcal Q(\Phi)=0 \qquad \qquad \pi^\high (\Phi|_{r=1})=\varrho.
\label{CHSbvp}
\ee

As explained in the introduction, we now recast this boundary value problem by explicitly controlling boundary values via the Poisson extension. Define the Poisson operator for $L_{C(\Sigma)}$, separated into its action on the low and high eigenspaces of $L_\Sigma$, as follows. The two components are 
\[
P_0:   \mathcal B_0 \to r^{\mu+2}H^{s+2}_b(X_0; NX_0),\qquad P_0^\perp: \mathcal B_0^\perp \to r^{\mu+2}H^{s+2}_b(X_0; NX_0),
\]
where,  writing $\calB_0 \ni \zeta=\sum_{j=1}^J \zeta_j \phi_j$ and $\calB_0^\perp \ni \varrho=\sum_{j>J}\varrho_j \phi_j$ for $J=\dim \mathcal B_0$ as before, then
\[
P_0(\zeta)=\sum_{j=1}^J \zeta_j u_j^+(r,x) \hspace{1cm}P_0^\perp(\varrho)=\sum_{j>J} \varrho_j u_j^+(r,x)
\]
\noindent with $u_j^+$ as in \eqref{conicalsolutions}. By definition, the left composition of either of these with restriction to $r=1$ is the identity. 

Given $\varrho \in \mathcal B_0^{\perp, s+3/2},$ solving \eqref{CHSbvp} is equivalent to finding a zero of the map 
\begin{align}
\calB_0 \oplus r^{\mu+2}H^{s+2}_{b,D}(X_0, NX_0) &\lre  r^{\mu} H^s_b(X_0; NX_0) \nonumber\\
(\zeta,\ph)& \longmapsto \mathcal Q(\ph  \  + \  \chi P_0^{}(\zeta) \ + \   P_0^{\perp}  (\varrho) ), 
\label{parameterizedIFT}
\end{align}
\noindent where $\chi$ is a smooth cut-off function equal to $1$ for $r\geq 1/2$ and vanishing for $r\leq 1/4$. The cut-off is included only
on the first extension because the Poisson extension fails to be integrable in precisely these eigenmodes by Lemma \ref{linmapconic}.  The existence of a zero of this map is immediate from the Implicit Function Theorem when $\ph$ is small provided the linearization 
\begin{align*} 
 \mathbb L_{}:  \calB_0\oplus r^{\mu+2}H^{s+2}_{b,\Dir}(X_0, NX_0) &\to r^{\mu}H^s_b(X_0; NX_0) \\
(\zeta,\ph) &\mapsto L_{C(\Sigma)}(\ph + \chi P_0(\zeta)) 
%\label{conicalJacobiextension}
\end{align*}
is surjective. We show, in fact, that $\mathbb L_{}$ is an isomorphism. By Proposition~\ref{linmapconic}, $L_{C(\Sigma)}:  r^{\mu+2}H^{s+2}_{b,\Dir} \to r^{\mu}H^s_b$ is Fredholm and injective, with $J$-dimensional cokernel. 
{Decompose $r^\mu H^s_b$ into the direct sum of $\text{range}(L_{C(\Sigma)})$ and its $r^\mu L^2$-orthogonal complement, with corresponding
$r^\mu L^2$-orthogonal projections $\Pi_0$ and $\Pi_0^\perp$. In terms of this splitting, the linearization is a block lower-triangular matrix 
\be 
\mathbb L (\zeta, \ph)=\begin{pmatrix} T_0 & 0 \\ S_0   & L_{C(\Sigma)} \end{pmatrix}\begin{pmatrix} \zeta \\ \ph \end{pmatrix},\label{conicalblockdiag}
\ee
\noindent where $T_0=\Pi_0 L_{C(\Sigma)} \chi P_0$ and $S_0 =\Pi_0^\perp L_{C(\Sigma)}\chi  P_0$.} The lower right block is an isomorphism since $L_{C(\Sigma)}$ is injective, and 
obviously surjective onto its range, while the only relevant feature of $S_0$ is simply that it is a bounded map. It therefore suffices to show that the top 
left block $T_0: \calB_0 \to \text{range}(\Pi_0^\perp) \simeq \R^J$ is an isomorphism.  

The orthogonal complement of the range is spanned by the functions
\be
\Psi_j=r^{2\mu}( \alpha_j r^{\gamma_j^-} + \beta_j r^{\gamma^+_j})\phi_j,  
\label{prelimcoker}
\ee 
where $\alpha_j,\beta_j$ are chosen so that $\Psi_j$ vanishes at $r=1$, and normalized by $\|\Psi_j\|_{r^\mu L^2}=1.$  The upper left block $T_0$ 
is then given by the pairing 
\be 
T_0(\zeta)= \sum_{j=1}^J  \br L_{C(\Sigma)} (\chi P_0 (\zeta)) , \Psi_j \kt_{r^{\mu}L^2} \ \phi_j.
\label{T0conical}
\ee
{Using the pairwise orthogonality of the $\phi_j$, then integrating first in $\Sigma$ factor shows that this pairing is diagonal.}
To show that it is nondegenerate, we integrate by parts to obtain
\begin{equation}
%\begin{align}
\label{Greens2ndbd}
\begin{split}
& \br  L_{C(\Sigma)} (\chi P_0(\zeta)) \  , \  \Psi_j\kt_{r^{\mu}L^2}  = \int_{X_0}\br  L_{C(\Sigma)} (\chi P_0(\zeta)) \  , \  \Psi_j\kt \,  r^{-2\mu}r^{n-1} \, dr dV_\Sigma  \\
& = \int_{X_0} \br  \chi P_0(\zeta)  \ , \ L_{C(\Sigma)}  (r^{-2\mu}\Psi_j)\kt_{L^2} \, r^{n-1} dr dV_\Sigma + \int_{\Sigma} \br  P_0(\zeta), \del_r \Psi_j \kt_{ } - 
\br \del_r  P_0(\zeta) , \Psi_j \kt_{} \, dV_\Sigma\\ & = \int_{\Sigma} \br \zeta \ , \ \del_r \Psi_j  \kt  \, d\Omega,
\end{split}
\end{equation}
since $\chi\equiv 1$ at the boundary. The first term in the middle expression vanishes because the $r^{-2\mu}$ weight cancels with the opposite factor in $\Psi_j$, while the last term vanishes since 
$\Psi_j$ vanishes at $r=1$.  However, because $\Psi_j$ has vanishing Dirichlet boundary values, the corresponding Neumann boundary value $\del_r \Psi_j$ is nonzero 
at $r=1$ for each $j$, which implies that the pairing is non-degenerate. It follows that $T_0$ is an isomorphism, which completes the proof.
% with $\Phi=\ph  \  + \  \chi P_0^{}(\zeta) \ + \   P_0^{\perp}  (\varrho)$ for $(\zeta,\ph)$ giving the zero of (\refeq{parameterizedIFT}).
\end{proof}

\begin{rem} As explained in the introduction to this section, the preceding proof is a simplified model fro the cylindrical/edge case. In that setting, we include two additional infinite-dimensional degrees of freedom $(\xi,\alpha)$ which are $S^1$-parameterized translations and rotations of the $\R^N$ factor in the ambient $S^1\times \R^N$ respectively. As in this conic setting, the linearization there 
maintains the lower triangular block structure, but in that setting the low eigenmode components are infinite dimensional,
hence the pairing analogous to the one in \eqref{T0conical} involves infinite rank factors.   In fact, the now infinite sequence of inner products decay 
rapidly, thus the pairing, although it remains non-degenerate, fails to have closed range. This is the source of the loss of regularity, which necessitates Nash-Moser theory in the edge case. 
\label{lossrem}
\end{rem}
\medskip

%% file: sections/sec_edge_calculus.tex
%%
%% sec_edge_calculus.tex
%%

\section{Basics of the Edge Calculus}
\label{PreliminariesSection}
We begin our analysis of minimal submanifolds with edge singularities by extending the discussion of the Jacobi operator $L_0$ to $S^1\times C_1(\Sigma).$ 
This differs from the Jacobi operator on $C_1(\Sigma)$ only by the addition of the Laplacian in the product factor. Indeed,
\be
L_0\ph = \left(L_{C(\Sigma)}  +\Delta_t \right) \ph 
=\left(-\left(\p{^2}{r^2} \ + \ \frac{(n-1)}{r}\p{}{r} \right) \ + \ \frac{\Delta_\Sigma -hh^\star}{r^2} \ + \ \Delta_t \right)\ph\label{L0expression},
\ee
\noindent where $\Delta_t=-\del_t^2$ is the positive definite Laplacian on $S^1$. In the previous setting, we used the $b$-calculus to study the Jacobi operator, but the degeneracy structure here is a bit different: indeed, $r^2 L_0$ is an `elliptic 
combination' of the basic vector fields $r\del_r, r\del_t, \del_{x_j}$ for $(x_1,...,x_{n-1})$ again local coordinates on $\Sigma.$ Operators of this type are called {\bf elliptic differential edge operators}. 

Since we shall be drawing heavily on the existing theory of such operators, as developed in \cite{Mazzeo1991EdgeOperators, MazzeoEdgeOperatorsII},
we begin by reviewing some basic facts about this theory, including definitions of natural function spaces on which such operators act and
the main `classical' mapping properties of elliptic edge operators on these spaces. 

\begin{notation} \label{calDdef}

Throughout the article, we tacitly use complex Fourier series with the complex-valued eigenfunctions $e^{i\ell t}$ for vector fields along the edge $S^1.$ It is understood in this that for real vector bundles (e.g. $NX_0, \mathcal B_0$) we restrict to the real part, i.e. if $\eta:S^1\to \R$ then
\[
\eta(t) = \sum_{\ell \in \N} \Big (a_\ell \cos(\ell t) -b_\ell \sin(\ell t)\Big)= \text{Re}\left(\sum_{\ell \in \Z} \eta_\ell e^{i\ell t}\right),
\]
where $a_\ell, b_\ell \in \R$ and $\eta_{\pm \ell} = \frac{1}{2}(a_\ell \pm ib_\ell) \in \C$, but we will simply write $\sum_{\ell \in \Z} \eta_\ell e^{i\ell t}.$ 
\end{notation}

\subsection{Edge operators}  \label{edgeoperators}   
Let  $X_0=S^1\times C_1(\Sigma)$ denote the product minimal submanifold, and $\wtX_0$ its blow-up, i.e. the space obtained by
replacing $\Gamma_0 =  S^1\times \{0\}$ by $ S^1\times \Sigma.$ Alternately, if $r$ is the polar distance in $\RR^N$, $t$ the 
standard arclength coordinate on $S^1$ and $x = (x_1, \ldots, x_{n-1})$ local coordinates on $\Sigma$, then $(t,r,x)$ constitute a smooth local coordinate 
system on $\wtX_0$ near $r=0$; we have added $S^1\times \{0\} \times \Sigma$ as a new boundary face.
\begin{defn} We define \label{edgedefs}
\begin{enumerate}[leftmargin=30pt]
\item[(A)] The algebra of {\bf edge vector fields} $\calV_e(\wtX_0)$ is the space of smooth vector fields on $\wtX_0$ which, at $r=0$,
lie tangent to the fibers $\{r = 0, t = \mathrm{const} \}$. In local coordinates, any $V \in \calV_e(\wtX_0)$ is a finite
linear combination of smooth (on $\wtX_0$) multiples of the generating vector fields $\{r\del_r, r\del_t, \del_{x_j}\}.$ Note that $\calV_e$ is closed under the usual Lie bracket of vector fields. 
\item[(B)] The ring of {\bf edge operators} $A$,
which are those differential operators which can be expressed as locally finite sums of products of elements of $\calV_0(\wtX_0)$.
These are filtered by their order $m$. Thus, again in local coordinates, we can write any $A \in \mathrm{Diff}_e^m(X_0)$ as
\begin{equation}
A = \sum_{j + \ell + |\alpha| \leq m}  a_{j \ell \alpha}(t,r,x) (r\del_r)^j (r\del_t)^\ell \del_x^\alpha,
\label{edgeop}
\end{equation}
where all $a_{j\ell \alpha} \in C^\infty(\wtX_0).$
\item[(C)] Given such an $A,$ the {\bf edge symbol} is
\[
{}^e\sigma_m(A)(t,r,x; \tau, \rho, \xi) = \sum_{j+\ell + |\alpha| = m} a_{j\ell \alpha}(r,t,x) \rho^j \tau^\ell \xi^\alpha.
\]
(This has an invariant interpretation, omitted here, as a smooth function, polynomial on the fibers of a certain vector bundle over $\wtX_0$.)
The edge operator $A$ is called {\bf elliptic} if this symbol is invertible for $(\rho,\tau,\xi) \neq 0.$ 
\end{enumerate}
\end{defn}

\noindent The above definition supposes that $A$ acts on scalar functions. There is a straightforward extension to such operators acting on sections of some vector
bundle $E,$ in which setting, using local trivializations, each coefficient $a_{j\ell \alpha}\in \text{End}(E)$ is a smooth bundle map. The derivatives 
are likewise replaced by covariant derivatives $\nabla_V$ for $V\in \calV_e$ with respect to a chosen connection on $E$. 
As already noted, the Jacobi operator $L_0$ is, up to the multiplicative factor $r^2$, an elliptic edge operator acting on sections of the normal bundle $NX_0$. 

Just as for $b$-operators, we define the set of {\it indicial roots} of any (elliptic) edge operator. This set consists of those values $\gamma \in \CC$ 
for which there exists a smooth function $\varphi(x)$ such that
\be
A (r^\gamma \varphi(x)) = O( r^{\gamma+1}). \label{indicialcancellation}
\ee
The expected order of this expression on the left is $r^{\gamma}$, so in order for $\gamma$ to be an indicial root, there must
be a leading order cancellation.  Carrying this out for the Jacobi operator, we see that the indicial roots of $L_0$ are precisely
the same as those for $L_{C_1(\Sigma)}$, namely %{\color{blue} the expression (\refeq{indicialroots})
\be 
\text{Ind}(L_0)=\left\{ \gamma_{j}^\pm \ \Big | \ \gamma_j^\pm = \frac{2-n}{2} \pm \sqrt{ \frac{(n-2)^2}{4} + \lambda_j}\right\},
\label{indroots}
\ee
where $L_0$ is now the edge operator \eqref{L0expression}, and the $\lambda_j$ are the eigenvalues of the Jacobi operator $L_\Sigma$. The coefficient function
$\varphi(x)$ associated to each $\gamma_j^\pm$ for which \eqref{indicialcancellation} holds is precisely the corresponding eigenfunction of $L_\Sigma$ with eigenvalue $\lambda_j$.
 Although there most likely exist examples of minimal submanifolds $\Sigma$ where some $\lambda_j \leq - (n-2)^2/4$ for some
finite number of $j$, we tacitly assume this is not the case (and later impose  spectral rigidity in the sense of Definition \ref{spectralassumptions} anyway, from which this follows {\it a fortiori}) which is equivalent to the assumption that $\gamma_j^\pm \in \R$ are real for all $j$.

\begin{rem}\label{constantindicialrem}
For general edge operators, the indicial roots may depend on the point $(t_0,0, x_0) \in X_0^\circ$, but we shall restrict attention here only to
those operators where indicial roots are independent of boundary variables.  This property is obvious for $L_0$, but is not true for 
the Jacobi operator $L$ on arbitrary perturbations of $S^1\times C(\Sigma)$. Nevertheless, we ultimately work with a class of perturbations for which this is true, constructed in detail in Section \ref{deformationssection}. We therefore tacitly assume, henceforth, that we work with a class of edge operator for which the indicial roots coincide with those of $L_0.$

\end{rem}

\subsection{Sobolev spaces}
\label{SobolevSpaces} 
We now define an appropriate family of Banach spaces on which elliptic edge operators are well-behaved. As in Remark \ref{constantindicialrem}, we restrict to the class of operators $L$ with $\text{Ind}(L)=\text{Ind}(L_0).$

 The simplest mapping properties of any such $L$ 
are on function spaces associated to the structure algebra $\calV_e$, so accordingly we define the (weighted) edge Sobolev spaces.  

\begin{defn}  \label{edgesobolevdef} For any integer $k \geq 0,$ and $\mu\in \R$ the {\bf Edge Sobolev Spaces} of sections of a bundle $E\to X_0$ of regularity $k$ and weight $\mu$ are 
\[
r^\mu H^k_e(X_0;E) = \left\{u \in L^2 \ \Big | \  \int_{X_0} | \nabla_e^\alpha u|^2 \ r^{-2\mu} \ r^{n-1}dr dt dV_\Sigma < \infty, \ \ \forall\, \alpha\ \mbox{s.t.}
\  |\alpha| \leq k\right\},
\] where $\nabla_e$ denotes the covariant derivative with respect to an edge vector field $r\del_r, r\del_t, \del_x$. These spaces are equivalent for different choices of connection on $E$. When $k=0$ these spaces coincide with $r^\mu L^2(X_0; E)$.  
\end{defn}
We remark that these spaces are most properly defined over the blow-up $X_0^\circ,$ but since we are ultimately interested in weights $\mu$ that ensure vanishing at $r=0,$ the conflation of these domains is innocuous here. It is immediate from this definition that
\be
L: r^{\mu+2} H^{k+2}_e(X_0, N X_0) \longrightarrow r^{\mu} H^k_e(X_0, NX_0)
\label{bmL0e}
\ee
is bounded for every $\mu \in \RR$ and $k \geq 0$. This follows immediately from the fact that $r^2 L$ is (locally) a sum of products
of at most two elements of $\calV_e$. As we discuss below, using the edge ellipticity, it is possible to show that this mapping 
is Fredholm for a limited range of values of $\mu$,
and semi-Fredholm for most other values. Unfortunately, the geometric requirements of our problem --- namely that the perturbing vector fields vanish at least linearly along the singularity --- demand that we consider this operator
acting on space where the weight parameter is outside the Fredholm range, with an infinite dimensional cokernel.

One serious defect of the edge Sobolev spaces (Definition \ref{edgesobolevdef}) is that elements in them have very little regularity along the edge $\Gamma_0=S^1\times \{0\}.$  One
partial remedy for this is to consider a slightly less singular set of structure vector fields, namely, the algebra of $b$-vector fields
$\calV_b(\wtX_0)$ introduced in the last section. These are generated over $\calC^\infty(\wtX_0)$ by $r\del_r$, $\del_t$, and the 
$\del_{x_j}$. Stated invariantly, these are the vector fields which are smooth on $\wtX_0$ and merely required to 
lie tangent to the $r=0$ component of $\del \wtX_0$.    We defined the $b$-Sobolev spaces in terms of 
these, and while it is tempting to simply use these instead of the edge Sobolev spaces, this is not possible since it is not necessarily
true that $L$ has good mapping properties on $b$-Sobolev spaces.  We thus compromise by considering a class of hybrid spaces. These hybrid spaces are defined to consist of those sections which remain in $L^2$ even when differentiated by at most $k+s$ 
vector fields, where $s$ of these are allowed to lie in $\calV_b$ and the rest in $\calV_e.$ If $\nabla$ is a connection on a vector bundle $E\to X_0,$ then the covariant derivatives with respect to the two types of vector fields are \begin{equation}
\begin{split}
\nabla^b&=dr\otimes r\nabla_r \ + \ dx_i \otimes \nabla_{x_i} \ + \ dt \otimes \nabla_t,  \\
\nabla^e&=dr\otimes r\nabla_r \ + \ dx_i \otimes \nabla_{x_i} \ + \ dt \otimes r\nabla_t. 
\end{split}
\label{bederivs} 
\end{equation}

\noindent Observe that $|\nabla_e u|\leq |\nabla_b u|$ pointwise. 

\begin{defn} \label{hybridspacesdef}
For integers $k,s\geq 0$ and $\mu\in \R$, the {\bf hybrid $b$-e Sobolev spaces} are 
\begin{align}
r^\mu H^{s,k}_{b,e} (X_0; E)&: =\Big\{u \in L^2 \  \Big | \  \|u\|_{\mu,(s,k)} < \infty\Big\}. \nonumber\\
\intertext{where}
\|u\|^2_{\mu,(s,k)}&:= \sum_{\alpha,\beta}\int_{X_0} | \nabla_e^\alpha \nabla_b^\beta u|^2 \ r^{-2\mu} \ r^{n-1}dr dt dV_\Sigma \label{hybridnorms}
\end{align}
\noindent where the sum is over multi-indices $\alpha,\beta$ such that $
 |\beta| \leq s  \ , \ |\alpha|+|\beta|\leq s+ k.$
 When either $s,k$ or both are zero, these coincide with the previous spaces 
$r^\mu H^s_b$ and $r^\mu H^k_e, r^\mu L^2,$ and when $k=0$ the norms simply by $\smash{\|-\|_{\mu,s}}$ These spaces are once again independent of the choice of a smooth background connection on $E$. As before, denote subspaces with vanishing Dirichlet boundary values as $\smash{r^\mu H^{s,k}_{b,e,D}(X_0;NX_0)}.$
\end{defn}

If $V \in \calV_e$ and $W \in \calV_b$, then $[V,W] \in \calV_e$, which yields that 
\be
L: r^{\mu+2} H^{s,k+2}_{b,e}(X_0, NX_0) \longrightarrow r^{\mu} H^{s,k}_{b,e}(X_0, NX_0)
\label{bmL0eb}
\ee
is bounded for every $\mu$.

Due to the quasi-linearity of the mean-curvature equation, we must work with a range of weights $\mu$ and regularities $(s,k)$ that afford sufficient control on vector fields under (i) multiplication and (ii) composition with smooth functions. In fact, because our later use of Nash-Moser theory allows an arbitrary loss of $b$-derivatives, it suffices to show these for the pure $b$-spaces, i.e. those with $k=0$ edge derivatives.

In this case, we also consider the $b$-analogue of the $C^k$-spaces
\be 
r^\mu C^{k}_b (X_0; NX_0):=  \Bigg \{ \ph \in L^2(X_0; NX_0) \  \ \Big |  \ \ \|\ph\|_{r^{\mu} C^{k}_b}<\infty\Bigg \}. 
\label{stdCk}
\ee
 using the norms
\be 
\|\ph\|_{r^{\mu}C^{k}_b} :=  \sum_{|\alpha| \leq k} \sup_{y \in X_0}  \big | \left(r^{-\mu} 
\nabla_b^\alpha \ph\right)(y) \big | 
\label{holderCbnorm} 
\ee
There are similar generalizations involving H\"older seminorms, but we will not need these. The analogue of the standard H\"older embedding is: 

\begin{lm} \label{standardHolder}Suppose that $s > \tfrac{\dim(X_0)}{2}.$ Then 

\[ r^\mu H^s_b(X_0;NX_0)\hookrightarrow r^{\mu-\tfrac{n}{2}}C^k_b(X_0; NX_0)\]

\noindent is a continuous embedding. 
\end{lm}
\begin{proof} Let $r=e^{\tau},$ so that $X_0=S^1 \times (-\infty,0]_\tau \times \Sigma.$ The chain rule shows that $r\del_r=\del_\tau.$ Thus this change of variables induces an isomorphism 

\[r^{\mu}H^s_b(X_0;NX_0) \simeq e^{(\mu-\tfrac{n}{2})\tau}H^s(S^1 \times (-\infty,0]_\tau \times \Sigma; r^*NX_0)\]

\noindent where the shift in the weight comes from the volume form $dV_{C(\Sigma)}=r^{n-1}dr dV_\Sigma.$ Multiplicaiton by the weight 
$e^{(\mu-n/2)\tau}: H^s \to e^{(\mu-n/2)\tau}H^s$ is an isomorphism, and likewise for $C^k.$ The conclusion then follows from the standard embedding on $H^s(S^1\times (-\infty,0] \times \Sigma)$ after conjugating by the weight. 

\noindent 
\end{proof}

Similarly, we have the following multiplication and interpolation theorems: 

\begin{lm} \label{multiplications} Suppose that $s> \frac{\dim(X_0)}{2}.$
\begin{enumerate}

\item[(A)] Given $0\leq s_1 < s < s_2$, there is a constant $C=C({s,s_1,s_2})$ such that the interpolation inequalities 
\[ 
\|\ph\|_{\mu,s} \leq C \|\ph\|^\alpha_{\mu, s_1} \|\ph\|_{\mu, s_2}^{1-\alpha}
\]
hold, where $\alpha=\tfrac{s_2-s}{s_2-s_1}$.
\bigskip 

\item[(B)] If $\mu+1>\tfrac{n}{2}$, there are constants $C=C(\mu,\sigma)$ such that a smooth fiberwise bilinear bundle 
map extends to a continuous bilinear operator 
\[ 
r^{\mu+1}H^s_b(X_0;NX_0)   \times r^{\mu+1}H^s_b(X_0;NX_0)  \lre r^{\mu+1}H^s_b(X_0;NX_0).  
\]

\bigskip

\item[(C)] When $\mu+1>\tfrac{n}{2}$ but $\mu<\tfrac{n}{2},$ there is a dimensional constant $a=a(n)$ such that a smooth fiberwise bilinear bundle map  extends to a bounded operator
\[ 
 r^{\mu+1}H^{s+a(n)}_b(X_0;NX_0)   \times r^{\mu}H^s_b(X_0;NX_0)  \lre r^{\mu}H^s_b(X_0;NX_0).  
\]
\end{enumerate}
\medskip 

\noindent The same results hold for smooth bilinear bundle maps between any bundles of tensors. 
\end{lm}

\begin{proof}Under the same coordinate change as in the proof of Lemma \ref{standardHolder}, the sub-multiplicativity of the exponential weights 
\[e^{(\mu+1-n/2)\tau }f\cdot e^{(\mu+1-n/2)\tau}g\leq e^{(\mu+1-n/2)\tau}(fg) \hspace{2cm} e^{(\mu+1-n/2)\tau}f\cdot e^{(\mu-n/2)\tau}g\leq e^{(\mu-n/2)\tau}(fg) \]
\noindent show that, after conjugation by the weight, these reduce to the standard interpolation for (A), and multiplications $H^s\times H^s\to H^s$ and $C^s\times H^s\to H^s$ on $S^1\times (-\infty,0] \times \Sigma$ for (B) and (C) respectively.
\end{proof}

Finally, we have the following composition theorem. For this, we let $n(\Sigma)=\{(x,\nu) \in NX_0 \ | \ |\nu|<\rho_0\}$

\noindent denote the disk bundle of radius $\rho_0.$ For $\rho_0$ sufficiently, exponential normal coordinates $(x,\nu)$ identify this with a tubular neighborhood $n(\Sigma)\subseteq S^{N-1}$ denoted by the same symbol. In these exponential normal coordinates the metric may be written 

\[g_{S^{N-1}}= g_{\Sigma} + d\nu^2 + h(x,\nu)\]

\noindent where $h(x,\nu)=O(\nu).$ Perturbations by the graphs of normal vector fields $\Phi$ contribute terms of the form $h(x,\Phi(t,r,x))$ and its derivatives to the mean curvature. These are bounded by the following lemma:

\begin{lm} \label{composition} Suppose that $h: n(\Sigma)\to \text{Sym}^2(T^*S^{N-1})$ is a smooth symmetric tensor vanishing along the zero section. Then,  there is a $\delta_0>0$ if $\Phi \in r^{\mu+1}H^s_b(X_0;NX_0)$ with $s> \text{dim}(X_0)/2$ and $\|\Phi\|_{\mu+1,{ \dim(X_0)/2}} <\delta_0,$ then 

\[\|h(x,\Phi(t,r,x))\|_{\mu+1, s} \leq C_s \|\Phi\|_{\mu+1, s} \]
\end{lm}

\begin{proof} This is a straightforward generalization of \cite[Prop 13.3.9]{TaylorPDEsIII}\footnote{There are several typos in \cite[Prop 13.3.9]{TaylorPDEsIII}, including the statement, where the $+1$ on the right-hand-side is not needed when $F(0)=0$ using the notation therein.} using similar considerations as the proofs of Lemmas  \ref{holderCbnorm} and \ref{multiplications}. 
%See also the proof Lemma \ref{compositionslice}, which establishes the same estimate in the analytic setting.   
    \end{proof}

\subsection{Fredholmness and closed range}
Unlike in standard (uniform) elliptic theory on compact manifolds, it is not true in this edge theory that edge ellipticity guarantees Fredholmness; another `symbolic' condition is needed.

\begin{defn} We define the following. 
\begin{enumerate}
\item[(A)]Given an edge operator $A$ as in \eqref{edgeop}, we define its {\bf normal operator} at a point $q = (0, x_0, t_0) \in X_0$ by
\begin{equation}
N_q(A) = \sum_{j + \ell + |\alpha|\leq m}  a_{j \ell \alpha}(0, x_0, t_0) (\curs{r}\del_{\curs{r}})^j (\curs{r} \del_{\curs{t}} )^\ell 
\del_{\curs{x}}^\alpha,
\label{normalop}
\end{equation}

\noindent Note that, while this
is still an edge operator, it is much simpler than $A$ itself inasmuch as it is translation invariant in $\mathtt t$ and
invariant with respect to dilations $(\curs{r}, \curs{t}, \curs{x})\mapsto (\lambda \curs{r}, \lambda \curs{t}, \curs{x})$
for any $\lambda > 0$.  
\medskip 
\item[(B)] An elliptic edge operator $A$ is called {\bf fully elliptic} with respect to some weight $\mu$ if its normal operator is globally invertible as a map
\begin{equation}
N_q(A): \curs{r}^{\mu} H^{m}_{e}( \RR^+ \times \Sigma \times \RR, dV_0) \longrightarrow \curs{r}^\mu L^2( \RR^+ \times \Sigma \times \RR, dV_0),
\label{norope}
\end{equation}
\end{enumerate}
\end{defn}
\medskip

The normal operator be reduced to a much simpler operator by passing to the Fourier transform in $\curs{t}$
and rescaling, cf.\ \cite{Mazzeo1991EdgeOperators}.  For the operator $r^2 L_0$, the normal operator perhaps confusingly is exactly the same as $r^2 L_0$ itself, save for
the fact that the variable $\curs{t}$ now lies in $\RR$ rather than $S^1$. This is because there are no `higher order terms'
in this simple product situation.  However, for the perturbations of $S^1\times C(\Sigma)$ considered below, the associated
Jacobi operator $L$ (or rather $r^2 L$) does have higher order contributions, so its normal operator is obtained by dropping
all of these higher order corrections, hence for these perturbations $N_q (r^2 L) = N_q(r^2 L_0).$

We now state the main mapping properties for $L$. For simplicity, we state the result only for perturbed Jacobi operators $L$ with $\text{Ind}(L)=\text{Ind}(L_0)$ rather than for an arbitrary elliptic edge operator. In the following, we impose vanishing Dirichlet boundary conditions at $r=1$ on the domain space in all instances, and set 

\[\mu_0^\pm:=\gamma_0^\pm + n/2-2,\]

\smallskip 
\noindent which are the zeroth indicial weights defined following \eqref{weightconstraintsbcase}. The basic result is then that \eqref{bmL0e} is Fredholm for a given value of $\mu$ if and only if \eqref{norope} is 
invertible.  Similarly, $L$ (with Dirichlet boundary conditions) is semi-Fredholm, i.e., has closed range with infinite dimensional nullspace,
respectively cokernel, but not both, if and only if \eqref{norope} is injective, respectively surjective. More precisely, 
\begin{prop}
Let $L$ be a second-order elliptic edge operator with $\text{Ind}(L)=\text{Ind}(L_0)$. Then $L$ is fully elliptic for all weights $\mu \in (\mu_0^-, \mu_0^+)$, and hence 
\begin{equation}
L: r^{\mu+2} H^{k+2}_{e,D}(X, NX) \longrightarrow r^{\mu} H^k_e(X, NX)
\label{bmLe}
\end{equation}
is Fredholm. Moreover, 
\begin{enumerate}
\item[(A)] If $\mu > \mu_0^+$ and $\mu$ is non-indicial, then 
\eqref{bmLe}  has closed range and a finite dimensional nullspace, but the cokernel is infinite dimensional.
\item[(B)] Meanwhile if $\mu < \mu_0^-$ and $\mu$ is non-indicial, then \eqref{bmLe} has closed range of finite codimension, but an infinite dimensional nullspace.  

\end{enumerate}
Moreover, {\it unperturbed} Jacobi operator $L_0$ is injective for weights $\mu > \mu_0^-$ and surjective if $\mu < \mu_0^+$ is non-indicial. 
\label{1mapprop}
\end{prop}

\begin{proof}[Proof sketch] 
This is proved using a parametrix for $L$.  This is an approximate inverse which lies in the space of pseudodifferential edge
operator $\Psi_e^*(X_0)$.  The precise definitions are omitted here, but we recall that this is an enlargment of the space of differential
edge operators which is sufficiently robust so as to contain approximate inverses for fully elliptic edge operators (as well as 
just left and right approximate inverses in the semi-Fredholm cases), up to compact errors. The conclusions of this proposition then follow 
directly from the mapping properties of these parametrices (see \cite[Sec. 6]{Mazzeo1991EdgeOperators} for further details).
\end{proof}

One useful feature of this approach is that these parametrices are constructed as Schwartz kernels (on an appropriate further blow-up of $X_0^\circ \times X^\circ_0$ de-singularizing corners), rather than as abstract operators, which
means that if we can prove these operators are bounded on other function spaces, then we immediately obtain mapping properties for the
original operator on those spaces instead.  In particular, we can apply this to the hybrid edge-b spaces from Definition \ref{hybridspacesdef} to obtain the following. 
\begin{cor}Let $L$ be a second-order elliptic edge operator with $\text{Ind}(L)=\text{Ind}(L_0)$. 
\begin{enumerate}
\item[(A)] For any $\mu\in \R,$ $L$ has the same mapping property as an operator 
\[r^{\mu+2}H^{s,k+2}_{b,e,D}(X_0; NX_0) \to r^{\mu} H^{s,k}_{b,e}(X_0;NX_0)\]
as it does acting between the edge spaces in \eqref{bmLe}.

\item[(B)] In particular, $\mu> \mu_0^-$ is non-indicial, then the {\it unperturbed} Jacobi operator obeys the elliptic estimates 
\[\|\ph\|_{\mu+2,(s,2)} \leq C \|L_0 \ph\|_{\mu,(s,0)}\]

\noindent holds for $\ph \in r^{\mu}H^{s,2}_{b,e}(X_0;NX_0).$
\end{enumerate}
\label{higherorderelliptic}
\end{cor}

\begin{proof}[Proof sketch]The proof proceeds by bounding commutators $[V, G]$ where $V \in \calV_b$ and $G$ is the parametrix, see \cite[Sec. 6]{Mazzeo1991EdgeOperators} for details once more. 
\end{proof}

In the remainder of this paper, our main concern will be with Jacobi operators acting on spaces of vector fields with $\mu+2-\tfrac{n}{2}>1,$ so as to ensures the resulting graphical submanifolds preserve the tangent cone along the singularity. 
With this choice of weights, however, as per Proposition \ref{1mapprop}, we are confronted with an operator
with closed range and infinite dimensional cokernel.  The next subsection takes up the task of characterizing this cokernel more precisely. 

%% file: sections/sec_infinite_obstruction.tex
%%
%% sec_infinite_obstruction.tex
%%

\section{The Infinite-Dimensional Obstruction}
\label{section4}
By Proposition \ref{1mapprop}, the Jacobi operator $L_0$ has closed range of infinite codimension when $\mu> \mu_0^+$ is nonindicial.
The orthogonal complement of this range, i.e., the cokernel, is naturally identified with the space of homogeneous solutions to the dual problem: 
\be 
\Psi \in \text{range}(L_0)^\perp \hspace{.3cm} \Longleftrightarrow \hspace{.3cm} L_0^\star \Psi =0, \label{dualproblem}
\ee
\noindent where $L_0^\star$ is the adjoint of the operator in \eqref{bmLe}.   This cokernel is the main obstruction to solving
the deformation problem via the implicit function theorem, which justifies
\begin{defn}\label{Obdef}
The  {\bf Obstruction Space} is the $r^{\mu}L^2$ orthogonal complement of the range of $L_0$: 
\[
\text{\bf{Ob}}_0= \text{range}(L_0)^\perp\subseteq r^{\mu}L^2(X_0; NX_0)
\]
As in Section 2, we denote the $r^{\mu}L^2$-orthogonal projections to $\text{\bf Ob}_0$ and $\text{range}(L_0)$ by $\Pi_0$ and 
$\Pi_0^\perp$, respectively. 
\end{defn}

\medskip 

Our goal in this section is to identify $\text{\bf Ob}_0$ with the space of sections of fixed regularity of a certain vector bundle 
$\mathcal C_0 \to \mathcal Z_0$ over the singular set. This identification is a particular instance of the boundary-value problem framework for elliptic edge operators developed in \cite{MazzeoEdgeOperatorsII}. See also \cite{krainer2013kernel,krainer2015elliptic} 

\subsection{Dual Spaces and Weights} \label{dualspacessection}  
There are several natural, but inequivalent ways of defining duality on a weighted $L^2$ space. For clarity we review these 
definitions and the relationship between them here. 
%The distinction between these two duals is elementary in this setting, 
%but the same framework applies in the more intricate setting of the analytic function spaces see Section \ref{LinearanalyticSection}. 
Namely, there is a dual of a given Hilbert space associated to any nondegenerate pairing, and in the present setting,
there are two obvious pairings: the inner product $r^{\mu}L^2 \times r^\mu L^2 \to \CC$, and the pairing 
$r^\mu L^2 \times r^{-\mu} L^2 \to \CC$ using the unweighted inner product. (We work with complex Hilbert spaces.)  
Denoting these duals by a ${}^\star$ and ${}^\dag$, respectively, then 
\begin{align*}
(r^\mu L^2(X_0; NX_0) )^\star &\cong r^\mu L^2(X_0; NX_0), 
\\ (r^\mu L^2(X_0; NX_0) )^\dag &\cong r^{-\mu} L^2(X_0; NX_0),
\end{align*}
The first is, of course, the standard identification given by Riesz representation. Slightly more generally, these same pairings 
also give the identifications $(r^{\mu + \alpha} L^2)^\star  \cong r^{\mu-\alpha} L^2$, and  $(r^{\mu+\alpha} L^2)^\dag \cong r^{-\mu-\alpha } L^2.$

Similarly, the duals of the higher-regularity edge spaces (with Dirichlet boundary values) with respect to these pairings are defined by
\begin{align}
(r^{\mu+2}  H^{2}_{e,\Dir})^\star  &:= \Big \{\, v \, \big | \ \ \exists C_v \ \text{s.t.} \ \br u, v\kt_{r^\mu  L^{2}} \leq C_v \|u\|_{\mu, (0,2)}  \ \ \forall u \in r^{\mu+2} H^{2}_{e,\Dir} \Big \}  \label{stardualH2} \\
(r^{\mu+2} H^{2}_{e,\Dir})^\dag & := \Big \{\,  \psi  \, \big | \ \ \exists C_\psi \ \text{s.t.} \ \br u, \psi \kt_{L^2} \leq C_\psi \|u\|_{\mu,(0,2)}  \ \ \forall u \in r^{\mu+2} H^{2}_{e,\Dir} \Big \},
\end{align}
where, as before, $\| \cdot \|_{\mu,(0,2)}$ is the norm on $r^{\mu+2} H^{2}_e$. We write these duals as
\begin{align}
(r^{\mu+2}  H^{2}_{e,\Dir})^\star & = r^{\mu-2} H^{-2}_e \label{dualstardef1}\\ 
(r^{\mu+2} H^{2}_{e,\Dir})^\dag& = r^{-\mu-2}  H^{-2}_e. \label{dualstardef2}
\end{align}
There are versions, defined analogously, associated to higher regularity, $r^{\mu-2}H^{s,-2}_{b,e}$, which require $b$-derivatives $\nabla_b^\alpha$ for $|\alpha|\leq s$ 
to lie in the spaces above. 

These dualities each determine an adjoint of $L_0.$. 
\begin{defn} \label{twoadjoints}  The adjoints of $L_0: r^{\mu+2}H^2_{e,\Dir}\to r^\mu L^2$ corresponding to these two dualities are:
\begin{alignat}{2}
& L_0^\star : r^{\mu}  L^{2} \to r^{\mu-2}  H^{-2}_e  & \qquad &L_0^\dag : r^{-\mu}    L^2 \to r^{-\mu-2}  H^{-2}_e \\
& L_0^\star \psi:= \br L_0( \cdot), \psi \kt_{r^\mu   L^{2}} & \qquad &L_0^\dag := \br L_0( \cdot ) , \psi \kt_{L^2},
\end{alignat}
\noindent where $\br \cdot ,v\kt_{L^2}$ and $\br \cdot,v\kt_{r^\mu L^2}$ are the unweighted and weighted $L^2$-pairings on $X_0.$  
\end{defn}
These two adjoints are related by  
\begin{align}\label{twoadjiontsexp}
L_0^\dag \  &= \  L_0: \,  r^{-\mu}L^2 \longrightarrow r^{-\mu-2} H^{-2}_e\\ \ 
L_0^\star  \ &= \   r^{2\mu} L_0 r^{-2\mu} : \,  r^{\mu}L^2 \longrightarrow r^{\mu-2}H^{-2}_e,
\end{align}
which follows from the formulas 
\begin{align*}
\langle L_0 u, v \rangle_{L^2}  &= \int (L_0 u) \overline{v}\, dV_0 = \int u \, (\overline{L_0 v})\, dV_0 = \langle u, L_0^\dag v \rangle_{L^2}, \\ 
\langle L_0 u, v \rangle_{r^{\mu} L^2} &= \int (L_0 u) \overline{v} \, r^{-2\mu}\, dV_0 = \int u ( \overline{r^{2\mu} L_0 r^{-2\mu}v})\, r^{-2\mu} \, dV_0 =
\langle  u,  L_0^\star v \rangle_{r^{\mu} L^2}
\end{align*}
using that $L_0$ is symmetric with respect to the volume form $dV_0$.

%The following formalizes the duality property of $L_0^\star.$ 
In the following, write $\text{range}(L_0)={\bf Rg}_0$; then 
\[
r^\mu L^2=\text{\bf Ob}_0 \oplus {\text{\bf Rg}}_0
\]
is an $r^\mu L^2$-orthogonal splitting. % by Definition \ref{Obdef}.
\begin{lm} \label{secondorderinvertibility}
Let $\mu>\mu_0^+$ be non-indicial, which ensures the conclusions of Proposition \ref{1mapprop}(A). Then 
\begin{enumerate}
\item[(A)] The operator
$$
L_{0}^\star L_{0}: r^{\mu+2}  H_{e,D}^{2} \lre r^{\mu-2}   H_{e}^{-2}
$$
is a bounded linear isomorphism with inverse denote $\mathcal P_0.$ 
\item[(B)]The composition   
\begin{center}
\tikzset{node distance=4.2cm, auto}
\begin{tikzpicture}
\node(A){$r^{\mu+2}  H^{2}_{e,D}$};
\node(C)[right of=A]{$ \begin{matrix}\text{\bf Ob}_0 \\  \oplus \\ \text{\bf {Rg}}_0  \end{matrix}$};
\node(C')[right of=C, yshift=0cm]{$r^{\mu-2}   H^{-2}_{e} $};
\draw[->] (A) to node {$L_{0} $} (C);
\draw[->] (C) to node {$L^\star_{0} $} (C');
\draw[<-,>=stealth] (A) edge[bend right=20]node[below]{$\mathcal P_{0} $} (C');
\end{tikzpicture}\end{center}
factors through the $\text{\bf Rg}_0$ summand. In particular, $L_0 \mathcal P_0$ is a right-inverse of $L_0^\star$ with image in $\text{\bf Rg}_0.$
\item[(C)] The $r^\mu   L^{2}$-orthogonal projections to $\text{\bf Ob}_0$ and $ {\text{\bf Rg}}_0$ can be written as
\[
\Pi_0=\text{Id}-L_{0}\mathcal P_{0} L_{0}^\star \qquad \Pi^\perp_0=L_{0}\mathcal P_{0} L_{0}^\star ,
\]
and these respect regularity in the sense that $\mathrm{Im}(\Pi_0 |_{r^{\mu} H^{s}_b})= \text{\bf Ob}_0^s :=\text{\bf Ob}_0 \cap r^{\mu} H^{s}_b$ and 
similarly for $\smash{\text{\bf Rg}^s_0}.$  
\end{enumerate}
Moreover, the same conclusions hold uniformly for any family of operators $L_p$, $p\in \mathcal V$, provided Proposition \ref{1mapprop}(A) holds uniformly.
\end{lm}
\begin{proof}
(A) This follows immediately from the semi-Fredholmness and injectivity of $L_0$.
%Definition \ref{twoadjoints}  implies that the fourth order operator $L_0^\star L_0$ is injective. For surjectivity, let $f\in r^{\mu-2} H_e^{-2}$, and consider the functional
%\[
%E_f(\ph):= \|L_0 \ph\|^2_{r^{\mu} L^{2} } + \br \ph, f \kt_{r^{\mu}  L^{2} }
%\]
%where the last term is the dual pairing between $r^{\mu+2}  H_{e}^{2}$ and $r^{\mu-2}  H_{e}^{-2}$. Proposition \ref{1mapprop}(A), Corollary \ref{higherorderelliptic}(B)  
%and Young's inequality imply that $E_f$ is coercive, and it is straightforward to verify that it is weakly lower semicontinuous. It thus admits a minimizer by 
%standard theory \cite[Ch. 13]{nicolaescu2020lectures}. The Euler-Lagrange equation of the minimizer is precisely $L_{0}^* L_{0} \ph=f$. 

(B) and (C) then follow precisely as the analogous statements in a finite-dimensional context. That the projections respect regularity follows from the 
version of (A) with $s\geq 0$. This implies that each individual operator is bounded on the same spaces with $s$ additional $b$-derivatives. 
\end{proof}

\subsection{Trace Bundles} Definition \ref{twoadjoints} and Lemma \ref{secondorderinvertibility} substantiate the characterization \eqref{dualproblem}
of $\text{\bf Ob}_0$.  By \eqref{twoadjiontsexp}, the multiplication map $r^{-2\mu}\cdot$ yields an isomorphism 
\[
r^{-2\mu}: \mathbf{Ob}_0 \, \overset{\cong}\lre \, \mathrm{ker}\, \big( L_0: r^{-\mu}H^{k+2}_{e,D} \longrightarrow r^{-\mu-2}H^k_e \big).
\]
This uses elliptic regularity for edge operators \cite{Mazzeo1991EdgeOperators}, which gives that $u \in \ker(L_0)\cap r^{-\mu}L^2$ iff $u \in r^{-\mu} H^k_e$ for 
any $k\geq 0$. Note, in particular, that we require here that elements of the space on the right vanish at $r=1$. 

We next show that $\mathbf{Ob}_0$ admits a parametrization in terms of the boundary traces of its elements at $r=0$.  
The general definition of the {\it trace bundles}, i.e., the bundles where these $r=0$ boundary traces take values, is somewhat intricate, 
cf.\ \cite{MazzeoEdgeOperatorsII, krainer2013kernel, krainer2015elliptic, usula2024boundary}, but in the present situation, we can use a simpler ad hoc definition.
\begin{defn} \label{tracebundledef}The {\bf trace bundle} $\mathcal C_0 \to S^1$ associated to a non-indicial weight $\mu> \mu_0^+$ is the trivial 
bundle with fiber 
\[
\mathcal C_0=\bigoplus_{j=1}^J \R \phi_j.
\]
Here $\R \phi_j$ is the span of $\phi_j$, and the sum is taken over those $\lambda_j$ (counted with multiplicity) such that \eqref{weightconstraint} holds. 
\end{defn}

Using the general theory developed in \cite{MazzeoEdgeOperatorsII, usula2024boundary}, then just as for the classical theory of boundary value problems, 
elements of $\ker(L_0)$ are in 1-1 corresponding with their boundary traces at $r=0.$ 
\begin{prop} \cite{MazzeoEdgeOperatorsII}. There is a (regularized) boundary trace map 
\[ \text{tr}_e: \text{\bf Ob}_0 \to L^2(S^1;\mathcal C_0),\]
\noindent which is a bounded linear isomorphism. 
\label{btraceregularity}
\end{prop}
\begin{proof} If $\Psi \in \mathrm{ker}(L_0) \cap r^{-\mu}L^2$, then by the regularity theory in \cite[Sec. 7]{Mazzeo1991EdgeOperators}, $\Psi$ admits a 
`weak' asymptotic expansion
\be 
\Psi(t,r,x) \sim \sum_{j = 1}^J \Psi_j^-(t) \phi_j(x) r^{\gamma_j^-} + \sum_{j=1}^\infty \Psi_j^+(t) \phi_j(x) r^{\gamma_j^+}.
\label{weakexpansion1}\ee
The fact that only those exponents $\gamma_j^-$ with $j \leq J$ appear in the first sum on the right is because
$\smash{r^{\gamma_j^-} \in r^{-\mu}L^2}$ if and only if 
\[
2(\gamma_j^- + \mu) + n > 0 \Longleftrightarrow 1 + \mu > \sqrt{ \frac{(n-2)^2}{4} + \lambda_j},
\]
which is precisely the condition \eqref{weightconstraint}, cf.\ also Definition \ref{tracebundledef}. 

The notion of a weak expansion is a bit subtle; it means simply  that if $\chi(t)$ is any smooth compactly supported function, then there
is a `classical asymptotic expansion
\[
\int_{S^1} \Psi(t,r,x) \chi(t)\, dt \sim \sum_{j = 1}^J \langle \Psi_j^-(t), \chi(t)\rangle \phi_j(x) r^{\gamma_j^-} + \sum_{j=1}^\infty \langle \Psi_j^+(t), \chi(t)\rangle 
\phi_j(x) r^{\gamma_j^+},
\]
where the coefficients are the distributional pairings of $\Psi_j^\pm$ with $\chi$. 

We claim that $\Psi$ is uniquely characterized by the collection of functions $(\Psi_1^-(t),...,\Psi_J^-(t))$, or in other words, that the 
map $\Psi \longmapsto ( \Psi_i^- )_{i=1}^J$ is injective.  To see this, suppose that $\Psi_j^-$, $j = 1, \ldots, J$ are all $0$.  
Then $\Psi \in r^{\mu+2}H^2_e$ and $|\Psi| \leq \smash{Cr^{\gamma_0^+}}$ \cite[Sec. 7]{Mazzeo1991EdgeOperators}, 
hence $\Psi\equiv 0$ by Corollary \ref{higherorderelliptic}.  Again by \cite[Sec. 7]{Mazzeo1991EdgeOperators}, there are values $a_j > 0$ such that 
\[
\text{tr}^\circ_e(\Psi):=(\Psi_1^-(t),...,\Psi_J^-(t)): \ker(L_0 |_{r^{-\mu}L^2}) \lre \bigoplus_{j=1}^J H^{-a_j}(S^1; E_{\lambda_j})
\]
is bounded (and of course, still injective). Conversely, \cite{MazzeoEdgeOperatorsII} shows that each such boundary trace admits an extension to a solution 
with $r^{-\mu}L^2$ which also vanishes at $r=1$ (this also follows from the explicit formul\ae\ from Lemma \ref{obnormalization} below).
We denote the regularity-shifting operators by $\text{tr}_e:=((1+\Delta_t)^{-a_1/2}, ...,(1+\Delta_t)^{-a_J/2})\circ \text{tr}_e^\circ$. 

This completes the proof.
\end{proof}

%\begin{rem}
%The reader should beware that the asymptotic expansions \eqref{weakexpansion} are classical expansion, and correct convergence statements for them require some care. They become classical expansions (in the sense that there are well-controlled smaller remainder terms) when integrated against
%a test function $\chi(y)$.  In other words, for any $\chi \in \calC^\infty(S^1)$, then 
%\be
%\langle \Psi(\cdot,r,x) , \chi(\cdot)\rangle \sim \sum_{j=0}^J \langle \Psi_j^-, \chi\rangle \phi_j(x) r^{\gamma_j^-} + 
%\sum_{j=0}^\infty \langle \Psi_j^+,\chi\rangle \phi_j(x) r^{\gamma_j^+}
%\label{weakexpansion}\ee
%is a classical expansion in the standard (pointwise) sense.
%\end{rem}

\subsection{The Obstruction}\label{obstructionsubsection}
Introducing the spectral decompositions for $\Delta_t$ and $L_\Sigma$, the Jacobi operator $L_0$ reduces to a countable collection of ordinary differential operators
\be 
L_{j\ell} :=  \frac{\del^2\,}{\del r^2} + \frac{n-2}{r}\frac{\del\,}{\del r} - \frac{\lambda_j}{r^2} - \ell^2 %\right) \Psi_{j,\ell}(r) = 0. 
\label{BesselODE}
\ee
This allows us to describe the elements of $\mathbf{Ob}_0$ explicitly in terms of Bessel functions of imaginary argument of the first and second kind, $I_\alpha(s)$ and $K_\alpha(s)$ (also called Modified Bessel functions).  
%This is done by reducing  elements $\Psi \in \mathbf{Ob}_0$ into Fourier modes with respect to both $L_\Sigma$ and $\Delta_{S^1}$. 
The precise characterization of the obstruction reduces to uniformity of various statements about this Bessel-type ODE. We begin with the

\begin{lm} \label{obnormalization} For each pair $(j,\ell)\in \{1, \ldots, J\} \times \mathbb Z$, the unique solution to $L_{j\ell} \Psi_{j\ell}^\dag = 0$ with
$\Psi_{j\ell}^\dag(1) = 0$ and $r^{-\mu}L^2$ norm equal to $1$ equals
%\begin{lm} \label{obnormalization} For each pair $(j,\ell)\in \{1,...,J\}\times \Z,$ there is a unique $r^{-\mu}L^2$-normalized solution $\Psi_{j\ell}^\dag(r)$ to 
%\eqref{BesselODE} satisfying $\smash{\Psi_{j\ell}^\dag|_{r=1}=0},$ given by
\[
\Psi^\dag_{j\ell}(r):=\left[ A_{j\ell} \, r^{\frac{2-n}{2}} I_{\nu_j}(|\ell| r) \ + \ B_{j\ell} \, r^{\tfrac{2-n}{2}} K_{\nu_j}(|\ell| r) \right], \quad  \quad 
\nu_j =  \sqrt{ \frac{(n-2)^2}{4} + \lambda_j},
\]
where the coefficients $A_{j,\ell}$ and $B_{j,\ell}$ satisfy
\[ 
A_{j\ell} = C_j |\ell|^{\mu+1} e^{-2|\ell|}(1+O(|\ell|)), \qquad B_{j\ell} =D_j |\ell|^{\mu+1}(1+O(|\ell|)).
\]
for some constants $C_j, D_j\in \R$ independent of $\ell.$  
\end{lm}
\begin{proof}
The homogeneous solutions of \eqref{BesselODE} are linear combinations of the functions 
\[
r^{(2-n)/2} I_{\nu_j}( |\ell|r), \ \ \mbox{and}\quad  r^{(2-n)/2}K_{\nu_j} (|\ell| r).
\]
We must choose a linear combination of these which vanishes at $r=1$ and satisfies the norm condition.

Recall from e.g. \cite[Eq. (1.12)]{nemes2017error} the asymptotic bounds of these modified Bessel functions for large arguments:
\begin{equation}
\begin{split}
& I_\nu(z) \cong e^z/\sqrt{2\pi z}, \qquad K_\nu(z) \cong  e^{-z}\sqrt{\pi/2z},\ \ z \geq 1 \\
& I_\nu(z) \cong   z^\nu,\qquad \hspace{.3cm}  \qquad K_\nu(z) \cong z^{-\nu},  \hspace{1.05cm}\ \ z \leq 1,
\end{split}
%I_\nu(z) \sim \frac{e^z}{\sqrt{2\pi z}}\sum_{n=0}^\infty (-1)^n\frac{a_n(\nu)}{z^n}, \qquad K_\nu(z) \sim \frac{\sqrt{\pi}e^{-z}}{\sqrt{2 z}}\sum_{n=0}^\infty \frac{a_n(\nu)}{z^n},
\label{IKas}
\end{equation}
where here and below, we write $F \cong G$ to mean that $C^{-1} F \leq G \leq C F$ for some constant $C > 0$ which is uniform over the finite number of
values of $\nu_j$ needed here.

When $\ell \neq 0$, then by \eqref{IKas},
\begin{equation}
\begin{split}
 \|  r^{(2-n)/2} I_{\nu_j}(|\ell| r) e^{i\ell t}\phi_j\|_{r^{-\mu} L^2}  &= \sqrt{2\pi} \left(\int_0^1 r^{2-n}|I_{\nu_j}(|\ell| r)|^2 r^{2\mu}r^{n-1} \, dr\right)^{1/2} \nonumber\\
& \cong |\ell|^{-\mu-1} \left( 1 + \int_1^{|\ell|} s^{2\mu} e^{2s}\, ds\right)^{1/2}  \cong |\ell|^{-1} e^{|\ell|},
\end{split}
\label{Inunorm}
\end{equation}
since $\mu > 0$, and similarly,
\begin{equation}
\begin{split}
& \|  r^{(2-n)/2} K_{\nu_j}(|\ell| r) e^{i\ell t}\phi_j \|_{r^{-\mu} L^2} = \sqrt{2\pi} \left(\int_0^1 r^{2-n}|K_{\nu_j}(|\ell| r)|^2 r^{2\mu}r^{n-1} \, dr\right)^{1/2} \\
&  = \frac{\sqrt{2\pi} }{|\ell|^{\mu+1}}\left(\int_0^{|\ell|} s^{2\mu+1} |K_{\nu_j}(s)|^2 \,ds\right)^{1/2}  \cong |\ell|^{-\mu-1}
\end{split}
\label{Knunorm}. 
\end{equation}
We also have that
\begin{equation}
\left.r^{(2-n)/2}I_\nu(|\ell|r) \right|_{r=1} \cong |\ell|^{-1/2} e^{|\ell|}, \qquad \left. r^{(2-n)/2}K_\nu( |\ell|r)\right|_{r=1} \cong |\ell|^{-1/2} e^{-|\ell|},
\label{IKnorm1}
\end{equation}
hence there exist coefficients $a_{j\ell}$, $b_{j \ell}$ with $|a_{j\ell}|  \cong |\ell|^{1/2} e^{-|\ell|}$, $|b_{j\ell}| \cong |\ell|^{1/2} e^{|\ell|}$ such that
\[
\left. a_{j\ell} r^{(2-n)/2} I_{\nu}( |\ell|r) + b_{j\ell} r^{(2-n)/2} K_\nu(|\ell|r) \right|_{r=1} = 0.
\]
Finally, we must divide this linear combination by a factor $D_{j\ell}$ to have $r^{-\mu}L^2$ norm equal to $1$.  Since
\[
\| a_{j\ell} r^{(2-n)/2} I_\nu(|\ell|r) \|_{r^{-\mu}L^2} \cong |\ell|^{-1/2}, \qquad \|b_{j\ell} r^{(2-n)/2} K_{\nu}( |\ell|r)\|_{r^{-\mu}L^2} \cong |\ell|^{-\mu-1/2} e^{|\ell|},
\]
it is clear that $D_{j\ell} \cong |\ell|^{-\mu-1/2} e^{|\ell|}$, and hence finally
\[
|A_{j\ell}| \cong  |\ell|^{\mu+1} e^{-2|\ell|}, \qquad |B_{j\ell}| \cong |\ell|^{\mu+1}. 
\] 
\end{proof}

To simplify notation, in terms of the same notation as above, introduce the {\bf renormalized (modified) Bessel functions}: 
\[
\widehat{I}_{\nu_j}(|\ell| r) = {A_{j\ell}} r^{2\mu}r^{\tfrac{2-n}{2}} I_{\nu_j}(|\ell| r), \qquad   %\label{Ihatder}\\ 
\widehat{K}_{\nu_j}(|\ell| r) := {B_{j\ell}} r^{2\mu}r^{\tfrac{2-n}{2}} K_{\nu_j}(|\ell| r) % \label{Khatdef}.
\]
Since $r^{2\mu}: r^{-\mu}L^2\to r^\mu L^2$ is an isometry, the $r^\mu L^2$-norms of these functions are given by the same expressions \eqref{Inunorm}, \eqref{Knunorm}, respectively. 
The previous lemma yields the following more concrete realization of Proposition \ref{btraceregularity}. 

\begin{defn}\label{obbasisdef} Define the {\bf obstruction basis} of $\text{\bf Ob}_0$ by %as the $r^{\mu}L^2$-orthonormal basis of given by 
\be 
\Psi_{j\ell}:=r^{2\mu}\Psi_{j\ell}^\dag(r)\cdot \phi_j e^{i\ell t},\label{obstructionbasisdef}
\ee
Observe that this basis is $r^\mu L^2$ orthonormal. Similarly, define the {\bf obstruction map} as the parametrization of $\mathbf{Ob}_0$ % by basis vectors given by
\begin{equation}
\text{ob}_0:  \  L^2(S^1; \mathcal C_0)\to \text{\bf Ob}_0,  \qquad \text{ob}_0( \eta)= \sum_{j=1}^{J} \sum_{\ell \in \Z}  \eta_{j\ell} \Psi_{j\ell},
\label{ob0def}
\end{equation}
where $\eta=(\eta_1,...,\eta_J)$ and $\eta_{j\ell} = \br \eta_j(t) , e^{i\ell t}\kt_{L^2(S^1)}.$
\end{defn}

The following is then immediate.
\begin{cor}  \label{firstob} The following hold
\begin{itemize}
\item[(A)]  $\mathrm{ob}_0$ as defined in \eqref{ob0def} is a bounded linear isomorphism, 
\item[(B)]  the inverse of $\mathrm{ob}_0$ is given by 
\[ 
\mathrm{ob}_0^{-1}(\Psi)= \sum_{j=1}^{J} \sum_{\ell \in \Z} \br \Psi, \Psi_{j\ell}\kt_{r^\mu L^2(X_0)}  \ \cdot \  e^{i\ell t}\phi_j, \ \ \Psi \in \text{\bf Ob}_0.
\]
\item[(C)] $\text{ob}_0$ respects regularity in the sense that 
\vspace{-.5cm}
\begin{center}
\be 
\tikzset{node distance=3.4cm, auto}
\begin{tikzpicture}
\node(C) {$ L^{2}(S^1; \mathcal C_0)$};
\node(C')[right of=C, yshift=0cm]{$\text{\bf Ob}_0 $};
\node(D)[below of=C, yshift=1.7cm]{$H^s(S^1; \mathcal C_0)$};
\node(D')[below of=C', yshift=1.7cm]{$\text{\bf Ob}_0 \cap r^{\mu}H^s_b$};
\draw[->] (C) to node {$\text{ob}_{0}$} (C');
\draw[->] (D) to node {$\text{ob}_{0} $} (D');
\draw[<-] [swap](C) to node {$ $} (D);
\draw[->] [swap](D') to node {$ $} (C');
%\draw[->] (C) to node {$ L_\xi^\star $} (C');
%\draw[->,>=stealth] (A) edge[bend right=20]node[below]{$L_0\circ \underline P_0 $} (C');
\end{tikzpicture}
\label{obregularity}
\ee
\end{center}
\noindent commutes, where the vertical arrows are the inclusions. 
\end{itemize}
\end{cor}

\begin{proof} (A) is immediate, and (B) is simply a restatement of orthogonality. For (C), elliptic edge regularity shows that $\Psi \in \text{\bf Ob}_0 \Rightarrow \Psi \in r^{\mu}H^s_e$ 
for all $s\geq 0,$ so it suffices to check $\nabla_t^s \Psi \in r^\mu L^2$ iff $\Psi=\text{ob}(\eta)$ for $\eta \in H^s$.  This is immediate from \eqref{obstructionbasisdef}.    
\end{proof}

\subsection{Unexpectedly Good Boundedness}
\label{smoothingprojectionsection}
Lemma \ref {secondorderinvertibility} show that $\Pi_0: r^\mu H^s_b \to r^\mu H^s_b$ is bounded, i.e. 
\[
\Psi \in r^\mu H^s_b \ \  \ \Rightarrow \ \  \ \Pi_0 \Psi \in r^\mu H^s_b . 
\]
More abstractly, this follows because the projector is an edge pseudodifferential operators of order $0$ and all such operators are bounded on these spaces. 
Since $\text{ob}_0^{-1}$ is also order $0$ by Corollary \ref{firstob}(C), the composition $\text{ob}_0^{-1}\circ \Pi_0$ inherits the same boundedness.

In the upcoming sections, we will encounter a number of operators that act as `delocalizing extensions` of functions on the edge. These are operators 
\begin{equation}
\begin{split}
H^{s}(S^1;\R)&\to r^\mu H^{s}_b(X_0;NX_0) \\
f(t)&\mapsto \Psi_f:=(\nabla_t^k f(t)) r^p \phi_j
\end{split}
\label{weakasympotitic}
\end{equation}
for $k\in \Z.$  These are not pseudodifferential operators in the standard sense, but rather $C(\Sigma)$-parameterized families of pseudodifferential operators on the edge $S^1,$ whose image resembles the asymptotic expansions \eqref{weakexpansion1}. Acting on sections $\Psi_f,$ however, the projection $\Pi_0\Psi_f$ can exhibit unexpectedly good tangential, i.e. $b$-regularity. In other words, this projection lies in $r^\mu H^{s+r}_b$ for some $r>0.$ Observe that \eqref{weakasympotitic} is bounded provided $p> \mu -\tfrac{n}{2}.$ As we now explain, the defect in vanshing order $p-(\mu-\tfrac{n}{2})$ can be converted into extra $b$-regularity in $\Pi_0\Psi_f.$

\begin{defn} \label{effectivescaling}
Let $f,\Psi_f$ be as in \eqref{weakasympotitic}. Then the operator $f\mapsto \Psi_f$ is said to have {\bf effective order} 
\[
\boxed{m= k- \left(p + \tfrac{n}{2}-\mu\right)}.
\]
\end{defn}
 
More generally, the same applies to operators of the form 
\be 
f \mapsto \Psi_f:=\sum_{i=1}^M A^i(t,r,x) (\nabla_t^{k_i} f(t)) r^{p_i}\phi_{j_i},\label{effectiveordersum}
\ee
for integers $p_i, k_i,$ and $j_i\leq J,$ where  $A^i\in C^\infty_b(X_0;\text{End}(NX_0)).$ In this case, the effective order is $m=\max_i (k_i + p_i +(\tfrac{n}{2}-\mu) ).$ The terminology indicates the order of the operator {\it as a function of $f$,} and is thus independent of the ambient regularity $f\in H^s.$ We emphasize that this definition does not require {\it a priori} that $k \leq s;$ indeed, provided it is paired with a sufficiently high order of vanishing, the projection $\Pi_0(\Psi_f)$ has high regularity even if $\nabla_t^k f(t)$ itself is a very rough distribution. The orthogonal component $(1-\Pi_0)\Psi_f,$ however, is still beholden to the ambient regularity, meaning \eqref{weakasympotitic} lies in $r^\mu L^2$ only if both $k\leq s$ and $p>\mu-\tfrac{n}{2}.$

\begin{lm}  \label{projectionsmoothing}Suppose that $\Psi_f$ is a term of effective order $m$, where $f(t)\in H^s.$ Then $  \text{ob}_0^{-1}\circ \Pi_0(\Psi) \in H^{s-m}(S^1;\mathcal C_0),$ and  

\[\| \text{ob}_0^{-1}\circ\Pi_0(\Psi)\|_{H^{s-m}(S^1; \mathcal C_0)} \leq C \sum_{i=1}^M\|\psi_i\|_{H^{s_i}}.\]

\noindent In particular, there is a gain in tangential regularity of order $\alpha+ \tfrac{n}{2}-\mu.$ Here, the constant can depend on up to the $C^{s_i+2}_b$-norm of $A^i.$
\end{lm}

\begin{proof} By Corollary \ref{firstob}, it suffices to show that, for each fixed $j \leq J$, the inner products 
 \[
\left\{\br \Psi, \Psi_{j\ell}\kt_{r^\mu L^2(X_0)}\right\}_{\ell = -\infty}^\infty
\]
lies in a weighted sequence space $l^{m,2}(\Z)$. 

Write $\Psi$ as in \eqref{weakasympotitic}, and decompose $f(t)=\sum_{\ell} f_\ell (1+|\ell|^2)^{-s/2}.$  By virtue of \eqref{obstructionbasisdef} and Lemma \ref{obnormalization},
then after integrating over $\Sigma$, 
%, to leading order in $\ell$ we have: 
\begin{equation*}
\begin{split}
\Big | \br \Psi, \Psi_{j\ell}\kt_{r^\mu L^2}\Big |&=\Big | \int_0^1 \int_{S^1} \br \nabla_t^k f(t) r^{\alpha} \ , \  r^{2\mu}\Psi_{j\ell}^\dag(r)  e^{i\ell t} \kt \ r^{-2\mu}\, r^{n-1}dr dt \Big |\\
&\leq  |f_\ell| \,  |\ell|^{k-s + \mu+1} \int_0^1 r^{n/2 + \alpha} ( e^{-2|\ell|}I_{\nu_j}(|\ell|r) + K_{\nu_j}(|\ell|r)\, dr \\
& \leq |f_\ell| |\ell|^{k-s + \mu+1 - n/2 - \alpha  - 1} \int_0^{|\ell|} ( e^{-|\ell|}/|\ell|^{1/2} + K_{\nu_j}(\rho)) \, d\rho,
%\frac{|\ell|^{\mu+1}}{|\ell|^{s-k}} \int_0^1 r^\alpha  r^{\tfrac{2-n}{2}} K_{\nu_j}(|\ell| r ) r^{n-1}dr \\
%&\leq |f_\ell|  \frac{|\ell|^{\mu+1}}{|\ell|^{s -k+\alpha + \tfrac{n}{2}+1}} \int_0^{|\ell|} \rho^{\alpha+\tfrac{n}{2}}K_{\nu_j}(\rho) d\rho \ \leq \ \frac{|f_\ell|}{|\ell|^{s-k+\alpha+(\tfrac{n}{2}-\mu)}}.
\end{split}
\end{equation*}
where we estimated $I_{\nu}(|\ell|r) \leq I_{\nu}(|\ell|) \leq C e^{|\ell|}/|\ell|^{1/2}$. The final integral is finite and uniformly bounded in $|\ell|$, so we are left with 
$|f_\ell| \, |\ell|^{ -s  + m}$, where $m = k - (\alpha - (\mu-n/2)$, as desired. This proves the assertion when $\Psi$ is a monomial. The general result in the case of \eqref{effectiveordersum} is similar, using $C^{s}_b$ bounds on the coefficients $A^i.$ 
\end{proof}

%% file: sections/sec_nash_moser.tex
%%
%% sec_nash_moser.tex
%%

\section{Review of Nash-Moser Theory}
\label{FrechetSection}

As explained in the introduction, the proof of Theorem \ref{maina} uses Nash-Moser theory (cf. Remark \ref{lossrem}). This section contains a concise review of Nash-Moser theory, including the version of the Implicit Function Theorem 
used below. The remaining sections of this paper verify the hypotheses of this theorem in our setting, which thus proves Theorem \ref{maina}.  
Our exposition follows \cite{ParkerDeformations, SecchiNashMoser}; we refer to \cite{HamiltonNashMoser, HormanderNashMoser,StraymondNashMoser,lojasiewicz1979inverse}
for more detailed treatments of this general theory.

\subsection{Tame Fr\'echet Spaces}
\label{IFTSubsection}

Let $\mathbb X^\infty$ be a  graded Fr\'echet spaces, realized as the intersection of a family of Banach spaces:
\be 
\mathbb X^\infty:= \bigcap_{s\geq 0} \mathbb X^s.  \label{frechetnotation}
\ee
We assume that the norms on these increase monotonically,
$$
\|x\|_0\leq \|x\|_{1}\leq\ldots \leq \|x\|_{s},\ \ x \in \mathbb X^\infty.
$$ 
  The spaces $\mathbb X^\infty$ is endowed with the initial topology of the families 
of inclusions $\mathbb X^\infty \hookrightarrow \mathbb X^s$, i.e., the coursest topology with
respect to which these inclusion maps are continuous for all $s.$ 

% which means simply that 
%a set $\calU$ is open if and only if for each point $x\in \calU$ there exists an $s\geq 0$ and $r>0$ such that the ball 
%$\{ x \ | \ \|x\|_m < r \}\subset \calU$ measured in the $s$-norm is contained in $\calU$.   

\begin{defn} \label{tamedef} A Fr\'echet  space $\mathbb X^\infty$ is {\bf tame} if it is endowed with a family of smoothing operators 
\[
S_\e: \mathbb X^\infty \to \mathbb X^\infty ,\ \ \e \in (0,1],
\]
such that for each $m, n$, there exist constants $C(m,n)$ such that
\begin{enumerate}
\item[(i)] $\|S_\e x\|_{n}\leq C_{mn} \ \e^{m-n}\|x\|_m $ \ \ \  \ \ \ \ \ \ for $n\geq m$ \ \ \  and  \ \ \  $\|S_\e x\|_m \leq C_{mn}\|x\|_n$ for $n\leq m$. 
\item[(ii)] $\|S_\e x -x\|_{m}\leq C_{mn} \e^{n-m} \|x\|_n$   \ \ \ \ \  for $n\geq m.$ 
\item[(iii)]$ \|\d{}{\e} S_\e x\|_n\leq C_{mn} \e^{m-n-1} \|x\|_{m}$  \ \ \ \ \ for $m,n\geq 0.$ 
\end{enumerate}
\end{defn}
\medskip 

Given such a family of smoothing operators, one obtains interpolation inequalities: for any $s_1<s< s_2$, there exist constants $C=C(s, s_1, s_2)$ 
\be 
\|x\|_{s}\leq C \  \|x\|_{s_1}^\alpha\|x\|_{s_2}^{1-\alpha}
\ee
holds, with $\alpha=\tfrac{s_2-s}{s_2-s_1}$.  The proof is omitted, but we note that in our setting, these interpolation inequalities 
follow from Proposition \ref{multiplications}.  If a Fr\'echet space has a presentation such that $\mathbb X^s=\mathfrak W H^s$ for a weight $\mathfrak W$ and $H^s$ a standard Sobolev space (or a similarly familiar family of spaces with known smoothing operators), then $\mathbb X^\infty$ is automatically tame under the smoothing operators $S_\epsilon'=\mathfrak W^{-1}\circ S_\epsilon\circ \mathfrak W$ defined by conjugation. All Fr\'echet spaces we consider are tame by this criterion, and we leave it to the reader to check explicitly in the various cases. 

\medskip

Tame maps are those that respect the Fr\'echet structure in the following way:   
\begin{defn} \label{smoothtamedef}
For tame Fr\'echet spaces $\mathbb X^\infty, \mathbb Y^\infty$, a map
\[
\mathbb Q: \mathbb X^\infty \supset \calU\to \mathbb Y^\infty
\] 
where $0 \in \calU$ is said to be {\bf tame} if there exists some fixed $d\in \N$  such that
\begin{equation}\label{7.firsttameinequality}
\|\mathbb Q(x)\|_{s}\leq C_s  \ ( 1 + \|x\|_{s+d})
\end{equation}
for all sufficiently large  $s.$ A such a map is a {\bf smooth tame map} if it is smooth, and its derivatives of all orders are tame. 
\end{defn}
The key point is that there is a fixed drop in `regularity' when applying $\mathbb Q$, independent of $s.$
Tame Fr\'echet manifolds and smooth tame maps between them are defined in the standard way using charts modeled on tame Fr\'echet spaces,
cf.\ \cite{HamiltonNashMoser}. Provided the Banach spaces at finite regularity $s$ have an appropriate algebra structure for $s$ sufficiently large, the following useful criterion ensures tameness. 
%Tame Fr\'echet manifolds with smooth tame maps between them provide the natural category which, as the name indicates, 
%tames the potentially wild behavior possible in arbitrary Fr\'echet spaces; enough so that a version of the Implicit Function Theorem holds.
%The essential restriction of tameness is that each norm depends only on a fixed finite number $p$ of norms larger than it. A map with an estimate of the form \eqref{7.firsttameinequality} where $p=2s$ would, for example, not be tame. 

\begin{lm}\cite[Sec. II.2]{HamiltonNashMoser}  \label{tamenesscriterion} If $\mathbb Q: \calU\to \mathbb Y^\infty$ is a finite nesting of sums, products, and 
compositions of tame maps, then $\mathbb Q$ is tame. 
\end{lm}
We note, in particular, that pointwise analytic operations such as $f\mapsto \sqrt{1+f}$ are tame maps (provided the range of $1+f$ is uniformly bounded away from the origin). 
The lemma  makes it easy to ensure the tameness of many maps constructed naturally from geometric operations. It implies, for instance, the tameness of the mean curvature operator, 
which is built from the components of the metric via from pointwise sums, products and composition with pointwise analytic operations. 

\begin{notation} \label{Frechetnotation} By and large, the Fr\'echet spaces we employ are given as the intersection over the family of Banach spaces with hybrid norms from Definition \ref{hybridspacesdef}.  To simplify notation, we introduce the following abbreviations for the most frequently appearing version of these spaces. Denote 
\begin{align*}
\FrDms(X_0;NX_0)&:= r^{\mu+2} H^{s,2}_{b,e}(X_0; NX_0),\\
\FrRgs(X_0;NX_0)&:= r^{\mu}H^{s}_{b}(X_0; NX_0),\\
\mathbb D^{s}(S^1;V) &:=H^s(S^1; V),\\
\mathbb P^{s}(\del X_0;NX_0) &:=H^s(\del X_0; NX_0),
\end{align*}
\noindent The norms continue to be denoted by \eqref{hybridnorms} and by $\| \cdot \|_{H^s(S^1)}$ for this last space. Finally, $\FrDms_\Dir$ continues to denote the subspace with vanishing Dirichlet boundary conditions. When any ambiguity may arise, or alterations are needed, we refer back to the notation of Definition \ref{hybridspacesdef} for concreteness. 
\end{notation}

\subsection{The Nash-Moser Implicit Function Theorem}
\label{NashMoserSection}

We now state a suitable version of the Nash-Moser Implicit Function Theorem. This requires stronger hypotheses than the classical version
on Banach spaces, namely that the linearization be invertible with suitable control on an open neighborhood of the initial approximation 
rather than just at the initial point. This requirement is necessary because the accelerated convergence achieved from linearizing at each successive 
approximation must overcome the disruption of applying the smoothing operators, which compensate for the loss of regularity. 
The statement of the theorem and hypotheses here are taken directly from \cite[Thm. 7.4]{ParkerDeformations} (see there also for an expanded 
discussion of the reason for the hypotheses). 

Let $B_{r}(x, s)$ denote the open ball of radius $r$ around $x$ measured in the $s$-norm. Let $\mathbb X^\infty, \mathbb Y^\infty, \mathbb P^\infty$ be tame Fr\'echet spaces and  $\calU\subseteq \mathbb X^\infty, \calV\subseteq \mathbb P^\infty$ 
be open neighborhoods of the origin. Consider a map 
\be 
\mathbb Q: \calU \times \calV \lre \mathbb Y^\infty \label{tameFrechetparameterized}
\ee
satisfying $\mathbb Q(0,0)=0$ such that the following conditions hold. 
\begin{hyp1*}
\label{NM1}
$\mathbb Q$ is a smooth tame map. \end{hyp1*}

\begin{hyp2*}
\label{NM2}
There exists $\delta>0$ and $s_0\in \N$ such that for $p\in \calV$ and $u\in \calU_0= B_{\delta }(0, s_0)$, the partial linearization 
with respect to $u$ 
\[
(\text{d}\mathbb Q)_u: \mathbb  X^\infty \to \mathbb Y^\infty
\] is an isomorphism, with inverse $\mathbb G_u$ satisfying $(\text{d}\mathbb Q)_u \circ \mathbb G_u=\text{Id}_{\mathbb Y}$ and $\mathbb G_u \circ (\text{d}\mathbb Q)_u =\text{Id}_{\mathbb X}.$ 
\end{hyp2*}

\begin{hyp3*}
\label{NM3}
The inverse $\mathbb G_u$ is smooth and tame, i.e., 
%for $u\in \mathcal U_0$ and $\text{d}_x\mathbb F$ as in Hypothesis {(II)}, 
there exist $d,d'\in \N$ and constants $C_m$ such that the solution $\ph=\mathbb G_u(y)$ of $\text{d}_u\mathbb Q(\ph)=y$ 
obeys
\be 
\|\ph\|_{s}\leq C_s \ \Big( \|\text{d}_u\mathbb Q (\ph)\|_{s+d} +\|\text{d}_u\mathbb Q(\ph)\|_{s_0}\cdot  \|(u,p)\|_{s+d'}\Big),
\ee
\noindent for $s\geq s_0$ and $u \in \calU_0$.
\end{hyp3*}

\begin{thm}\label{nashmoser}{\bf (Nash-Moser Implicit Function Theorem)}  Let $\mathbb Q:\calU\times \calV\to \mathbb Y^\infty$ be a smooth tame map
such that Hypotheses {{(I)}}--{ {(III)}} hold and $\mathbb Q(0,0)=0$.  Then, there exists an open neighborhood $\mathcal V_1\subseteq \mathcal V$ such that for each $p\in \mathcal V_1$, there exists a unique solution $u(p)$ of 
\be 
\mathbb Q(u(p), p)=0.
\ee
Moreover, the solutions $u(p)$ is a smooth tame function of $p\in \mathcal V_1$, i.e., $\mathbb Q^{-1}(0) \subseteq  \calU\times \calV_1$ is a tame Fr\'echet submanifold 
at $(0,0)$ that is a graph over $\mathcal V_1.$ 
\qed
\end{thm}

There is an extension to the case where $\mathcal U,\mathcal V$ are tame Fr\'echet manifolds $\mathbb Y^\infty \to \calU\times \calV$ is a tame Fr\'echet vector bundle, and 
$\mathbb Q:  \mathcal U\times \mathcal V \to \mathbb Y^\infty$ is a smooth tame section. Since the statement is local, this is immediate from applying the above version in suitable charts and trivializations. We refer once more to \cite{HamiltonNashMoser} for complete details of the setup.

%% file: sections/sec_perturbations.tex
\subsection{Conventions and Weights}
\label{weightconventions}
The theory developed in the previous sections holds for an arbitrary minimal $\Sigma\subset S^{N-1}.$ Henceforth, we restrict to the case that $\Sigma$ is spectrally rigid in the sense of Definition \ref{spectralassumptions}, and fix a single weight $\mu$ as follows. 

Spectral rigidity implies $\text{Spec}(L_\Sigma)=\{0, n-1, n-1 + z,...\}$ for some $z>0.$ Let $\epsilon>0$ be such that 

\[\sqrt{\frac{(n-2)^2}{4} + (n-1) + z}\ = \ \sqrt{\frac{n^2}{4}+z} \ =:  \ \frac{n}{2}+2\epsilon.\]

\noindent In the case of the Veronese embeddings following Definition \ref{spectralassumptions} \cite{rauchenberger2026index} calculates that $z=\frac{8}{3}$ in all cases, which yields $2\epsilon= 0.588, 0.3204, 0.165$ in the three examples of \eqref{embeddingsKPN} respectively. We then define $\mu$ so that 

\[\mu+1:=\frac{n}{2}+ \epsilon \hspace{1cm}\Rightarrow \hspace{1cm} \mu +2 -\frac{n}{2} \notin \text{Ind}(L_{C(\Sigma)}).\]

\noindent This choice ensures that the following hold simultaneously: 
\begin{enumerate}
\item The (left) semi-Fredholm elliptic theory of $L_0$, i.e. the conclusions of Corollary \ref{higherorderelliptic}(B).
\medskip
\item Sections with constant pointwise norm at $r=0$ have finite $r^\mu H^{s}$-norm because $\mu< \tfrac{n}{2}.$ 
\medskip
\item $r^{\mu+1} H^{s}_b$ is an algebra for $s$ sufficiently large.
\medskip 

\item The embedding $r^{\mu+1}H^{s}_b(X_0; NX_0) \hookrightarrow r^{\epsilon}  C^{s-f(n)}_b(X_0; NX_0)$ holds, for a dimensional constant $f(n),$ and similarly for $\mu+2$ (resp $r^{1+\epsilon}$).  
\end{enumerate}

\medskip

\noindent (3) and (4) follow from Lemma \ref{standardHolder} and Proposition \ref{multiplications} respectively. (4) also implies that $\Phi \in r^{\mu+2}H^s_b$ implies $\Phi \in C^{1,\alpha}$ for some $\alpha>0$ (here the standard, not $b$-version of the H\"older space). The latter implies that the perturbation by $\Phi$ does not affect the tangent cone model at the singularity.

 Finally, by definition, spectral rigidity implies that the restriction map $P_{NC(\Sigma)}: \R^N \oplus \, \mathfrak{so}(N) \to \mathcal B_0$ is surjective. The first factor is injective, and we let $\R^K\subseteq \mathfrak{so}(N)$ be a subspace such that the restriction 
 \be P_{NC(\Sigma)}: \R^N \oplus \R^K\simeq \mathcal B_0\label{B0iso}\ee given by the low eigenmode content of the projection to the normal bundle is an isomorphism. That is to say, we fix a choice of infinitessimal ambient rotations that restrict to a basis of the $(n-1)$ eigenspace of $L_\Sigma.$
\bigskip

\section{Deformations of Singular Submanifolds}
\label{deformationssection}
The goal of this section is to construct an open tame Fr\'echet manifold $\FrDmmu$ consisting of submanifolds with edge singularities near $\Gamma_0,$ as well as a tame Fr\'echet vector bundle $\FrRgmu\to \FrDmmu$ such that the mean curvature becomes a smooth tame section 

\be 
\tikzset{node distance=3.0cm, auto}
\begin{tikzpicture}
\node(D)[yshift=-1.5cm,xshift=1.5cm]{$  \FrDmmu.  $};
\node(Y)[yshift=-1.5cm][above of=D]{$  \FrRgmu $};
\node(Q)[yshift=-2.3cm][xshift=-.7cm][above of=D]{$  {\mathbb Q} $};
%\node(Y'')[xshift=-2.2cm][yshift=-.1cm][right of=D]{$=$};
%\node(Y')[xshift=0.6cm, yshift=-.2cm][right of=D]{$ \left\{ X_{\Phi,\xi} \   \Big | \  \  \begin{matrix} \text{reg}(X_{\Phi,\xi})=\text{graph}(\Phi)  \\ \text{sing}(X_{\Phi,\xi})=\text{graph}(\xi) ) \end{matrix}\right\} $};
%\node(D')[xshift=-.8cm, yshift=0.9cm]{$\Theta $};
\draw[->](Y) to node {$\pi_{\mathbb Y}$} (D);
\draw[->][swap] (D) edge[bend left=45]node[below] {$$} (Y);
%\draw[->] (D) to node {$\widetilde {\slashed {\mathbb D}}$} (E);
\end{tikzpicture}
\label{vectorbundles.}
\ee

\noindent The total space of $\FrDmmu$ is itself the total space of a Fr\'echet fiber bundle over the space of parameterized Euclidean isometries. We begin by constructing this fiber bundle in detail. 

\subsection{Deformations by Parameterized Euclidean Isometries} \label{diffssection1} In this first subsection we define the base of the fibration giving $\FrDmmu,$ with an accompanying set of submanifolds that yield a chosen reference section. 

For $\delta >0$ and an integer regularity $s_0 \geq 1$ to be specified later, let 
\be 
\mathcal U_0 := \Big \{(\xi(t),\alpha(t)) \in C^\infty(S^1; \mathcal B_0) \ \Big | \ \|\xi\|_{s_0}<\delta \Big\},\label{U0def}
\ee
where $\xi(t) \in \R^N$ and $\alpha(t) \in \R^K\subseteq \frak{so}(N).$ Each $\xi \in \mathcal U_0$ gives rise to the corresponding 
families of translations and rotations 
\[
(t,y)\mapsto y+\xi(t) \qquad  (t,y)\mapsto A(t)y=\text{Exp}(\alpha(t)) y
\]
of $\R^{N},$ where the exponential map is that of $SO(N).$

We now define a family of diffeomorphisms (onto their respective images) 
\begin{equation}
\begin{split}
\Fdiffeos :   \mathcal U_0& \to \text{Diff}(S^1 \times B^N_2) \\
(\xi,\alpha) &\mapsto \Fdiffeos(\xi,\alpha):= F_{\xi,\alpha}.
\label{mathbbFdef}
\end{split}
\end{equation}
These have the property that $F_{\xi,\alpha}[\Gamma_0]=\Gamma_\xi$ where $\Gamma(t)=(t,\xi(t))$ is the graph of $\xi.$ This 
collection is later used to define a chart identifying a set of of edge submanifolds with $\mathcal U_0.$ Given $(\xi,\alpha)
\in \mathcal U_0,$ these diffeomorphisms are defined as follows: writing $N\Gamma_\xi \subseteq T(S^1\times \R^N)$ for 
the normal bundle, choose a periodic orthonormal frame $e_i(t)\subseteq N\Gamma_\xi$ for $i=1,...,N$.   Combined with 
the unit tangent vector $e_\Gamma$ to $\Gamma_\xi$, these form an orthonormal frame of $T(S^1\times \R^N)$, and we
may assume that this frame depends smoothly and tamely on $\xi$.  Define 
\be  
F_{\xi,\alpha} (t,y):=  \Gamma(t) \ + \  \sum_{a,b=1}^N A_{ab}(t)y_b e_a(t) \label{Fxidef}
\ee
where $y_a$ are the coordinates on $\R^N.$ We then have the following calculation of the pullback metrics, which is given in terms of the quantities 
\begin{equation*}
\begin{split}
  a_\xi(t)&:=|\Gamma'_\xi(t)|=\sqrt{1+|\xi'(t)|^2}\\
  \omega_{\alpha\beta}(t)&:=\br \nabla_t e_\alpha, e_\beta\kt\\
  \Omega_A(t)&:=  -A^T(t) \omega(t) A(t)+ A(t)^T A'(t)\\
  k(t)&:= \frac{1}{a_\xi(t)} \left[ \sum_{b=1}^N\br   \nabla_t  e_\Gamma  \ , \ f_b \kt  \  dy_b\right]
\end{split}
\end{equation*}
where $f_b(t)=\sum_{a}e_a(t)A_{ab}(t).$ Recall $g_0=g_{S^1\times \R^N}$ is the product metric.  

\begin{lm} \label{Fxidiffeos} Restricted to $|y|\leq 2$, \eqref{Fxidef} defines a family of smooth, injective immersions \eqref{mathbbFdef}, 
i.e., $F_{\xi,\alpha}$ is a diffeomorphism onto its image. Moreover, the pullback of the product metric $g_{\xi,\alpha}:=F_{\xi,\alpha}^* g_0$ 
is given by 
\begin{align} F^*_{\xi,\alpha}(g_{0})&= \Big(  a^2_\xi(t) (1- \br k(t),y\kt)^2 dt^2  \ + \ |\Omega_A(t) y|^2 \Big)dt^2 \ + \ 2\sum_{a=1}^N (\Omega_A(t)y)_a   dt dy_a   \ + \ \sum_{b=1}^N d y_b^2 \label{gxiformula} \end{align}
where $a_\xi(t), \omega(t), \Omega_A(t), k(t)$ are as above. In particular, all metric coefficients depends smoothly and tamely on $(\xi,\alpha)
\in \mathcal U_0.$
\end{lm}
\begin{proof}
 Write $F=F_{\xi,\alpha}$ for shorthand. That $F$ is an injective immersion onto a normal neighborhood of $\Gamma_\xi$ is immediate from 
the description as the exponential normal map (precomposed with an isometry of the fiber). Moreover, if $\xi$ is $C^2$-small, this neighborhood may be taken to have radius 2. 
Since $H^{s_0}\hookrightarrow C^2$ for $s_0\geq 3$, the first assertion follows by choosing $\delta$ sufficiently small.
  
 Next, let $g$ denote the pullback metric and $f_b(t)=\sum_{a}e_a(t)A_{ab}(t)$ the rotated orthonormal frame.  The tangent vector
$e_\Gamma=a_\xi(t)^{-1} \Gamma'_\xi(t),$ is normalized so that $e_\Gamma, f_1,...,f_N$ are an orthonormal frame of $T(S^1\times \R^{N}).$  
\begin{equation*}
\begin{split}
\p{F_{}}{t}&=(1, \xi'(t)) \ + \ \sum_{a,b=1}^N y_b \Big[e_a'(t)A_{ab}(t) \ + \  e_a(t)A'_{ab}(t)\Big]\\
&= a_\xi(t)e_\Gamma(t) \ + \  \sum_{b,c=1}^N y_b\left \br \sum_{a=1}^N e_a'(t)A_{ab}(t) \ + \   \sum_{a=1}^N e_a(t)A'_{ab}(t)  \ , \ f_c \right \kt f_c\\
& \qquad + \   \sum_{b=1}^N y_b\left \br  \nabla_t f_b  \ , \ e_\Gamma \right \kt e_\Gamma \\ 
& = a_\xi(t)e_\Gamma(t) \ + \  \sum_{b,c=1}^N y_b\left \br \sum_{a=1}^N e_a'(t)A_{ab}(t) \ + \   \sum_{a=1}^N e_a(t)A'_{ab}(t)  \ , \ 
\sum_{d=1}^N A_{dc} e_d \right \kt f_c\\
& \qquad  - \ \ a_\xi(t) \br k(t) , y  \kt \cdot   e_\Gamma \\ 
&=a_\xi(t) \Big ( 1- \br k(t),y\kt \Big)e_\Gamma(t) \ + \  \sum_{b,c=1}^N y_b (\Omega_A(t))_{cb} f_c,
\end{split}
\end{equation*}
where we used $A^T=A^{-1}$, the definitions of $\omega , \Omega_A$, and have integrated by parts on the $e_\Gamma$ term. 

Much more straightforward is the fact that 
\[  
\p{F_{}}{y_b}=f_b(t).
\]
    
Since $e_\Gamma, f_1,...,f_N$ is an orthonormal frame, we can thus compute the components of the pullback metric:
\begin{equation*}
\begin{split}
  g_{tt}&=\left\br \p{F}{t}, \p{F}{t}\right\kt =   \Big |a_\xi(t) \Big ( 1- \br k(t),y\kt \Big)\Big |^2  \ + \  \Big | \Omega_A(t)y \Big |^2\\
  g_{ta}&= \left\br \p{F}{t}, \p{F}{y_{a}}\right\kt= \br  (\Omega_A(t))_{bc}y_b f_c \ , \ f_a \kt=(\Omega_A(t)y)_a \\
  g_{ab}&= \left\br \p{F}{y_a}, \p{F}{y_{b}}\right\kt=\br f_a , f_b\kt=\delta_{ab},
\end{split}
\end{equation*}
which yields the desired expression. 
\end{proof}
  
We now define the reference submanifolds.
\begin{defn} \label{referenceXxi}
For $(\xi,\alpha) \in \mathcal U_0$, the {\bf reference cylindrical submanifolds} are the images
\[
X_{\xi,\alpha}:=F_{\xi,\alpha}[X_0].
\]
By Lemma \ref{Fxidiffeos} and  Eq. \eqref{Fxidef}, $X_{\xi,\alpha}$ is a conical submanifold with singular set $\Gamma_\xi=
F_{\xi,\alpha}[\Gamma_0]$ and boundary $F_{\xi,\alpha}[S^1\times \del B^N]$. The space $X_{\xi,\alpha}$ inherits the restriction of the 
ambient product metric $g_{S^1\times \R^N}$, and we set
\be 
g_{\xi,\alpha}:=F_{\xi,\alpha}^* (g_{0}) \hspace{1.5cm} \frak g_{\xi,\alpha}=g_{\xi,\alpha} \big |_{X_0}
\label{gxidef}
\ee
Thus $F_{\xi,\alpha}:(X_0, \frak g_{\xi,\alpha})\to (X_{\xi,\alpha}, g_{0} |_{X_{\xi,\alpha}})$ is an isometry. 
\end{defn}

\subsection{The Fibration of Submanifolds by Singular Sets}  In this subsection we construct $\FrDmmu$ as a Fr\'echet fiber bundle 
equipped with a projection $\pi_{\mathbb X}: \FrDmmu \to \mathcal U_0.$ This bundle is equipped with a distinguished section 
$\Xi(\xi,\alpha):=(X_0, \gamma_{\xi,\alpha})$ given by the family of reference submanifolds from Definition \ref{referenceXxi}; 
the charts are constructed by exponentiating vector fields normal to these. 
    
Let $s_0,\delta$ be as in \eqref{U0def}, and define 
\[
\mathcal U_1:= \Big \{ \Phi \in  r^{\mu+2}H^{\infty,2}_{b,e}(X_0; NX_0) \ \Big | \ \|\Phi\|_{\mu+2,(s_0,2)}<\delta \Big\} 
\]
We define a second family of diffeomorphisms 
\[
\begin{split} \Gdiffeos:  \mathcal U_1 &\to \text{Diff}(S^1\times B^N; \Gamma_0)\\
\Phi& \mapsto \Gdiffeos(\Phi):= G_\Phi
\end{split}
\]
into the space of $C^1$-diffeomorphisms of the ambient space fixing the product singular set $\Gamma_0=S^1\times \{0\}$. Here, for each 
$\Phi$, the corresponding $G_\Phi$ is defined as follows. Let $(x,w)$ be coordinates on the tubular neighborhood $n(\Sigma)\in S^{N-1}$, where 
$w$ are exponential normal coordinates in the $N\Sigma$ directions. Choose a smooth cut-off function $\chi:[0,\infty)\to \R$ 
supported where $|w|$ is less than the radius of $n(\Sigma)$ and equal to $1$ where $|w|$ is less than half that radius. Then define
\be 
G_{\Phi}(t,r,x,w):=\left(t, r, x, \ w + \chi\left(|w|\right)\Phi\right).\label{GPhidef}
\ee

\begin{lm}  \label{GPhidiffeo} For $\delta$ sufficiently small, there is an $\beta(\mu)>0$ such that 
\[G_{\Phi}: S^1\times B^N\to S^1\times B^N\]
\noindent is a $C^{1,\beta(\mu)}$-diffeomorphism fixing $\Gamma_0.$  This extends smoothly to $(S^1\times B^N) \setminus \Gamma_0,$ 
and fixes the boundary when $\Phi|_{r=1}=0$ has vanishing Dirichlet boundary conditions. 
\end{lm}
\begin{proof} Since $s \gg 1$ is large, $\Phi$ is a $C^3$ map on each sphere $\{(t_0,r_0)\}\times S^{N-1}\subseteq S^1\times B^N$ for each 
fixed $t_0 \in S^1$ and $r_0\in [0,1]$. For $\Phi$ sufficiently small, $\text{d}G_{\Phi}|_{S^{N-1}}$ is close to the identity, thus is a local 
diffeomorphism. It follows from the compactness of the sphere that $\text{d}G_\Phi |_{S^{N-1}}$ is a covering map, and equals the identity 
outside $n(\Sigma)$. This means that it is a degree one covering map, and thus defines a diffeomorphism on each sphere. For $r>0,$ 
smoothness is obvious, as is the claim that it fixes the boundary when $\Phi=0$ there. Finally, that the map is globally $C^{1,\beta}$ 
follows from (4) in Subsection \ref{weightconventions}.  
\end{proof}

\begin{defn}  \label{edgesubmanifolds}
Define the following: 
\begin{enumerate}
\item[(A)] The {\bf geometric family} of submanifolds with edge singularities, defined for $(\xi,\alpha,\Phi)\in U_0 \times \mathcal U_1$ is defined by 
\[X_{\xi,\alpha,\Phi}:= G_{\Phi}[X_{\xi,\alpha}]= F_{\xi,\alpha} \circ G_\Phi[X_0],\]
\noindent equipped with the restriction of the ambient Euclidean metric, where $F_{\xi,\alpha}, G_\Phi$ are as in (\refeq{Fxidef}, \refeq{GPhidef}). Thus  $F_{\xi,\alpha}\circ G_\Phi: (X_0, \frak g_{\xi,\alpha,\Phi})\to (X_{\xi,\alpha,\Phi}, g_0)$ is an isometry, where 
\[
g_{\xi,\alpha,\Phi}:=G_\Phi^* \circ F_{\xi,\alpha}^\star (g_0)=G_\Phi^*(g_{\xi,\alpha}) \hspace{1.5cm} \frak g_{\xi,\alpha,\Phi}:=g_{\xi,\alpha,\Phi} \big|_{ X_0}
\]
denote the pullback metrics. 

\item[(B)] Let $(\xi,\alpha,\Phi) \in \mathcal U_0\times \mathcal U_1,$ and $\pi_{\mathbb X}$ be the projection onto the first factor. The identification 
\begin{align*}
\text{Exp}: \mathcal U_0 \times \mathcal U_1 &\lre \{(X_{\xi,\alpha,\Phi}, g_0)\}\\
(\xi,\alpha,\Phi)&\mapsto(X_{\xi,\alpha,\Phi}, g_0)
\end{align*}
\noindent endows the geometric family on the right with the structure of a (trivial) Banach fiber bundle over $\mathcal U_0.$ It may equivalently be 
viewed as $X_0$ equipped with the family of pullback metrics $g_{\xi,\alpha,\Phi}.$
\end{enumerate}
\end{defn}

We now give a more explicit description for the pullback metric $g_{\xi,\alpha,\Phi}$. Recall that an {\bf edge metric}, is one of the form 
\[
g= g_{tt}dt^2+ g_{tr} dt dr + g_{rr}dr^2 +r  \sum_{\alpha=1}^{n-1} g_{r\alpha} dr  dx_\alpha \ + \  
r\sum_{\alpha=1}^{n-1} g_{t\alpha} dt  dz_\alpha \ + \ r^2 \sum_{\alpha,\beta} g_{\alpha\beta}dz_\alpha dz_\beta 
\]
i.e., it is $r^2$ times a non-degenerate symmetric tensor in the basis $\tfrac{dr}{r}, \tfrac{dt}{r}, dz_j$ dual to the basis of edge vector fields 
from Definition \ref{edgedefs}.

\begin{lm}  \label{pullbackgPhixi}Let $(\xi,\alpha,\Phi)\in \mathcal U_0 \times \mathcal U_1$. Then the pullback metric has the form 
\[
g_{\xi,\alpha,\Phi}= g_{\xi,\alpha} \ + \   h_{\xi,\alpha,\Phi} 
\]
with $g_{\xi,\alpha}$ as in \eqref{gxidef} and $h_{\xi,\alpha,\Phi}$ a $O(r^{1+\epsilon})$ symmetric (edge) tensor field depending smoothly 
and tamely on $(\xi,\alpha,\Phi).$ More specifically, it has components formed by finite sums, products, and compositions with global analytic functions of $\Phi, \nabla^e \Phi,$ and smooth functions $h_i(x,\Phi)$ as in Lemma \ref{composition}.
\end{lm}
\begin{proof} By definition, $g_{\xi,\alpha,\Phi}=(F_{\xi,\alpha}\circ G_{\Phi})^*g_0=G_{\Phi}^* g_{\xi,\alpha}.$ Writing $z=(x,w)$ for the coordinates 
on $n(\Sigma)$, then
\be 
g_{S^{N-1}}=dx^2 + dw^2 + \mathrm k(x,w)\label{nSigmametric}
\ee
where $\mathrm k=O(|w|)$ is a smooth symmetric tensor vanishing along $\Sigma.$

Since $G_{\Phi}$ preserves sphere of constant $(t,r),$
\begin{align*}
G_{\Phi}^*(dt)&=dt, \qquad G_{\Phi}^*(dr)=dr  \\
G^*_{\Phi}(dz)&= dz +  \del_t Z_\Phi dt + \del_r Z_\Phi dt + \del_z Z_\Phi dz
\end{align*}
\noindent where $Z_{\Phi}(x,w)=(x,w + \chi(|w|)\Phi).$ Therefore, 
\[
G^*_{\Phi}(r dz)= (\text{Id} + \del_z Z_{\Phi}) rdz + (r\del_t Z_\Phi) dt + (r\del_r Z_\Phi) dr.
\]
Now, since $g_{\xi,\Phi}$ is an edge metric, it may be written in terms of $dr,dt, rdz,$ and in this basis 
\begin{align}
G_{\Phi}^* g_{\xi,\alpha}&=\begin{pmatrix} 1 & 0& 0 \\ 0 & 1 & 0 \\ (r\del_t)Z_\Phi & (r\del_r)Z_\Phi & \del_z Z_\Phi \end{pmatrix}^T 
\Big[  g_{\xi,\alpha}(t,r,x, w+Z_\Phi)\Big] \begin{pmatrix} 1 & 0& 0 \\ 0 & 1 & 0 \\ (r\del_t)Z_\Phi & (r\del_r)Z_\Phi & \del_z Z_\Phi \end{pmatrix}, 
\nonumber\\
&=\begin{pmatrix} 1 & 0& 0 \\ 0 & 1 & 0 \\ (r\del_t)Z_\Phi & (r\del_r)Z_\Phi & \del_z Z_\Phi \end{pmatrix}^T \Big[  g_{\xi,\alpha}(t,r,x, w) + 
O(\Phi)\Big] \begin{pmatrix} 1 & 0& 0 \\ 0 & 1 & 0 \\ (r\del_t)Z_\Phi & (r\del_r)Z_\Phi & \del_z Z_\Phi \end{pmatrix}.\label{matrixmult}
\end{align}
Here 
\begin{align}
O(\Phi) =  g_{\xi,\alpha}(t,r,x,w+Z_\Phi)-g_{\xi,\alpha}(t,r,x,w)=h_0(x,\Phi(t,r,x)) \ + \ \sum_{n=i}^I H_i(\xi', \alpha) h_i(x,\Phi(t,r,x)), \label{his}
\end{align} 
where the $h_i$ are functions vanishing along the zero section of $n(\Sigma)$ as in Lemma \ref{composition}, and $H_i$ are smooth combinations 
of $\xi,\alpha$ and $\xi'.$  Specifically, $h_0$ arises from the $\mathrm k$ term in \eqref{nSigmametric}, while $h_i$ arise from writing the 
$y$ and $dy$ factors in \eqref{Fxidiffeos} in polar coordinates $y=r\Theta,$ so that e.g. $h_i=\Theta(x,w)-\Theta(x,0)$ are the restrictions 
of various trigonometric functions minus their $0$-jet along $\Sigma.$

Multiplying the matrices \eqref{matrixmult} yields 
\[ 
G_\Phi^*g_{\xi,\alpha}=g_{\xi,\alpha} + \underbrace{ O(\Phi) + O(\nabla^e \Phi) + O(\Phi, \nabla^e \Phi) + 
O(\Phi, \nabla^e \Phi, \nabla^e \Phi)}_{:=h_{\xi,\alpha,\Phi}},
\]
and since $\Phi, \nabla^e \Phi \in r^{\mu+2}H^{s_0,2}_{b,e},$ Lemma \ref{standardHolder} shows they lie in $r^{1+\epsilon}C^{s_0-f(n)}_b$ for a 
dimensional constant $f(n).$  Here  $h_{\xi,\alpha,\Phi}$ is expressed in the edge basis $dt,dr, r dx_j.$
\end{proof}

\subsection{Tame Fr\'echet Manifolds}\label{trivializationsSection}
In this subsection, we collect what was done in the previous two subsections to define tame Fr\'echet manifolds and vector 
bundles as in \eqref{vectorbundles.}. 

Let 
\begin{align} 
 \FrDmmu&:= \bigcap_{s\geq s_0} r^{\mu+2} H^{s,2}_{b,e}(X_0; NX_0) \hspace{1.5cm}  \label{bbXdef} \\ 
 \medskip
 \mathbb D^\infty&:=\bigcap_{s\geq s_0} H^{s+2}(S^1;\mathcal B_0)  \\ 
 \mathbb P^{\infty}&:= \bigcap_{s\geq 0} H^{s+3/2}(\del X_0; NX_0)^\perp, \label{bbPdef}
\end{align}
where $H^s(\del X_0; NX_0)^\perp=\{\rho \in H^s(\del X_0;NX_0) \ | \ \rho \perp_{L^2} \mathcal B_0\}$ denotes the {\it finite codimension} 
subspace orthogonal to $\mathcal B_0$ viewed as the eigensections constant in the $S^1$ direction. This orthogonality 
refines the definition in Notation \ref{Frechetnotation}. Here $\mathcal U_0\subseteq \mathbb D^{s_0}, \mathcal U_1 \subseteq \mathbb X^{\mu,s_0}$ 
are open balls around the origin of radius $\delta>0,$ to be chosen later. Additionally, we write $\mathbb X^{\mu,\infty}_D$ for 
the subset obeying vanishing Dirichlet boundary conditions at $r=1.$

These spaces serve as the domain for our parameterized mean curvature map. For the codomain, set 
\[
\mathbb Y^{\mu,s}:=\bigsqcup_{(\xi,\alpha,\Phi)} r^\mu H^s_b(X_0; NX_{\xi,\alpha,\Phi}) \hspace{2cm}\mathbb Y^{\mu,\infty}:=\bigcap_{s\geq s_0} \mathbb Y^{\mu,s} 
\]
where $NX_{\xi,\alpha,\Phi}$ denotes the normal bundle to $X_0$ measured in the pullback $g_{\xi,\alpha,\Phi}$-metric for each 
$(\xi,\alpha,\Phi)\in \mathcal U_0\times \mathcal U_1$. The mean curvature takes values in this bundle. 

Let $P_0: T(S^1\times \R^N)\to NX_0$ denote the orthogonal projection to the normal bundle in the ambient product metric $g_0$. 
Write $u=(\xi,\alpha,\Phi)\in \mathcal U_0\times \mathcal U_1$ for short, and let $P_{u}: T(S^1\times \R^N)\to NX_{u}$ denote 
the orthogonal projection to normal bundle in the pullback metric $g_u=g_{\xi,\alpha,\Phi}.$  Finally, let
\[
\Upsilon_{P_0}: \Gamma(X_0; NX_{u})\to \Gamma(X_0;NX_0)
\]
be the induces map on sections by the $u=0$ projection. 
\begin{lm}\label{Ytrivialization} The induced map 
\[
\Upsilon_{P_0}: \mathbb Y^{\mu,\infty} \ \simeq  \ (\mathcal U_0\times \mathcal U_1)\times r^\mu H^\infty_b(X_0;NX_0) 
\]
is (smooth tame) family of (tame) linear isomorphisms, and thus endows $\mathbb Y^{\mu,s}$ with the structure of a Fr\'echet vector bundle. 
Moreover, the projections satisfy: 
\be 
P_0 P_{u} = P_0 \hspace{2cm}  P_{u}P_0 = P_{u}\label{projectionrelation}
\ee
pointwise on $T(S^1\times \R^N).$
\end{lm}
\begin{proof} Since $g_u=g_{\xi,\alpha,\Phi}$ is $C^1$-close to $g_0$ globally, the projection $P_0: NX_0 \to NX_{u}$ is an isomorphism on each 
fiber.  Similarly, $g_u \in C^\infty_b$, so the projection map, which is an algebraic combination of metric coefficients, is in this space 
as well. Moreover, the norm of the induced map on sections in $r^\mu H^s_b$ is bounded by the $({s+n})$-norm of $u$. This establishes the
existence of the tame isomorphism. 

For the projection statement, let $V\in T(S^1\times \R^N)$. Then 
\[
V-P_0V \in TX_0 \hspace{2cm} V-P_u V \in TX_0
\]
are both tangential. Since both projections annhilate the tangent bundle, 
\[ 
0=P_0(V-P_uV) \ \Rightarrow \ P_0V=P_0P_uV 
\]
and vice versa. 
\end{proof}

\medskip 

We now define the tame mean curvature section. Because the original `Caffarelli-Hardt-Simon' finite-dimensional obstruction is retained, 
we can perform a finite-dimensional alteration on the constant modes. Write 
\[
\xi(t)=\xi_0 + \xi^\perp(t), \quad \alpha(t)= \alpha_0 + \alpha^\perp(t)
\]
where $(\xi_0,\alpha_0) \in \mathcal B_0$ is the constant Fourier coefficient coefficient. As for the $r=1$ boundary,  let 
\be
\text{ext}: H^{s+3/2}(\del X_0; NX_0) \lre r^{\mu+2}H^{s,2}_{b,e}(\del X_0; NX_0)
\label{Sobolevext}
\ee
be {\it any} continuous bounded extension supported in $r\geq 1/2.$ Since $r$ is euniformly bounded above and below in this region,
$\smash{r^{\mu+2}H^{s,2}_{b,e}\simeq H^{s+2}}$ is equivalent to the standard Sobolev space here, and $\text{ext}$ is any standard 
Sobolev extension. By \eqref{bbPdef}, $H^{s+3/2}(\del X_0; NX_0)=\mathcal B_0 \oplus \mathbb P^s.$ Thus, for 
$\varrho \in \mathbb P^s,$ 
\[
(\xi_0,\alpha_0) + \varrho \in H^{s+3/2}(\del X_0; NX_0)
\]
is a sum of orthogonal components. Finally, let
\[
\mathcal V_2:= \Big \{ \varrho \in  H^{\infty}(\del X_0; NX_0)^\perp \ \Big | \ \|\varrho\|_{s_0,2}<\delta \Big\} \subseteq \mathbb P^\infty 
\]
with $s_0,\delta$ as in \eqref{U0def}.  Then, if $\mathcal U_{1,D}\subseteq \mathcal U_1$ denotes the subset where Dirichlet conditions
are satisfied at $r=1$, the extensions induce a tame inclusion
\[
\mathcal U_{1,D}\times \mathcal V_2 \lre \mathcal U_1,  \qquad (\ph, \varrho) \mapsto \Phi=\ph + \text{ext}(\varrho),
\]
the image of which is the complement of $\mathcal B_0.$

\begin{defn}\label{smoothtameMC}  The {\bf universal mean curvature} is the section
\be 
\tikzset{node distance=3.0cm, auto}
\begin{tikzpicture}
\node(D)[yshift=-1.5cm,xshift=.5cm]{$ \mathcal U_0\times \mathcal U_{1,D} \times \mathcal V_2$};
\node(Y)[yshift=-1.5cm][above of=D]{$  \mathbb Y^{\mu,\infty} $};
\node(Q)[yshift=-2.3cm][xshift=-.7cm][above of=D]{$   {\mathbb Q} $};
%\node(E)[overlay,yshift=-2.3cm][xshift=5.0cm][above of=D]{$\begin{matrix}{ \mathbb Q}(\xi,\alpha,\ph,\varrho)=\mathcal Q(g_{\xi^\perp,\alpha^\perp,\Phi}) \\ \Phi=\ph + \text{ext}((\xi_0,\alpha_0) + \varrho)\end{matrix}$};
%\node(Y'')[xshift=-2.2cm][yshift=-.1cm][right of=D]{$=$};
%\node(Y')[xshift=0.6cm, yshift=-.2cm][right of=D]{$ \left\{ X_{\Phi,\xi} \   \Big | \  \  \begin{matrix} \text{reg}(X_{\Phi,\xi})=\text{graph}(\Phi)  \\ \text{sing}(X_{\Phi,\xi})=\text{graph}(\xi) ) \end{matrix}\right\} $};
%\node(D')[xshift=-.8cm, yshift=0.9cm]{$\Theta $};
\draw[->](Y) to node {$$} (D);
\draw[->][swap] (D) edge[bend left=45]node[below] {$$} (Y);
%\draw[->] (D) to node {$\widetilde {\slashed {\mathbb D}}$} (E);
\end{tikzpicture}
\label{vectorbundles.}
\ee
\noindent where
\begin{align}
{ \mathbb Q}(\xi,\alpha,\ph,\varrho)=\mathcal Q(g_{\xi^\perp,\alpha^\perp,\Phi}) \qquad \Phi=\ph + \text{ext}((\xi_0,\alpha_0) + \varrho) 
\label{tameMCformula}
\end{align}
with $\mathcal Q(g)$ denoting the mean curvature of a metric $g$.   
\end{defn}
We emphasize that the constant Fourier modes are included with the boundary extensions, so the push-off diffeormorphism 
$F_{\xi^\perp,\alpha^\perp}$ is defined using the pair $(\alpha^\perp, \xi^\perp).$

\begin{cor} \label{tameMC}
$\mathbb Q$ is a smooth tame section. Moreover, if $E_0=\del_t$ and $E_1,...,E_{n}$ are a local orthonormal frame on $C(\Sigma),$ 
then in the trivialization of Lemma \eqref{Ytrivialization}, 
\[
\mathbb Q(\xi,\alpha,\ph,\varrho)=P_0\left(\sum_{I,J=0}^{n} \frak g^{IJ}_{u} \nabla^{g_u}_{E_i}E_j.\right)
\]
where $\frak g_u^{IJ}$ denotes the dual metric on $X_0,$ and $\nabla^{g_u}$ is the Levi-Civita connection of the ambient metric $g_u$ on 
$S^1\times \R^N.$ Here,  $u=(\xi^\perp, \alpha^\perp, \Phi),$ with $\Phi$ as in \eqref{tameMCformula}.
\end{cor}
\begin{proof}
By definition, the mean curvature of the metric $g_u|_{X_0}$ in a local frame equals
\be 
\mathcal Q(g_u |_{X_0})=P_u\left(\sum_{I,J=0}^{n} \frak g^{IJ}_{u}  \nabla^{g_u}_{E_i}E_j\right) \in \Gamma(X_0; NX_u). 
\label{trueMC}\ee
Thus, in the trivialization of Lemma \eqref{Ytrivialization}, the projection identity $P_0P_u=P_0$ from \eqref{projectionrelation} yields 
the desired expression. Smooth tameness is immediate from the smooth tame dependence on the metric coefficients in Lemmas \ref{Fxidiffeos} and \ref{pullbackgPhixi} and Lemma \ref{tamenesscriterion}. 
\end{proof}

\begin{notation} To avoid cumbersome subscripts, we tacitly continue to write the pullback metrics in \eqref{tameMCformula}, where $g_u=g_{\xi,\alpha,\Phi}$;
The constant Fourier mode is split off and subsumed in $\Phi$ as boundary data. 
\end{notation}

\subsection{The Jacobi Operator} For each $(\xi,\alpha,\Phi)\in \mathcal U_0\times \mathcal U_1\times \mathcal V_2,$ the (vertical component of) the partial linearization of mean curvature in the $\ph$-component extends to a bounded operator
\[
L_{\xi,\alpha,\Phi}:=(\text{d}\mathbb Q)_{\xi,\alpha,\Phi}(0,0,\ph,0):  \FrDms_{\Dir}(X_0; NX_0) \lre \FrRgs(X_0;NX_0)
\]
for every $s\geq 0,$ in the trivialization of Lemma \ref{Ytrivialization}. This is a perturbation of the Jacobi as in Section \ref{PreliminariesSection}. 

This perturbation is milder when viewed in arclength coordinates along the curve $\Gamma_\xi:=(t, \xi^\perp(t)).$  
Reparametrizing by the corresponding arclength, set 
\begin{align}
 2\pi(1+\gamma_\xi):=\text{length}(\Gamma_\xi) \hspace{1.5cm}\tau_\xi&:= \frac{1}{1+\gamma_\xi}\int_0^{t} \sqrt{ 1+|d\xi^\perp|^2} \ dt
\in [0, 2\pi] \label{txidef}
\end{align}
In this new coordinate, the metric equals
\be 
\hspace{0cm}g_{\xi,\alpha,\Phi}= g_\gamma \ + \ O(r)d\tau^2\ + \ \sum_{a=1}^NO(r)d\tau dy_a \ + \ O(r^{1+\epsilon}).  \label{ggammadef}
\ee
where $g_\gamma=(1+ \gamma_\xi^2)^2 d\tau^2 + dr^2 + r^2 g_{S^{N-1}}.$ %The following is immediate from the above formulas: 

\begin{cor}  \label{taunorms}
Altering all norms by replacing 
\[
\nabla_t \leftrightarrow \nabla_\tau \hspace{1.5cm} dV_{g_0}\leftrightarrow dV_{g_\gamma}
\]
gives a smooth tame equivalence of norms on $\FrDmmu, \FrRgmu, \mathbb D^{\infty},\mathbb P^{\infty}.$
\end{cor}
In particular, all metric coefficients in  \eqref{ggammadef} still depend smoothly and tamely on $(\xi,\alpha,\Phi)$ in the $\tau$-norms. 

Let $L_\gamma$ denote the product Jacobi operator $L_0$ as in \eqref{L0expression}, but with $\Delta_\tau$ in place of $\Delta_t$. 
\begin{lm} \label{Jacobioperatoredge}
%The Jacobi operator $L_{\xi,\alpha,\Phi}$ 
We have that
\be 
L_{\xi,\alpha,\Phi}= L_\gamma + \frak d_{\xi,\alpha,\Phi} \qquad \frak d_{\xi,\alpha,\Phi}=\sum_{|I|\leq 2} \frac{\frak a_{\xi,\alpha,\Phi}^I(t,r,x)}{r^2} \nabla_e^I. 
\label{Ldecomp}
\ee
Thus $\frak d_{\xi,\alpha,\Phi}$ is a second-order edge operator with coefficients $\frak a_{\xi,\alpha,\Phi}^{I}\in r C^{s}_b(X_0; End(NX_0))$
where $\smash{\|\frak a_{\xi,\alpha,\Phi}^{I}\|_{rC^{s_0}_b}\leq C\delta}$ in the $s_0$-norm. All are smooth tame functions of $(\xi,\alpha,\Phi)\in 
\mathcal U_0\times \mathcal U_1 \times \mathcal V_2.$ In particular, the indicial roots $\text{Ind}(L_{\xi,\alpha,\Phi})= \text{Ind}(L_0)$ coincide. 
\end{lm}
\begin{proof}
Recall that $L_{\xi,\alpha,\Phi}=\Delta_{\xi,\alpha,\Phi} + (hh^\star)_{\xi,\alpha,\Phi}$ (the last term is the shape operator).  The Laplacian and
shape operators take the following form on $(S^1_\tau\times \R^N, g_{\xi,\alpha,\Phi})$:
\[
\begin{split}
\Delta_{\xi,\alpha,\Phi}&= -\sum_{a,b}\frac{1}{\sqrt{|g|}} \del_a \left(\sqrt{|g|} g^{ab}\del_b\right)\\
\Big((hh^\star)_{\xi,\alpha,\Phi}\Big)^{a}_b &:=\sum_{c=0}^N \sum_{i,j,k,m=0}^{n-1} \frak g^{ik} \frak g^{jm} h_{ij}^a h_{km}^b \mathfrak n_{bc}
\end{split}
\]
We are careful to denote the restriction of an ambient metric by $g=(\frak g, \frak n) \in \text{Sym}^2(TX_0)\oplus \text{Sym}^2(NX_0)$ for the metrics induced on $TX_0$ and $NX_0$ respectively, with subscript $(\xi,\alpha,\Phi)$ omitted for brevity.   The indices $a,b,c$ range over ambient indices on $S^1\times \R^N,$ while $i,j,k,m$ range 
over $TX_0$, $h=\tfrac{h_\Sigma}{r} + O(\delta)$ is the second fundamental form, in terms of the one ($h_\Sigma$) on $\Sigma.$

The expression \eqref{ggammadef} for the metric shows that 
\[  
\Delta_{\xi,\alpha,\Phi}- \Delta_\gamma= O(\delta r) \nabla^2 \ + \ O(\delta) \nabla, \qquad
(hh^\star)_{\xi,\alpha,\Phi}- (hh^\star)_\gamma=O(\delta r^{-1}).
\]
In polar coordinates on $X_0$, we obtain the desired expression in terms of edge derivatives. 
\noindent 
\end{proof}

We now revert temporarily to the simpler case  $(\xi,\alpha,\Phi)=0$. The linearization $L_0(\ph)=(\text{d}\mathbb Q)_0(0,0,\ph)$ with respect to 
the $\mathcal U^1_\Dir$ factor is the Jacobi operator with Dirichlet boundary condition. The full linearization $(\text{d}\mathbb Q)_0(0,0,\ph)$ 
differentiates with respect to vector fields too singular to lie in $r^{\mu+2}H^{s,2}_{b,e}$.   
\begin{defn}\label{extendedJacobi} 
The {\bf extended Jacobi operator} linearization of the mean curvature with respect to the 
$\mathcal U_0\times \mathcal U_{1,D}$ components is 
\be 
\mathbb L_{0}:= (\text{d}\mathbb Q)_{0,0,0}: H^{s+2}(S^1; \mathcal B_0)  \ \oplus \ r^{\mu+2}H^{s,2}_{b,e}(X_0; NX_0) \to r^\mu H^s_b(X_0; NX_0), 
\ee
\end{defn}
The key point is that the more singular vector fields have image no longer orthogonal to the obstruction. Recalling the splitting 
$r^\mu L^2\simeq \text{\bf Ob}_0 \oplus \text{\bf Rg}_0$ from Lemma \ref{secondorderinvertibility}, the extended Jacobi operator 
has a block lower-diagonal operator at regularity $s=0$: 
\be
\mathbb L_0(\eta,\beta,\ph)=\begin{pmatrix} \Pi_0 \circ \text{d}\mathbb Q( \cdot , \cdot, 0  ) & & 0  \\  \\ \Pi_0^\perp \circ \text{d}\mathbb Q(\cdot,\cdot,0)  & &  L_0 \end{pmatrix} \begin{pmatrix} (\eta,\beta) \\  \\ \ph \end{pmatrix} : \begin{matrix}  H^{2}(S^1 ;\mathcal B_0)\\ \oplus \\ r^{\mu+2}H^{2}_{e,D}(X_0;NX_0) \end{matrix} \ \ \lre \ \ \begin{matrix}r^\mu L^2 \cap \text{\bf Ob}_0 \\ \oplus   \\ r^\mu L^2 \cap   \text{\bf Rg}_0\end{matrix}.\label{extendedL0}
\ee
The vanishing of the top right component simply reflectse the definition of $\text{\bf Ob}_0.$ The bottom right is an isomorphism by the 
injectivity proved in Corollary \ref{higherorderelliptic}. Thus to show invertibility of $\mathbb L_0,$ it suffices to investigate the top left component.

%% file: sections/sec_deformation_operator.tex
\section{The Deformation Operator}
\label{deformationoperatorsection}
In this section, we prove that the extended Jacobi operator in the product case \eqref{extendedL0} is an isomorphism, up to an appropriate shift of 
regularity. The key point is that deformations resulting from parametrized Euclidean isometries are sufficient to cancel the 
infinite-dimensional obstruction in the mean curvature operator, due to the spectral rigidity assumption on $\Sigma.$
 
Composing the differential with respect to these deformation with the isomorphism $\text{ob}_0^{-1}$ from Corollary \ref{firstob} 
yields an operator acting on sections of vector bundles {\it on the singular set $\Gamma_0\simeq S^1.$}  We begin with 
the linearization at the product submanifold $X_0.$

\begin{defn}\label{T0def}
The composition 
\begin{center}
\be \label{compositiondiagram}
\tikzset{node distance=3.8cm, auto}
\begin{tikzpicture}
\node(A){$ H^{2}(S^1;\mathcal B_0)$};
\node(C)[right of=A]{$\text{\bf Ob}_0$};
\node(C')[right of=C, yshift=0cm]{$L^2(S^1;\mathcal C_0 ). $};
\draw[->] (A) to node {$ \Pi_0 \circ (\text{d}\mathbb Q)_{0}$} (C);
\draw[->] (C) to node {$ \text{ob}_0^{-1}$} (C');
\draw[->,>=stealth] (A) edge[bend right=20]node[below]{$T_0$} (C');
\end{tikzpicture}
\ee \end{center}
\[
T_0 \ =  \text{ob}_0^{-1}  \ \circ \  \Pi_0 \  \circ  \ (\text{d}\mathbb Q)_0(\cdot, \cdot ,0)
\]
is called the {\bf deformation operator}.
\end{defn}

\smallskip 

We now state the main result of this section.  
\begin{thm}\label{boundarydeformation}   Fixing $\epsilon$ by $\mu+1=\tfrac{n}{2}+\epsilon$ as in Subsection \ref{weightconventions}, 
then the boundary deformation operator is an invertible pseudodifferential operator which decomposes as
\be 
T_0 \begin{pmatrix} \eta \\ \beta \end{pmatrix}= -\begin{pmatrix} (1+\Delta_t)^{-{(1-\epsilon)}/{2}} \Delta_t    & 0 \\ 0 & (1+\Delta_t)^{-{(2-\epsilon)}/{2}} \Delta_t  \end{pmatrix}\begin{pmatrix} \eta \\ \beta \end{pmatrix} \ + \ K_0\begin{pmatrix} \eta \\ \beta \end{pmatrix}
\ee
where $K_0$ is a strictly negative, diagonal self-adjoint pseudodifferential operator, with each component of order lower than the
corresponding component of the principal term.  In particular, $T_0$ %is an invertible, 
%formally 
%self-adjoint elliptic pseudo-differential operator 
has mixed order $(1+\epsilon, \epsilon)$ in the two components. % in $(\eta,\beta)$ respectively. 
\end{thm}
Implicit in this result is the polynomial loss of regularity that necessitates the use of Nash-Moser theory in the proof of our main result. 

\begin{cor} \label{invertibleLbb0} The extended Jacobi operator is a bounded linear isomorphism
\be
\mathbb L_0(\eta,\beta,\ph)=\begin{pmatrix} T_0  \\  \\ \Pi_0^\perp \circ (\text{d}\mathbb Q)_0(\_,\_,0)  & &  L_0 \end{pmatrix} \begin{pmatrix} (\eta,\beta) \\  \\ \ph \end{pmatrix} : \begin{matrix}   (H^3 \oplus H^2)\\ \oplus \\ r^{\mu+2}H^{2}_{e,D}(X_0;NX_0) \end{matrix} \ \ \lre \ \ \begin{matrix}   H^{2-\epsilon}(S^1;\mathcal C_0)\\ \oplus   \\ \text{\bf Rg}_0\end{matrix} ; \label{extendedL0}
\ee
the first (top) summand in the domain is shorthand for $H^{3}(S^1;\R^N)\oplus H^{2}(S^1;\R^K)\subseteq L^2(S^1;\mathcal B_0).$
\end{cor}
\begin{rem}\label{obconvention} The pseudodifferential order of $T_{0}$ depends on the choice of isomorphism 
$\text{ob}: L^2(S^1;\mathcal C_0)\to {\bf Ob}_0$.  Altering this choice by a factor of $(1+\Delta)^s$, for example, for any $s\in \R$,
would change the orders of the components of $T_0$ accordingly.
% ould revise it so that it appeared as any specified order as an operator into $L^2(S^1;\mathcal C_0)$. 
The ambient obstruction component $\Pi_0\circ (\text{d}\mathbb Q)_0(\cdot ,\cdot ,0)$, however, is independent of $\text{ob}_0$, 
has dense range so all three maps in  
\begin{center}
\tikzset{node distance=3.4cm, auto}
\begin{tikzpicture}
\node(A){$H^3\oplus H^2$};
\node(C)[right of=A]{};
\node(C')[right of=C, yshift=0cm]{$ \text{\bf Ob}_0\cap r^\mu H^{2-\epsilon}_b(X_0;NX_0)    $};
\node(D)[below of=C, yshift=1.4cm]{$H^{2-\epsilon}(S^1;\mathcal C_0)$};
\draw[->] (A) to node {$ \Pi_0 (\text{d}\mathbb Q)$} (C');
\draw[->][swap] (A) to node {$T_{\Phi_0}$} (D);
\draw[->][swap] (D) to node {$ \text{ob}$} (C');
\end{tikzpicture}\end{center}
are isomorphisms. Our convention to make $\text{ob}$ have order $0$ is to make the order of $T_0$ agree with the ambient 
order of $\Pi_0 (\text{d}\mathbb Q)$ (appropriately interpreted since this latter operator is a map between sections of
bundles on $S^1$ and $X_0$, respectively).
\end{rem}

\subsection{The Product Deformation Operator} 
We now prove Theorem \ref{boundarydeformation}.  The first step is to calculate the linearized mean curvature at the product metric. 
\begin{lm} \label{productlinearization} The linearized mean curvature at $(0,0)\in \mathcal U_0 \times \mathcal U_1$,
acting on linear deformations of parametrized Euclidean isometries $(\eta,\beta)\in C^\infty(S^1; \mathcal B_0)$,
is given by
\be 
(\text{d}\mathcal Q)_{0}(\eta,\beta, 0)= P_0\left[ \eta''(t) \ + \  \beta''(t) y\right], \label{productmeancurvature}
\ee
where $P_0$ is projection to the normal bundle, and $\beta''(t)y$ is the action of $\beta''(t)\in \frak{so}(N).$
We are using here the trivialization of Lemma \ref{Ytrivialization}.
\end{lm}
\begin{proof}  For each of the quantities in the left column below, with notation as in Lemma \ref{Fxidiffeos}, we 
list on the right the corresponding variations with respect to  the 1-parameter family $(\xi_s, \alpha_s)=(0,0) + (s\eta, s\beta)$: 
\begin{alignat}{2}
a_\xi(t)&=\sqrt{1 + |\xi'(t)|^2} & \qquad    \dot a_\xi(t) & =0   \label{derivs1}\\
\omega_{ab}(t)&= \br \nabla_t e_a, e_b  \kt  & \qquad   \dot \omega_{ab}(t)  & =0  \label{derivs2} \\
\Omega_A(t)&=  -A^T(t) \omega(t) A(t)+ A(t)^T A'(t) & \qquad   \dot\Omega_A(t)  & =\beta'(t) \label{derivs3} \\
k_b(t)&= a_\xi(t)^{-1} \br   \nabla_t  e_\Gamma  \ , \ f_b \kt  & \qquad   \dot k_b(t) & =\eta''(t). \label{derivs4} 
\end{alignat}
Here $e_a$ and $f_b$ are orthonormal bases as in \eqref{Fxidef}. Assuming this claim momentarily, we complete the proof. 

For shorthand, write
\[ 
g_s:=g_{\xi_s, \alpha_s} \hspace{1.5cm} \frak g_{s}:=g_{\xi_s, \alpha_s}\big|_{X_0}
\]
%for the family of pullback metrics and their restrictions to $X_0$. 
Choose a point $p\in C(\Sigma)$, set $E_0=\del_t$ and let $E_1,...,E_{n-1}$ denote a local geodesic orthonormal frame 
of $C(\Sigma)$ in the $s=0$ metric, i.e., so that $(\nabla_{E_i} E_j)(p) = 0$. The mean curvature at $(t,p)$ for $t\in S^1$ is 
\be 
\mathcal Q(g_s)=P_0\left[ \sum_{i,j=0}^{n-1} \frak g^{ij}_s \nabla^{g_s}_{E_i}E_j   \right] \label{fullMC}
\ee
Differentiating  gives
\begin{align*}
(\text{d}\mathbb Q)_{0}(\eta,\beta, 0)&=  P_0\left[ \sum_{i,j=0}^{n-1}\dot {\frak g}^{ij}_s \nabla^{g_s}_{E_i}E_j   \right]_{s=0} \ + 
\ P_0\left[ \sum_{i,j=0}^{n-1}{\frak g}^{ij}_s  \dot \nabla^{g_s}_{E_i}E_j   \right] \\
&= P_0\left( \sum_{i=1}^{n-1} \dot {\frak g}_{ti} (\nabla^{g_0}_{\del_t}E_i \ + \ \nabla^{g_0}_{E_i}\del_t )\right)  \ + \  
P_0 \left[ \sum_{i=0}^{n-1} \dot \nabla^{g_s}_{E_i}E_i \right].
\end{align*}
Note that the projection does not depend on $s$ by dint of the trivialization of Lemma \ref{Ytrivialization}. For the true normal bundle 
$\dot P_s$ would contribute a third term (which vanishes at the product for $s=0$). In the first term, we have used that $\dot g^{ij}=0$ for $i,j\geq 1,$ 
since $g^{ij}_s=\delta_{ij}$ for all $s \geq 0$, $i,j\geq 1$, by Lemma \ref{Fxidiffeos}. On the other hand, $\nabla_t E_i=\nabla_{i}\del_t =0$. 
Thus the first term vanishes. Meanwhile, $\frak g_0^{ij}=\delta^{ij},$ so the second term collapses to the trace. 

Using all of this, then writing $E_i=\sum_a E_i^a \del_{y_a}$ in the ambient frame, we have
\begin{align*}
(\text{d}\mathbb Q)_{0}(\eta,\beta, 0)&= P_0\left[\sum_{c=1}^{N}\left( \dot \Gamma_{tt}^c  \ +  \ 
\sum_{i=1}^{n-1} \sum_{a,b=1}^N E_i^a E_i^b \Gamma_{ab}^c\right) \del_{y_c}\right] \hspace{2.5cm}
\end{align*}
where the $\del_t$ term is dropped since $P_0(\del_t)=0.$ Recalling that 
\[ 
\dot  \Gamma_{ab}^c=\frac{1}{2}\left(\del_a \dot g_{bc} \ + \ \del_{b}\dot g_{ac} \ - \ \del_c \dot g_{ab} \right), 
\]
the fact that $\dot g_{ab}=0$ for $a,b\geq 1$ shows that the second term vanishes. Thus
\begin{align}
(\text{d}\mathbb Q)_{0}(\eta,\beta, 0) &= P_0\left[\sum_{c=1}^{N} \dot \Gamma_{tt}^c \del_{y_c}\right]= P_0\left[\sum_{c=1}^{N}  \left(\del_t \dot g_{tc} - \frac{1}{2}\del_c \dot g_{tt}\right)\right].
\end{align}
Using \eqref{derivs1}--\eqref{derivs4} to differentiate the expressions in Lemma \ref{Fxidiffeos} gives
\[ 
\dot g_{tt}= 2\br - \eta''(t) , y \kt \hspace{1cm} \dot g_{tc}=  \beta'(t)y 
\]
Thus 
\begin{align}
(\text{d}\mathcal Q)_{0}(\eta,\beta, 0) &= P_0\left[ \eta''(t) \ + \  y\beta''(t) \right],
\end{align}
as desired.  
\end{proof}

It remains to verify \eqref{derivs1}--\eqref{derivs4}.

\begin{proof}[Proof of Claims \eqref{derivs1}--\eqref{derivs4}]   First, \eqref{derivs1} follows directly by differentiating: 
\[
\dot a_\xi(t)=\frac{2 \xi'_s(t)\cdot \eta(t)}{2\sqrt{1+ |\xi_s'(t)|^2}}\Big |_{s=0}=0.
\]
For \eqref{derivs2}, observe that 
\[ 
\widetilde e_a:= (0, \del_{y_a})- s \eta'_a(t) e_{\Gamma_s}
\]
is a (not necessarily orthonormal) basis of $N\Gamma_{\xi_s},$. Since $e_{\Gamma_s}=(1,s\eta(s)),$ taking inner products shows 
$\br  \widetilde e_a , \widetilde e_b\kt= \delta_{ab} + O(s^2).$ The Gram-Schmidt process thus yields an orthonormal basis 
$e_a(s)=\widetilde e_a + O(s^2).$ We then compute
\begin{align*}
\dot \omega_{ab} &= \br  \nabla_t \tfrac{d}{ds}\big |_{s=0}{ \widetilde {e}_a}, (0,\del_{y_b})\kt \ + \  
\br  \nabla_t { \widetilde {e}_a}, \tfrac{d}{ds}\big |_{s=0}(0,\del_{y_b})\kt \\
&= \br (\nabla_t \eta_a'(t)) \del_t \ , \ \del_{y_a}\kt  \ + \  s\br \eta'(t)\dot e_{\Gamma}   \ , \ \del_{y_b}\kt  \big|_{s=0} \ + \ 
\br \nabla_t { \widetilde {e}_a}  \  , \  0  \kt = 0. \\
\end{align*}
%where the first term vanishes by orthogonality. 

For \eqref{derivs3},
\begin{align*}
\dot \Omega_A(t)&=\Big[ -\dot A_s^T \omega_s A_s -{ {A_s^T \dot \omega_s A_s}} - A_s^T \omega_s \dot A_s  + \dot A^T_s A'_s \ + \ A^T_s \dot A'_s\Big]_{s=0} \label{Omegawexpression}\\
&=  0 + 0 + 0 + \beta(t)^T (s\beta'(t))  \Big |_{s=0}\ + \  \beta'(t) =\beta'(t).
\end{align*}
using that $\omega_{ab}(0)=0$ by orthogonality, $\dot \omega_{ab}(0)=0$ as proved above, and $A(0)=\text{exp}(0)=\text{Id}.$

Finally, using that $e_{\Gamma_s}=(1, s\eta'(t) ) + O(s^2),$
\begin{align*}
\dot k_b(t)&= \dot a_\xi^{-1} \br\cancel{  \nabla_t e_{\Gamma_s}} , f_b \kt \Big |_{s=0} \  + \  a_\xi^{-1} \br  \nabla_t \dot e_{\Gamma_s} , f_b \kt \Big |_{s=0} \  + \ a_\xi^{-1} \br  \cancel{  \nabla_t e_{\Gamma_s}}  , \dot f_b \kt \Big |_{s=0}\\
&=  0 \ + \ \br  \eta''(t) \ , \ (0,\del_{y_b})\kt \ +  \ 0 =\eta''_b(t).
\end{align*}
This completes the proof of these claims. 
\end{proof}

We now prove the theorem. 
\begin{proof}[Proof of Theorem \ref{boundarydeformation}] Let $(\eta(t),\beta(t))=(\eta_0, \beta_0) + (\eta^\perp(t), \beta^\perp(t)).$  From Definition \eqref{tameMCformula} 
and Proposition \ref{productlinearization} 
\[
T_0(\eta,\beta)=\text{ob}_0^{-1} \circ \Pi_0 \Bigg[ P_0\Big( (\eta^\perp)''  \ + \  (\beta^\perp)'' y \Big)\Bigg] \ + \ \text{ob}^{-1}_0 \Pi_0 \Bigg[L_0(\frak E(\eta_0, \beta_0))\Bigg].
\]
The second term is supported entirely in the constant Fourier mode. The fact that it is an isomorphism $\mathcal B_0\to \mathcal C_0$ is simply the original argument of
Caffarelli-Hardt-Simon from Theorem \ref{CHS} (this is independent of the choice of boundary extension by \eqref{Greens2ndbd}. 
In fact, it is strictly negative since $\Psi_j$ is a monotonically decreasing function.

Write 
\[
\eta^\perp(t)=\sum_{a=1}^N \eta^\perp_a(t) \phi_a \hspace{1.5cm} \beta^\perp(t)=\sum_{k=1}^K \beta^\perp_k(t) \phi_k, 
\]
where $\phi_a$, $a=1,...,N$, and $\phi_k$, $k=1,...,K$, are orthonormal bases of the translational and rotational parts of $\mathcal B_0.$ 
By our choice of $\mathcal B_0$ in \eqref{B0iso}, $P_0$ is an isomorphism on the latter, so 
\[
P_0\Big( (\eta^\perp)''  \ + \  (\beta^\perp)'' y \Big)= \sum_{a=1}^N \eta_a^\perp(t)\phi_a + r\sum_{k=1}^K \beta^\perp_k(t) \phi_k.
\]

Because $\eta,\beta\in H^2,$ these second derivative terms have effective orders 
\[
s_\eta= 0 + 0 + (\tfrac{n}{2}-\mu)=1-\epsilon \hspace{1.5cm}s_\beta= 0 + 1 + (\tfrac{n}{2}-\mu)=2-\epsilon,
\]
respectively, in the sense of Definition \ref{effectivescaling}. By Lemma 4.11, the projections are bounded as maps into $H^{1-\epsilon}$ and $H^{2-\epsilon}$, respectively. 
Moreover, the same calculation as in the proof of that lemma yields here that, 
 \begin{align}
 \sum_{j,\ell}\br (\eta_a^\perp)''(t)\phi_a \ , \  \Psi_{j\ell}\kt_{r^\mu L^2(X_0)}&= \int_{X_0} \br (\eta_a^\perp)''(t)\phi_a \ , \  r^{2\mu}\Psi_{j\ell}^\dag(r) \phi_j e^{i\ell t} \kt \ r^{-2\mu}r^{n-1}dr dt dV_\Sigma\nonumber  \\
&= -\sum_{\ell}\eta_{n\ell} |\ell|^2 \cdot  |\ell|^{\mu+1} \int_0^1  r^{\tfrac{2-n}{2}} K_{\nu_a}(|\ell| r ) r^{n-1}dr \nonumber \\
&= -\sum_{\ell}\eta_{a\ell} |\ell|^2 \cdot  |\ell|^{\mu-\tfrac{n}{2}} \int_0^{|\ell|}  \rho^{\tfrac{n}{2}} K_{\nu_a}(\rho ) d\rho \nonumber \\
&=  -\sum_{\ell}\eta_{a\ell} |\ell|^2 \cdot  |\ell|^{\mu-\tfrac{n}{2}}  c_{n\ell} \label{finalint}
\intertext{\noindent where the $c_{n\ell}$ are a strictly positive sequence asymptotic to a constant as $|\ell| \to \infty$. Combining the lower order terms into the 
compact operator $K_0$, and recalling that $\mu-\tfrac{n}{2}=\epsilon-1$, we obtain}
&= -(1+\Delta_t)^{-{(1-\epsilon)}/{2}}\Delta_t(\eta^\perp) + K_0(\eta^\perp), 
 \end{align}
as desired. The calculation for $\beta^\perp$ is identical, but the additional factor of $r$ leads to an additional factor of $|\ell|^{-1}$ in \eqref{finalint}. 
Combining this with the strictly negative isomorphism on the constant modes completes the proof. 
\end{proof}

 \subsection{The General Deformation Operator} 
\label{generaldeformations}
In order to apply the Nash-Moser implicit function theorem \ref{nashmoser}, we must understand the deformation operator at a general configuration 
$(\xi,\alpha,\Phi) \in \mathcal U_0 \times \mathcal U_1.$  This will be accomplished in two step.  The first, carried out below, calculates the linearized mean curvature 
\[
\text{}(\text{d}\mathbb Q)_{(\xi,\alpha,\Phi)}(\eta,\beta,0) \in  \Gamma(X_0; NX_0)
\]
in the $T\mathcal U_0$ directions at a general $(\xi,\alpha,\Phi).$ The second is carried out in Section \ref{obstructionperturbations}, where we study the family of projections $\Pi_{\xi,\alpha,\Phi}$. % to the corresponding obstructions.  
  
The calculation of the full mean curvature is lengthy and unwieldy. Fortunately, we do not need its explicit expression, but only some general structural properties,
namely that $\text{d}\mathbb Q$ has the same leading effective order (in the sense of Definition \ref {effectivescaling}) independent of basepoint $(\xi,\alpha,\Phi).$  
\begin{prop} \label{generalmeancurvature}
The linearized mean curvature at $(\xi,\alpha,\Phi)\in \mathcal U_0 \times \mathcal U_1$, acting on infinitesimal parametrized Euclidean 
isometries $(\eta,\beta)\in C^\infty(S^1;\mathcal B_0)$, has leading effective order $(1+\epsilon, \epsilon).$ Specifically,  
\be 
\begin{split}
& (\text{d}\mathbb Q)_{(\xi,\alpha,\Phi)}(\eta,\beta, 0) \\ 
& \ \  = P_0\left[  \sum_{j=-1}^2 A_{\xi,\alpha,\Phi}^j(t,r,x) \left(  r^{-j}\del^{2-j}_t\eta\right) \ + \  \sum_{k=0}^2 B_{\xi,\alpha,\Phi}^k(t,r,x)
\left( r^{1-k}\del^{2-k}_t\eta\right) \ + \ \mathcal R_{\xi,\alpha,\ph}(\eta,\beta) \right],
\end{split}
\label{nonproductmeancurvature}
\ee
where $P_{0}$ is the $g_0$-orthogonal projection to $NX_0$, and $A_{\xi,\alpha,\Phi}^k,  B^j_{\xi,\alpha,\Phi} \in C^\infty_b(X_0; \text{End}(T(S^1\times B^N)))$ are bounded 
and $b$-smooth endomorphisms that are smooth tame functions of $(\xi,\alpha,\Phi)$ with $A^0_{0,0,0}=B_{0,0,0}^0=\text{Id}.$ The final term, 
$\mathcal R_{\xi,\alpha,\Phi}$ is a smooth tame remainder of strictly lower effective order.
\end{prop}

The proof relies on expressing $\mathbb Q$ as a collection of terms of the following type. A smooth tame function $M: \mathcal U_0 \times \mathcal U_1 
\to r^{\beta}C^\infty_b(S^1 \times B^N; E)$ for some bundle $E$ is said to be a {\bf smoothly order respecting} of order $(m_\xi, m_\alpha)$ and growth rate $p$ 
if it may be written 
\be 
M(\xi,\alpha,\Phi)=r^{p}F(t,y  \ ,  \ U_1,..., U_{J} \ , \  W_1,...,W_{K}) \hspace{2cm} \begin{matrix} U_j= r^{a_j} \del_t^{m_\xi+a_j}\xi(t) \\ W_k=r^{b_k} 
\del_t^{m_\alpha+b_k}\alpha(t) \end{matrix}\label{ordercompatibleeq}
\ee
for some smooth function $F$ and a collection of integers $a_1,...,a_n, b_1,...,b_{n'} \in \Z.$ Notice that these have effective orders 
\[
\text{ord}(M)=(m_\xi-p, m_\alpha-p) -c_{n,\mu},
\]
respectively, for each $j,k$, where $c_{n,\mu}=\tfrac{n}{2}-\mu=(1-\epsilon)$ is as in Definition \ref{effectivescaling}. More generally, we allow $F$ to be a 
$C^\infty_b$ function of $y\in \R^N.$ Note that such a map is automatically smooth and tame since it depends smoothly on a fixed number of derivatives 
of $(\xi,\alpha,\Phi)$. The examples we consider are essentially just algebraic combinations of $U_1,...,U_J, W_1,...W_{K}.$ 

The purpose of this definition is that these terms preserve effective order in a controlled way under both $b$-derivatives and functional derivatives. 
\begin{lm}  \label{ordercompatible}
Suppose that $M: \mathcal U_0 \times \mathcal U_1 \to r^{p}C^\infty_b(S^1 \times B^N; V)$ is a smoothly order-respecting map of order $(m_1-p,m_2-p)-c_{n,\mu}$ 
and growth rate $p.$ Then 
\begin{enumerate}
\item[(A)] Any finite combination of $e$-derivatives 
\[\nabla_{V_1}^e...\nabla_{V_s}^e M: \mathcal U_0\times \mathcal U_1 \to r^p C^\infty_b(S^1\times B^N; E)\]
is a smoothly order respecting map of the same order and growth rate. 
\item[(B)] The linearization has the form 
\[
(\text{d}M)_{(\xi,\alpha,\Phi)}(\eta,\beta,0)= \sum_{j=1}^J A^j(t,r,\omega) \Big[  r^{a_j+p} \del^{m_1 + a_j}\eta \Big] \ + \  \sum_{k=1}^K B^k(t,r,\omega)\Big[  r^{b_k+p} 
\del^{m_1 + b_k}\eta \Big]  
\]
where $A^j, B^k$ are each $C^\infty_b$-functions which depends smoothly and tamely on $(\xi,\alpha,\Phi).$ I.e. $\text{d}M$ consists of terms of effective order $(m_1 -p, m_2-p)-c_{n,\mu}$ in the two arguments. 
\end{enumerate}
\end{lm}
\begin{proof} After unpacking notation, this is just the chain rule.

(A) Let $V=r\del_t.$ Differentiatng,
\[
\nabla_V M=r^p \left[(\del_t F) \ + \  \sum_{j=1}^J (\del_{U_j} F) \cdot r\del_t U_j \ + \  \sum_{k=1}^K (\del_{W_k} F) \cdot r\del_t W_k \right]. 
\]
Since $F$ is smooth, any partial derivative $\del F$ (as a function of the same arguments) is smoothly order-respecting of the same order. Since $r\del_t$ 
preserves the effective order, $(\del F)\cdot r\del_t U_j$ has the same form \eqref{ordercompatible}, perhaps now with $J+1$ arguments, and likewise for 
the derivatives $(\del F)\cdot r\del_t W_k.$ The remaining and higher-order edge derivatives are similar. 

(B) Now let $(\xi_s,\alpha_s)=(\xi + s\eta, \alpha+s\beta).$ The chain rule again shows that 
\[
(\text{d}M)_{(\xi,\alpha,\Phi)}(\eta,\beta,0)=r^p \left[ \sum_{j=1}^J (\del_{U_j} F) \cdot \dot U_j  \ + \  \sum_{k=1}^K (\del_{W_k} F) \cdot \dot W_k \right], 
\]
where $\dot U_j=\d{}{s}|_{s=0} U_j= r^{a_j} \del^{m_1 + a_j}\eta$ and likewise for $\dot W_k.$ Accounting for the $r^{p}$ term, these have effective orders 
$(m_1-p, m_2-p)-c_{n,\mu}$, respectively, in the sense of Definition \ref{effectivescaling}. Smooth tameness is again immediate because only finite 
many derivatives appear.  
\end{proof}

More generally, a term $M: \mathcal U_0 \times \mathcal U_1 \to r^{p}C^\infty_b(S^1 \times B^N; E)$ is {\bf smoothly leading-order respecting} of order $(m_1-p,m_2-p)-c_{n,\mu}$ with growthrate $p$ if it is a finite-sum of terms $M_i$ of the form \eqref{ordercompatibleeq} with orders  $(m^i_1-p^i,m_2^i - p^i)-c_{n,\mu}$ and  \[(m_1-p, m_2-p)=\max_i \Big(m^i_1-p^i,m_2^i - p^i\Big).\]

The following is obvious after unpacking the definitions: 

\begin{cor} \label{ordercalculus}The order of smoothly leading-order respecting terms obeys the following: 

\begin{itemize}
\item[(A)] $\text{ord}\left(\sum_{i=1}^I M_i\right)=\max_{i\leq I}\text{ord}(M_i)$
\item[(B)] $\text{ord}\left(\prod_{i=1}^I M_i\right)=\max_{i\leq I}\text{ord}(M_i)$
\item[(B)] $\text{ord}\left(f \circ M\right)=\text{ord}(M)$
\end{itemize}
\noindent where \(f\) is smooth on a neighborhood of the range of \(M\). Products may more generally be taken using any fixed smooth pointwise multilinear bundle operations.
\end{cor}

\noindent Recall the quantities \eqref{derivs1}--\eqref{derivs4} occurring in the metric. These and their derivatives have the following effective orders. Here, we write $-\infty$ for the effective order when a quantity is independent of that variable. 

\begin{lm}  \label{leadingorderslem}The following terms are smoothly leading-order respect of the following orders (in the sense of \eqref{ordercompatible}) for  $(\xi,\alpha,\Phi)\in \mathcal U_0 \times \mathcal U_1$:
\begin{align*}
\text{ord}\Big(a_\xi(t)\Big)&=(1,-\infty) \ - \ c_{n,\mu}\\
\text{ord}\Big(A(t)\Big)&=(-\infty,0) \ - \ c_{n,\mu}\\
\text{ord}\Big(\omega(t)\Big)&=(2,-\infty) \ - \ c_{n,\mu}\\
\text{ord}\Big(\Omega_A(t)y\Big)&=(2-1, 1-1) \ - \ c_{n,\mu}\\
\text{ord}\Big(k(t)\Big)&= (2,0) \ - \ c_{n,\mu}
\end{align*}

\noindent 
\end{lm}
\begin{proof} By definition, 

\[a_\xi(t)= \sqrt{1+|\xi'(t)|^2} \hspace{2cm} A(t)=\text{exp}(\alpha(t))\]

\noindent which are obviously order-preserving of order $1-c_{n,\mu}$ and $c_{n,\mu}$ by Defintion \ref{effectivescaling}, since they involves a single derivative or none and no additional growth or decay in $r,$ respectively. 

Recalling that $e_a$ denotes the orthonoral basis from \eqref{Fxidef}, some calculation shows that 

\[e_a=\left( \frac{-\xi'_a}{a_\xi} \ , \ \del_{y_a}  -\frac{\xi_a'\xi'}{a_\xi(a_\xi+1)}\right) \ \ \Rightarrow \ \ \omega_{ab}(t)=\frac{\xi_a''(t)\xi'_b(t) - \xi_a'(t)\xi_b''(t)}{a_\xi(a_\xi+1)}.\]

\noindent The latter is a smooth function of up to the second derivatives, hence order $(2,-\infty)+c_{\mu,n}.$ 

Lemma \ref{ordercompatible} implies that $\nabla_t$ decreases the order by 1, since $r\del_t$ preserves it. Thus $A'(t)$ is order $(0,1)-c_{n,\mu}$, and since $\Omega_A(t)=  -A^T(t) \omega(t) A(t)+ A(t)^T A'(t)$ as in \eqref{derivs3}, $\Omega_A(t)$ has order $(2,1)-c_{n,\mu}$ by Corollary \ref{ordercalculus}(B). Accounting for the the order of growth from $y$ yields the assertion. 

Finally, 
\[k_b(t)=\frac{1}{a_\xi(t)} \br   \nabla_t  e_\Gamma  \ , \ f_b \kt \]

\noindent as in Lemma \ref{Fxidiffeos}, where $e_\Gamma=(t,\xi'(t))$ is order $(1,0)-c_{n,\mu}$ and $f_b=\sum_a (A_{ab}e_a)$ is order $(-\infty,0)-c_{n,\mu}$. Differentiating increases the order by $1$ and Corollary \ref{ordercalculus}(B) completes the lemma.
 \end{proof}

With this preparation, we can now read off the effective orders in the metric coefficients and thus the mean curvature. 

\begin{proof}[Proof of Proposition  \ref{generalmeancurvature}] First consider the case that $\Phi=0.$ Let $P_0, P_0^\perp$ denote the $g_0$-orthogonal projections to the two factors in $T(S^1\times \R^N)=NX_0 \oplus TX_0.$ Set 
\begin{align}
Q&:=a_\xi (1-\br k,y\kt)\\
C&:=  P_0^\perp( \Omega y)\\
D&:= Q^2 \ + \ P_0(\Omega_A y).
\end{align}

\noindent By Corollary \ref{ordercalculus} and Lemma \ref{leadingorderslem} show that 
\[\text{ord}(Q)=(1,-1) - c_{n,\mu} \hspace{1cm} \text{ord}(C)=(1,0) -c_{n,\mu} \hspace{1cm} \text{ord}(D)=(1,0) - c_{n,\mu}. \]

\noindent In the first of these, we borrow the extra growth factor from $y$ so that $\text{ord}(\br k,y\kt)=(1,-1)-c_{n,\mu}.$

With these abbreviations, the expressions for the metrics from Lemma \ref{Fxidiffeos} take the form  

\begin{align*}
g_{\xi,\alpha,\Phi}&=\begin{pmatrix} Q^2 + |\Omega_A y|^2 &(\Omega_Ay)^T  \\ \Omega_A y& \text{Id} \end{pmatrix} \ + \  O (r^\epsilon ...)\\   \\ 
\medskip
g^{-1}_{\xi,\alpha,\Phi}&=\frac{1}{Q^2}\begin{pmatrix} 1  & -(\Omega_A y)^T \\ -\Omega_A y & Q^2+ (\Omega_A y)(\Omega_A y)^T \end{pmatrix}\ + \  O (r^\epsilon ...) \\ \\ 
\medskip
\frak g_{\xi,\alpha,\Phi}&=\begin{pmatrix} Q^2 + |\Omega_A y|^2 &v^T  \\ v & \text{Id} \end{pmatrix}\ + \  O (r^\epsilon ...)  \\ 
\\ 
\frak g^{-1}_{\xi,\alpha,\Phi}&=\frac{1}{D}\begin{pmatrix} 1  & -v^T \\ -v & D+ vv^T \end{pmatrix} \ + \  O (r^\epsilon ...)\end{align*}

\noindent where $\frak g_{\xi,\alpha,\Phi}=g_{\xi,\alpha,\Phi}|_{X_0}$ as before. In each case, $O(r^{\epsilon}...)$ denotes smoothly leading-order respecting terms of the same orders as $Q,C,D$ with an additional growth factor $r^{\epsilon}.$ 
We conclude   
\[\text{ord}(g_{\xi,\alpha,\Phi})= (1,0)-c_{n,\mu} \hspace{2cm}\text{ord}(g_{\xi,\alpha,\Phi})= (1,0)-c_{n,\mu}\]

\noindent And likewise for the inverses. 

Recalling (from \eqref{fullMC}) that

\be \mathbb Q(\xi,\alpha,\Phi)=P_0\left[ \sum_{i,j=0}^{n-1} \frak g^{ij}_{\xi,\alpha,\Phi} \nabla^{g_{\xi,\alpha,\Phi}}_{E_i}E_j   \right]\ee

\noindent Lemma \ref{ordercompatible}(A) shows that each derivative increases the order by at most $1$, because the corresponding edge derivative (which has one additional  factor of $r$) preserves it. We conclude that $\mathbb Q$ is a smooth leading-order respecting map of order $(2,1)-c_{n,\mu}=(2,1)-(1-\epsilon)=(1+\epsilon, \epsilon).$ Taking the linearization and applying Lemma \ref{ordercompatible}(B) yields the desired conclusion and expression.

When $\Phi\neq 0,$ a brief inspection of Lemma \ref{pullbackgPhixi} (specifically \eqref{matrixmult}) shows that the only additional terms depending $\xi,\alpha$ are those arising from the $F_i$ in \eqref{his}. All such terms come with a least one factor of $\Phi=O(r^{1+\epsilon}),$ and so have strictly lower effective order, and can be collected into the remainder $\mathcal R_{\xi,\alpha,\Phi}.$
\end{proof}

%% file: sections/sec_perturbed_obstructions_quicker.tex
%%
%% sec_tame_estimates.tex
%%

\section{Perturbations of the Obstruction}

\label{obstructionperturbations}

Proposition  \ref{generalmeancurvature} shows that the family of linearizations used to form the deformation operator are well-behaved for $(\xi,\alpha,\Phi)\in \mathcal U_0 \times \mathcal U_1.$ The second piece of the verification of Hypothesis (II) of Theorem \ref{nashmoser} is that the corresponding family of projections is also well-behaved, namely that it continues to obey the ``unexpectedly'' good boundedness of Section \ref {secondorderinvertibility}. The current section proves that this is indeed the case. 

For the duration of this section, we write $u:=(\xi,\alpha,\Phi)\in \mathcal U_0\times \mathcal U_1$ for shorthand. Thus for $u\in \mathcal U_0\times \mathcal U_1,$ the Jacobi operator (as a function of $\ph \in \mathbb X^{\mu,s}_D$) is

\be L_{u}=L_\gamma + \frak d_{u}\label{Lxialphaphi}\ee

\noindent using Lemma \ref{Jacobioperatoredge}. The smallness of the coefficients of $\frak d_{u},$ combined with the estimates \eqref{higherorderelliptic} (which trivially hold equally well for $L_\gamma$) implies that \eqref{Lxialphaphi} is likewise injective with closed range. It retains an infinite-dimensional obstruction like $L_0.$ Let 

\begin{align} \text{\bf Ob}_u^s&:= \text{range}\left(L_\gamma + \frak d_{u}:\mathbb X^{\mu,s}_{D} \to \mathbb Y^{\mu,s}_{u}\right)^\perp\label{bbOdefpre} \\
\text{\bf Rg}_{u}^s&:= \text{range}\left(L_\gamma + \frak d_{u}:\mathbb X^{\mu,s}_{D} \to \mathbb Y^{\mu,s}_{u}\right)\nonumber 
\end{align}

\noindent be $r^\mu L^2$-orthogonal complement of the range of \eqref{Lxialphaphi}. Here, the subscript on $\mathbb Y^{\mu,s}$ denotes the fiber over that point in the trivialization of Lemma \ref{Ytrivialization} (recall $\mathbb Y^{\mu,s}_{u}\simeq r^\mu H^s_b(X_0;NX_0)$).

\begin{lm}  \label{obstructionbundlecor} The spaces \eqref{bbOdefpre} obey the following: 
\begin{enumerate}
\item[(A)] The fiberwise inclusion $\text{\bf Ob}_u^s \subseteq \mathbb Y^{\mu,s}$ is a locally trivial Banach subbundle for each $s\geq 0,$ and for $s=\infty$ is a locally trivial tame Fr\'echet subbundle. 
\item[(B)] The orthogonal projections to $\text{\bf Ob}_u^s,(\text{\bf Ob}_u^s)^\perp$ at $u$ are given by

\[ \Pi_{u}=\text{Id}- L_{u} \mathcal P_{u} L_{u}^\star. \hspace{3cm} \Pi_{u}^\perp=L_{u} \mathcal P_{u} L_{u}^\star,\]

\smallskip

\noindent where ${\mathcal P_{u}:= (L_{u}^\star L_{u})^{-1}}.$ Moreover, the projections respect regularity in the sense that

\noindent ${\text{Im}(\Pi_{u} |_{\mathbb Y^{\mu,s}})=\text{\bf Ob}_u^s}$ for each $s\geq 0,$ and likewise for the orthogonal complement. 
\end{enumerate}
 
\end{lm}

\begin{proof} (A) As explained above, Lemma \ref{Jacobioperatoredge} implies that $L_{u}$ form a fiberwise injective family with closed range, obeying uniform elliptic estimates. Thus $L_{u}$ itself provides a trivialization of 

\[(\text{\bf Ob}_u^s)^\perp=\bigsqcup_{u}\text{range}(L_{u})\]

\noindent modeled on $\mathbb X^{\mu,s}_D=r^{\mu+2}H^{s,2}_{b,e,D}(X_0; NX_0).$ Since $\mathbb Y^{\mu,s}$ is a Hilbert bundle, it follows that the orthogonal complement $\text{\bf Ob}_u^s$ is also locally trivial. (B) is precisely the ``moreover'' statement in Lemma \ref{secondorderinvertibility}. Finally, that the projection preserves regularity implies the $s=\infty$ version of (A). 
\end{proof}

The above justifies, 

\begin{defn}  \label{obstructionbundleanalyticdef} The bundle 

\[  {\text{\bf Ob}}^s\to \mathcal U_0\times \mathcal U_1.\]

\noindent whose fibers are \eqref{bbOdefpre} is called the {\bf obstruction bundle} of regularity $s\geq 0.$
\end{defn}

\medskip 

Recall from Section \ref{smoothingprojectionsection} that for $(\xi,\alpha,\Phi)=0,$ the projections to $\text{\bf Ob}_0$ enjoyed smoothing properties depending on the effective order of vector fields. The main result of this section, Proposition \ref{UGBstable}, shows that these smoothing properties persist across $\mathcal U_0\times \mathcal U_1.$

\subsection{Characterization of the Obstruction}
We now state two technical results used to prove Proposition \eqref{UGBstable}, whose proofs occupy the majority of Section \ref{obstructionperturbations}. 

For the duration of the section, we work in arclength coordinates $\tau=\tau(u)$ on the $S^1$-factor. Thus we have spaces $
H^s(S^1_\tau; \mathcal C_0)$ and $r^\mu H^s_{b}(X^\tau_0; NX_0)$ defined using the volume forms and derivatives as in Corollary \ref{taunorms}, which are smoothly and tamely equivalent to those defined using the coordinate $t$. Fourier modes are now written $e^{i\ell \tau},$ where it is understood that $\ell \in (1+\gamma_u)\Z$ (though we still tactily write $\ell \in \Z$).  In particular, there is a subbundle 

\[\text{\bf Ob}_\gamma^s:= \text{Im}\Big(\text{ob}_\gamma: H^s(S^1_\tau;\mathcal C_0) \lre r^\mu H^s(X_0^\tau; NX_0)\Big)=\text{range}(L_\gamma)^\perp  \]
\noindent given by the image of the smooth tame family of maps which are the map from Definition \ref{obbasisdef}, but using Fourier modes in $\tau.$ We use $\Psi_{j\ell}^\circ:= \Psi_{j\ell}$
to denote the corresponding obstruction basis from Definition \ref{obbasisdef}, also in arclength Fourier modes. Note that $\text{\bf Ob}_\gamma^s$ is ``nearly'' a product bundle, insofar as it depends only on the 1-dimensional parameter $\gamma_u.$

We now state the two technical lemmas, which generalize the obstruction basis from Definition \ref{obbasisdef} to a Resiz basis of the obstruction bundle (recall that a {\it Reisz basis} $\{e_k\}$ of a Hilbert space $H$  is a generalization of an orthonormal basis such that the norm  $\|\sum_k a_k e_k\|_H$ is bounded above and below by multiples of  $\|\{a_k\}\|_{\ell^2}$, rather than being equal to it). The Reisz basis is constructed as follows. Observe the two bundles $\text{\bf Ob}^s_\gamma$ and $\text{\bf Ob}^s$ coincide at $0 \in \mathcal U_0\times \mathcal U_1$, where they are both simply $\text{\bf Ob}^s_0.$ Lemma \ref{obstructionbundlecor} therefore implies that the orthogonal projections in both directions are isomorphisms, after possibly reducing the size of the neighborhoods. In particular, 

\be \Pi_\gamma^{-1}: {\text{\bf Ob}}^s_\gamma \to {\text{\bf Ob}}^s \label{projection}\ee

\noindent is an isomorphism. For $\Psi \in {\text{\bf Ob}}^{s}_\gamma$, the image is  $\Pi_\gamma^{-1}(\Psi)=\Psi + \psi$ where $\psi$ is the unique solution of 

\be (L_\gamma^\star + \frak d_{u}^\star) \psi = -\frak d_{u}^\star \Psi  \hspace{1cm}\psi \perp {\text{\bf Ob}}_\gamma. \label{Neumannadapted}\ee

\noindent with orthogonality again in $r^\mu L^2.$ This orthogonality condition is precisely the one that arises from solving (\refeq{Neumannadapted}) via Neumann series as perturbation of $L_\gamma^\star.$

\begin{prop}  \label{ReiszBasis}For each $u\in \mathcal U_0\times \mathcal U_1$, the projections 
\[ \Psi^u_{j\ell}:=\Pi_{\gamma}^{-1}(\Psi_{j\ell}^\circ)\in  \text{\bf Ob}_u,\]
\noindent are a Resiz basis depending smoothly and tamely on $u$ for fixed $(j,\ell)$, such that the following hold. 
\begin{enumerate} 
\item[(A) ]There is a linear map $\text{ob}_{u}$ that restricts to a bounded linear isomorphism 

\[\text{ob}_{u}: H^s(S^1_\tau; \mathcal C_0) \to {\text{\bf Ob}}_u^s \]

\noindent for all $s\geq 0$ that depends smoothly and tamely on $u.$ 

\item[(B)] The inverse is given by the sequence of inner products

\[\text{ob}_{u}^{-1}\circ \Pi_{u}(\psi)= \sum_{\ell,j}   \ \br  \psi \ , \   \Psi_{j\ell}  \ \kt_{r^\mu L^{2}(X_0^\tau)}  \cdot  e^{i\ell \tau}\phi_j \]
\noindent where $r^\mu L^2(X_0^\tau)$ denotes the inner product defined using the volume form $dV_\tau$ in $\tau$ arclength coordinates, and also depends smoothly and tamely on $u.$
\end{enumerate}

\end{prop}

The next result, which is also the main ingredient in the proof of the above Proposition \ref{ReiszBasis} is a characterization of the correction term for the projection $\Psi_{j\ell}^\circ \mapsto \Pi_{\gamma}^{-1}(\Psi_{j\ell}^u)$. Let $(\Psi^u_{j\ell})_h$ denote the $e^{ih\tau}$ Fourier coefficient of $\Psi_{j\ell}^u,$ which depends on $(r,x).$ Obviously $(\Psi_{j\ell}^\circ)_h=\delta_{h\ell}.$ Additionally, in a mild abuse of notation we simply write $|h|$ instead of $\smash{\br h\kt=\sqrt{1+h^2}}$ in places, with it being understood that we mean the regularized version when $h=0.$ For $u=0,$ note that the Fourier modes of $\Psi_{j\ell}^\circ$ are supported along the diagonal $\ell=h;$ the below says that, more generally, the Fourier support decays faster than polynomially away from this diagonal, and likewise for higher edge derivatives:

 \begin{prop} \label{Psidecomp} For each $u\in \mathcal U_0\times \mathcal U_1,$ and each $S,L\geq 0,$ there are constants $C_{M,L}$ such that

 \[\left\|(1+r^2 \Delta_\tau)^{M/2} \Big[(\Psi_{j\ell}^u)_h-\Psi_{j\ell}^\circ \Big]\right\|_{r^\mu L^2} \leq \frac{C_{M,L} \|u\|_{M+L+s_0}}{ |\ell| }\cdot  \frac{1}{|h-\ell|^L}. \]

 \noindent holds uniformly over $(j,\ell).$
 \end{prop}

 \noindent Note that $(1+r^2 \Delta_\tau)$ is the natural edge Laplacian in the tangential direction. The above is therefore a slightly more  quantitative version of the edge regularity statement for $\Psi_{j\ell}^u$ from \cite{Mazzeo1991EdgeOperators}.

 \medskip

Assuming the two proposition above momentarily, we now state and prove the main result of this section, which is the following stability result for the smoothing properties of the obstruction projection from Section \ref{smoothingprojectionsection}.

\begin{prop} \label{UGBstable} The ``unexpectedly'' good boundedness of the projections from Lemma  \ref{projectionsmoothing} is (tamely) uniform over $u\in \mathcal U_0\times \mathcal U_1.$ That is, if $\Psi$ is a term of effective order $m$ with $f(t)\in H^s(S_\tau^1;\R)$ in the notation of  \eqref{weakasympotitic}, then $ \text{ob}_u^{-1}\circ \Pi_u(\Psi) \in H^{s-m}(S^1_\tau;\mathcal C_0),$ and  

\[\| \text{ob}_u^{-1}\circ \Pi_u(\Psi)\|_{H^{s-m}} \leq C \|f\|_{H^s(S^1_\tau;\R)}.\]

\noindent In particular, there is a gain in tangential regularity of order $\alpha + (\tfrac{n}{2}-\mu)$. The same holds in the case of a sum of the form \eqref{effectiveordersum}. 
\end{prop}

\begin{proof} Let $\Psi_{j\ell}^u=\Psi_{j\ell}^\circ + \psi_{j\ell} + \frak e_{j\ell}$ where 
\be \psi_{j\ell}=\mathbf 1_{\big\{|h-\ell|<\tfrac{|\ell|}{2} \big\}} \Big[(\Psi_{j\ell}^u)_h-\Psi_{j\ell}^\circ \Big] \hspace{2cm}  \frak e_{j\ell}=\mathbf 1_{\big\{|h-\ell| \geq \tfrac{|\ell|}{2}\big\}}\Big[(\Psi_{j\ell}^u)_h-\Psi_{j\ell}^\circ \Big], \label{psiedef}\ee

\noindent where $\mathbf 1_A$ denotes the Fourier projection to a range of modes $A$. Thus we may write $\text{ob}_u^{-1}\Pi_u=\text{ob}_\gamma^{-1}\Pi_\gamma + \Pi_{\psi} + \Pi_{\frak e}$ where 
\begin{align}
\Pi_\psi(\Psi)&=\sum_{j\ell} \br\Psi , \psi_{j\ell} \kt_{r^\mu L^2(X_0^\tau)} \cdot e^{i\ell \tau}\phi_j \\
\Pi_{\frak e}(\Psi)&= \sum_{j\ell} \br\Psi , \frak e_{j\ell} \kt_{r^\mu L^2(X_0^\tau)} \cdot e^{i\ell \tau}\phi_j, 
\end{align}
\noindent after which it suffices to show the desired estimate for the latter two projections individually, as it is true for $\text{ob}_\gamma^{-1}\Pi_\gamma$ by Lemma  \ref{projectionsmoothing}.

Let $\Psi=r^p (\nabla_\tau^k f(t))\phi_a$ be as in \eqref{weakasympotitic}, and $\mathcal F_\ell=\mathbf 1_{\big\{|h-\ell|<\tfrac{|\ell|}{2} \big\}}$ be the Fourier support of $\psi_{j\ell}.$ Then 

\begin{align*}
\br \Psi, \psi_{j\ell}\kt_{r^\mu L^2(X_0^h)}&=\sum_{h\in \mathcal F_\ell} h^k f_h \int_0^1 \int_{\Sigma}  r^p \br \phi_a, (\psi_{j\ell})_h \kt r^{-2\mu+n-1}dr dV_{\Sigma} \\
&\leq \sum_{h\in \mathcal F_\ell} h^k f_h \left( \int_0^1 \frac{r^{2p}}{(1+r^2|\ell|^2 )^{M}} r^{-2\mu+n-1}dr\right)^{1/2} \big\|(1+r^2|\ell|^2)^M (\psi_{j\ell})_h\big\|\\
& \leq \sum_{h\in \mathcal F_\ell} \frac{h^k f_h}{|\ell|^{p+n/2-\mu}} \left( \int_0^1 \frac{\rho ^{2p-2\mu + n-1}}{(1+\rho )^{2M}} d\rho \right)^{1/2} \big\|(1+r^2|\ell|^2)^M (\psi_{j\ell})_h\big\| \\
&\leq \frac{C_{M,2}\|u\|_{M+2+s_0}}{|\ell|^{k-m+1}} \sum_{h\in \mathcal F_\ell} \frac{h^k f_h}{ |h-\ell|^2  }
\end{align*}

\noindent since $m=k-(p+\tfrac{n}{2}-\mu).$ Here, we have used that $(1+r^2\Delta_\tau)\sim (1+r^2 |\ell|^2)$ on Fourier modes in $\mathcal F_\ell.$ Thus 

\begin{align}
 |\ell|^{s-m} |\br \Psi, \psi_{j\ell}\kt|
& \leq C_{M,2} \|u\|_{M+2+s_0}  |\ell|^{(s -2)+ (-k + m-1)} \sum_{h \in \mathcal F_\ell} \frac{|h|^{k} |f_h|}{|h-\ell|^2 } \nonumber\\
&\leq C_{M,2} \|u\|_{M+2+s_0} |\ell|^{-1} \sum_{h\in \mathcal F_\ell } \frac{|h|^{s}|f_h|}{|h-\ell|^2},  \label{extraell}
\end{align}
\noindent {Because $|\ell|/2 \leq |h|\leq 3|\ell|/2$ for $h\in \mathcal F_\ell.$ Summing over $\ell$, the last sum is the convolution $\{|h|^s |f_h|\} * \{h^{-2}\}.$ And $\{h^{-2}\}\in l^1(\Z;\R),$ thus using Young's convolution inequality $l^2(\Z;\R)\times l^1(\Z;\R)\to  l^1(\Z;\R)$ gives }
\begin{align*}
\|\Pi_{\psi}(\Psi)\|^2_{H^{s-m}} &\leq \sum_{j,\ell} |\br \Psi, \psi_{j\ell}\kt|^2  |\ell| ^{2s-2m} \\
& \leq C^2_{M,2} \|u\|^2_{M+2+s_0}  \Big \| \{|h|^s |f_h|\} * \{h^{-2}\}\Big \|^2_{l^2} \leq C_{M,2}^2 \|u\|_{M+2+s_0}^2 \|f\|_{H^s}^2.
\end{align*}
\noindent using that $|\ell|/2 \leq |h|\leq 3|\ell|/2$ for $h\in \mathcal F_\ell$ once again. 

The $\frak e$ term is comparatively straightforward. Since $\frak e_{j\ell}$ has Fourier supported where $|h-\ell|>|\ell|/2$ by definition, Proposition \ref{Psidecomp} with $M=0$ shows 
\be \|\frak e_{j\ell}\|_{r^\mu L^2} \leq \frac{C_{0,L}\|u\|_{L+s_0}}{|\ell|^{L+1}}.\label{eestimate}\ee

\noindent Taking $L>s-m+2,$ 

\begin{align}
\|\Pi_{\psi}(\Psi)\|^2_{H^{s-m}} &\leq \sum_{j,\ell} |\br \Psi, \psi_{j\ell}\kt|^2  |\ell|^{2s-2m} \\
&\leq \|\Psi\|^2_{r^\mu L^2}\sum_{j\ell} \|\frak e_{j\ell}\|^2_{r^\mu L^2}|\ell|^{2s-2m} \leq \|\Psi\|^2_{r^\mu L^2}\sum_{j\ell} \frac{1}{|\ell|^2} \leq C_{L}\|u\|^2_{L+s_0}\|f\|^2_{H^k} 
\end{align}
\noindent by Cauchy-Schwartz.

\end{proof}

\begin{rem}\label{UGBstablerem} We remark on two strengthenings of Proposition \ref{UGBstable}. First, the extra power of $\ell$ in \eqref{extraell} shows in fact that $\Pi_\psi(\Psi)\in H^{s-m+1},$ and likewise for $\Pi_{\frak e}(\Psi).$ (2) If $\psi_{j\ell}'=M\psi_{j\ell}$ where $M$ is an order $0$ Fourier multiplier in $\ell,$ then  
$\|\Pi_{M\psi}\|_{H^{s-m}} \leq C\|f|_{H^s}$ holds equally well for $M\psi_{j\ell}.$
\end{rem}

\subsection{Banded Decay}

We now prove Propositions \ref{ReiszBasis} and \ref{Psidecomp}. We have 

\be  (L_\gamma^\star  + \frak d_{u}^\star) \Big(\psi_{j\ell}+\frak e_{j\ell}\Big) = -\frak d_{u}^\star \Psi_{j\ell}^\circ  \hspace{1cm}\psi \perp {\text{\bf Ob}}_\gamma. \label{Neumannadapted2}\ee

\noindent by \eqref{Neumannadapted}. Proposition \ref{ReiszBasis} now follows from tame elliptic estimates for $ L_\gamma^\star  + \frak d_u^\star .$

\begin{proof}[Proof of Proposition \ref{ReiszBasis}]  The operator $L_\gamma^\star + \frak d_u^\star $ obeys elliptic estimate 

\be  \|v \|_{r^\mu L^2} \leq C_0 \|(L_\gamma^\star +\frak d_u^\star) v \|_{r^{\mu-2}H^{-2}_e},   \hspace{1cm} v \perp \text{\bf Ob}_\gamma\label{dualellipticestimate}\ee

\noindent by Lemma \eqref{obstructionbundlecor}, using the right inverse defined by solving $L_\gamma^\star + \frak d_u^\star=f$ via the Neumann series of $L_\gamma^\star,$ which is convergent for $\|u\|_{s_0}<\delta$ sufficiently small. For $\nabla_\tau=-i\del_\tau$, repeatedly applying $(\nabla_\tau - \ell)$ leads to the tame estimates 

 \smallskip 
 
\be  \|(\nabla_\tau -\ell)^sv \|_{r^\mu L^2} \leq C_s\Big( \big\| (\nabla_\tau-\ell)^s\big(L_\gamma+\frak d_u\big)^\star v\big\|_{r^{\mu-2}H^{-2}_e}  \ + \ \|u\|_{s+s_0}\big\|(L_\gamma+\frak d_u)^\star v\big\|_{r^{\mu-2} H^{-2}_e}\Big). \label{taemH2}\ee

\smallskip

\noindent via iterating commutators and interpolation inequalities in the standard way (cf. Lemma \ref{tameL1} in the next subsection). Note here that $(\nabla_\tau-\ell)^s$ preserves the condition that $v\perp \text{\bf Ob}_\gamma,$ as can be seen by integration by parts using Corollary  \ref{firstob}(B). Additionally, we have used that  
$[(\nabla_\tau-\ell), L_\gamma^\star ]=0$ and $[\ell, \nabla_\tau^k\frak d_u^\star]=0.$

In the particular case of \eqref{Neumannadapted2}, shows 

\be  \big\|(\nabla_\tau -\ell)^s(\psi_{j\ell} + \frak e_{j\ell}) \big\|_{r^\mu L^2} \leq C_s\Big( \big\| (\nabla_\tau-\ell)^s\frak d_u^\star \Psi_{j\ell}^\circ \big\|_{r^{\mu-2}H^{-2}_e}  \ + \ \|u\|_{s+s_0}\big\|\frak d^\star_u \Psi_{j\ell}^\circ\big\|_{r^{\mu-2} H^{-2}_e}\Big). \label{edgetame}\ee

\noindent Applying $(\nabla_\tau-\ell)^s$ in the middle term, the product rule term 

\[ \sum_{0\leq k \leq s}\binom{s}{k} (\nabla_\tau^k\frak d_u^\text{low})^*(\nabla_\tau-\ell)^{s-k} \Psi_{j\ell}^\circ =  (\nabla_\tau^{s}\frak d_u^\star) \Psi_{j\ell}^\circ\]

\noindent collapses, because $(-i\del_\tau-\ell)\Psi_{j\ell}^\circ=0.$ We conclude that 

\begin{align*}
\big\|(\nabla_\tau -\ell)^s(\psi_{j\ell} + \frak e_{j\ell}) \big\|_{r^\mu L^2} &\leq C_s \Big(\|(\nabla_\tau^s \frak d_u^\star)\Psi_{j\ell}^\circ\|_{r^{\mu-2}H^{-2}_e} \ + \ \|u\|_{s+s_0}\|\frak d_u^\star\Psi_{j\ell}^\circ\|_{r^{\mu-2}H^{-2}_e} \Big)\\
&\leq \frac{C_s\|u\|_{s+s_0}}{|\ell|},
\end{align*}

\noindent where in the last line, the fact that $\frak d_u^\star, \nabla_\tau^s \frak d_u^\star = O(r)$ yields an extra factor of $|\ell|^{-1}$ when integrating, similar to the proof of Lemma \ref{projectionsmoothing}. In the $h$ Fourier mode, $(\nabla_\tau-\ell)=(h-\ell),$ and dividing by this factor yields the proof for $M=0$ and $L=s.$

For $M\geq 1,$ repeating the same tame estimates as \eqref{edgetame} for the higher edge spaces

\begin{align*}\big\|(\psi_{j\ell} + \frak e_{j\ell}) \big\|_{r^\mu H^M_e} &\leq C_M \Big(\| \frak d_u^\star\Psi_{j\ell}^\circ\|_{r^{\mu-2}H^{M-2}_e} \ + \ \|u\|_{M+s_0}\|\frak d_u^\star\Psi_{j\ell}^\circ\|_{r^{\mu-2}H^{-2}_e} \Big)\\
&\leq \frac{C_M\|u\|_{s+s_0}}{|\ell|}.
\end{align*}

\noindent In particular, 

\[C_M \big\|(1+r^2\Delta_\tau)^{M/2}(\psi_{j\ell}+\frak e_{j\ell})\big\|_{r^\mu L^2} \leq \frac{C_M \|u\|_{s+s_0}}{|\ell|}. \]

\noindent  This  is the result when $M>0$ and $L=0.$ Repeating the tame estimates \eqref{taemH2} in the $r^\mu H^M$-norm then yields the general case, using the $\|u\|_{M+L+s_0}$-norm on the right hand side.

\end{proof}

\begin{rem}
\label{smoothprojectionsrem}

Proposition \ref{Psidecomp} is ultimately a slightly more quantitative version of the standard edge-regularity statement for solutions of an elliptic edge operator. 
A similar type of regularity gain phenomenon appears in \cite[Section 5]{MazzeoEdgeOperatorsII}. More generally, such regularity statements can be deduced by considerations of the Schwartz kernel of the projection $\Pi_0.$ It is evident from the structure of the Schwartz kernel of $\Pi_0$ proved in that paper that $\Pi_0(\Psi)$ 
can be differentiated as many times as desired with respect to the edge vector fields $r\partial_r$, $r\partial_t$ and $\partial_{x_i}$ with no 
loss of regularity or decay rate. Said differently, repeated application of these vector fields yields a sequence of operators, each of which 
have the same $L^2$ boundedness properties as $\Pi_0$ itself. Thus if $\Psi \in r^\mu L^2$, then $\Pi_0(\Psi) \in r^\mu H^\infty_e$.
The composition $\partial_t \Pi_0$ is singular, and introduces a drop by $1$ in the weight factor: $\del_t \Pi_0 (\Psi) \in r^{\mu-1} H^\infty_e$.
However, if $\Psi \in r^{\mu+1}L^2$, for example, then this extra factor of $r$ in $\Psi$ can be combined with $\del_t$ to produce
an edge derivative, so that $\partial_t \Pi_0(\Psi) \in r^{\mu} H^\infty_e$, or equivalently $\Pi_0 (\Psi) \in r^\mu H^1_b \cap  r^\mu H^\infty_e$. 
A simple argument extends this to when the extra rate of vanishing of $\Psi$ is fractional, and does not necessarily match the number
of $t$ derivatives.  It then requires extra work to  characterize the precise regularity of the asymptotics as $r \to 0$ of this projection
to obtain the sharp regularity statement of Proposition \ref{UGBstable}. This Schwartz kernel perspective likely generalizes much more readily, e.g. to the case that the $S^1$ edge is replaced by $T^k.$ 
\end{rem}

The proof of Proposition \ref{ReiszBasis} is now follows from pre-composing with the unitary transformation re-orthogonalizing the Reisz basis $\Psi^u_{j\ell}.$

\begin{proof}[Proof of Proposition \ref{ReiszBasis}]
Define \[\text{ob}_{u}^{-1}(\Psi)= \sum_{\ell,j}   \ \br  \Psi \ , \   \Psi_{j\ell}  \ \kt_{r^\mu L^{2}(X_0^\tau)}  \cdot  e^{i\ell \tau}\phi_j \]

\noindent so that the conclusion of (B) holds by fiat. We show this is actually invertible, with inverse obeying (A). First, let
\begin{align*}
\text{ob}^\text{pre}_{u}: L^2(S^1_\tau; \mathcal C_0) &\lre \text{\bf Ob}_u
\end{align*}
\noindent be the linear extension of $e^{i\ell \tau}\phi_j \mapsto \Psi_{j\ell}^u.$ We claim that 
\smallskip
\be \text{ob}_{u}^{-1} \circ \text{ob}^\text{pre}_{u}=\text{Id} + A_{u} \hspace{1.5cm} \text{where}\hspace{1.5cm} \|A_{u}\|_{L^2 \to L^2} \leq C\delta, \label{obclaim}\ee
\smallskip

\noindent and $A_u$ is zeroth order operator (thus respects regularity, tamely). In particular, $\text{Id}+ A_{u}$ is invertible, provided $\delta$ is sufficiently small, and in this case $\text{ob}_u:= \text{ob}_u^{\text{pre}}\circ (\text{Id} + A_u)^{-1}$ obeys the conclusions of both (A) and (B). It therefore suffices to establish  (\refeq{obclaim}). 

On basis elements, we calculate directly: 

\begin{align*}
\text{ob}_u^{-1}\circ \text{ob}_u^\text{pre}(e^{i\ell \tau}\phi_j)&= \sum_{i,k} \br \Psi_{j\ell}^u \ , \ \Psi_{ik}^u\kt_{r^\mu L^2(X_0^\tau)}\\
&= \sum_{i,k} \br \Psi_{j\ell}^\circ + \zeta_{j\ell} \ , \ \Psi_{ik}^\circ + \zeta_{ik}\kt_{r^\mu L^2(X_0^\tau)}= \delta_{ij}\delta_{k\ell} + \sum_{i,k} A_{j\ell, ik} \cdot e^{ik\tau}\phi_i.
\end{align*}
\noindent where $A_{j\ell, ik}=\br \zeta_{j\ell} , \zeta_{ik}\kt_{r^\mu L^2(X_0^\tau)}.$ Indeed, the cross-terms cancel by the orthogonality \eqref{Neumannadapted}. Expanding further, 

\begin{align*}
\sum_{k,m} \Big | A_{j\ell, km} \Big |^2 &\leq \sum_{k,m} | \br \psi_ {j\ell}, \psi_{km}\kt  |^2 \ + \ \sum_{k,m} | \br \psi_ {j\ell}, \frak e_{km}\kt  |^2   \ + \ \sum_{k,m} | \br \frak e_{j\ell}, \psi_{km}\kt  |^2  \ + \ \sum_{k,m} | \br \frak e_{j\ell}, \frak e_{km}\kt  |^2 \label{threematrixelements}, 
\end{align*}

\noindent where all inner products are in $r^\mu L^2(X_0^\tau).$  Cauchy-Schwartz and the bounds from Proposition \ref{Psidecomp} with $M=L=0$ show that each term is bounded by at most ${C\delta}{|k|^{-2} |\ell|^{-2}}$. $A_{j\ell, km}$ is therefore a matrix whose columns are square summable with size at most $\delta$. (\refeq{obclaim}) follows, where $A_{u}$ the corresponding operator. It remains to show that $A_u$ preserves regularity; in terms of the matrix this is the statement that 

\[\sum_{k,m} \Big | A_{j\ell, km} \Big |^2 |k|^{2s} \leq C_s |\ell|^{2s}.\]

\noindent This follows from the Fourier mode restriction on $\psi_{j\ell}$, the bounds and bounds \eqref{eestimate},  along with integration by parts. For the third term, specifically, we use that the statement of Proposition \ref{Psidecomp} and induction on $s$ imply $\|\nabla_\tau^s \frak e_{j\ell}\|\leq C_s|\ell|^s\|\frak e_{j\ell}\|.$ 

\end{proof}

%% file: sections/sec_tame_estimates.tex
\section{Tame Estimates}
\label{tameestimate}
In this final subsection, we prove Theorem \ref{maina} by showing that the mean curvature section in Definition \ref{smoothtameMC} obeys the hypotheses of the Nash-Moser Implicit Function Theorem \ref{nashmoser}.

The first was already established in Corollary \ref{tameMC}:
\begin{lm} $\mathbb Q$ obeys Hypothesis (I) of Theorem \ref{nashmoser}, i.e. it is a smooth tame section. 
\label{hyp1verification}
\end{lm}

\subsection{Stability of Invertibility}
The verification of Hypothesis (II) is effectively immediate from the preparation in Sections \ref{deformationoperatorsection} and \ref{obstructionperturbations}. Using the notation from \eqref{bbXdef}, we have:

\begin{lm} $\mathbb Q$ obeys Hypothesis (II) of Theorem \ref{nashmoser},i.e. after a possibly reduction of $\mathcal U_0\times \mathcal U_1,$ the family of linearizations 

\smallskip

\[\mathbb L_{\xi,\alpha,\Phi}: \mathbb D^\infty(S^1;\mathcal B_0)\oplus \mathbb X^{\mu,\infty}_D(X_0; NX_0) \lre \mathbb Y^{\mu,\infty}(X_0;NX_0)\]

\smallskip

\noindent is invertible for $(\xi,\alpha,\Phi)\in \mathcal U_0\times \mathcal U_1.$
\label{hyp2verification}
\end{lm}

\begin{proof} Let $u=(\xi,\alpha,\Phi)\in \mathcal U_0\times \mathcal U_1$ be as above. In the trivialization by arclength from Corollary \ref{taunorms}, the linearization is a map 

\[\mathbb L_{u}: H^{\infty}(S^1_\tau; \mathcal B_0) \oplus r^{\mu+2}H^{\infty,2}_{b,e,D}(X_0^\tau; NX_0) \lre r^\mu H^\infty_b(X_0^\tau; NX_0).\]

\noindent Splitting the codomain $r^\mu L^2 H^\infty_b(X_0^\tau; NX_0)=\text{\bf Ob}_{u}^0 \oplus \text{\bf Rg}_{u}^0$ as in \eqref{bbOdefpre}, and composing the first factor with $\text{ob}_{u}^{-1}$ this becomes the block lower-triangular operator

\be
\mathbb L_u(\eta,\beta,\ph)=\begin{pmatrix} T_{u} & & 0  \\  \\ B_{u}  & &  L_{u} \end{pmatrix} \begin{pmatrix} (\eta,\beta) \\  \\ \ph \end{pmatrix} : \begin{matrix}  H^{\infty}(S^1_\tau ;\mathcal B_0)\\ \oplus \\ r^{\mu+2}H^{\infty,2}_{b,e,D}(X^\tau_0;NX_0) \end{matrix} \ \ \lre \ \ \begin{matrix} H^\infty(S^1_\tau;\mathcal C_0) \\ \oplus   \\   \text{\bf Rg}^\infty_{u}\end{matrix}
\ee
\noindent generalizing \eqref{extendedL0}, where 
\begin{align*}
T_{u}(\eta,\beta)&:= \text{ob}_{u}^{-1}\circ \Pi_{u}\circ (\text{d}\mathbb Q)_{u}(\eta,\beta,0)\\
B_{u}(\eta,\beta)&:= \ \hspace{1.0cm}\Pi^\perp_{u} \circ (\text{d}\mathbb Q)_{u}(\eta,\beta,0).
\end{align*}
\noindent so that $T_u=T_0$ is the product deformation operator (Definition \ref{T0def}) at the origin. 

Proposition \ref{generalmeancurvature} shows that $(\text{d}\mathbb Q)_u(\eta,\beta,0)$ consists of terms of leading effective order $(1+\epsilon, \epsilon)$ depending smoothly and tamely on $u\in \mathcal U_0\times \mathcal U_1,$ depending on up to $\eta''', \beta''.$ Proposition \ref{UGBstable} therefore shows that 

\medskip 
\be
\mathbb L_u(\eta,\beta,\ph)=\begin{pmatrix} T_u &  & 0   \\  \\ B_u  & &  L_u \end{pmatrix} \begin{pmatrix} (\eta,\beta) \\  \\ \ph \end{pmatrix} : \begin{matrix}   (H^3 \oplus H^2)\\ \oplus \\ r^{\mu+2}H^{2}_{e,D}(X_0^\tau;NX_0) \end{matrix} \ \ \lre \ \ \begin{matrix}   H^{2-\epsilon}(S^1_\tau;\mathcal C_0)\\ \oplus   \\ \text{\bf Rg}_u\end{matrix}\label{extendedL0prime}
\ee

\bigskip 

\noindent is bounded, where $H^3\oplus H^2=H^3(S^1_\tau;\R^N)\oplus H^2(S^1_\tau;\R^K).$ Moreover, since $\text{ob}_u^{-1}, \Pi_u, (\text{d}\mathbb Q)_u$ all depend smoothly and tamely on $u,$ so does each entry of $\mathbb L_u.$ \eqref{extendedL0prime} is therefore, {\it a fortiori} a continuous family of bounded operators on fixed Banach vector bundles, which coincides with the invertible $\mathbb L_0$ (Corollary \ref{invertibleLbb0}) at $u=0.$ It is therefore invertible, provided that the radius $\delta$ of $\mathcal U_0, \mathcal U_1$ is sufficiently small that $ \delta \|\mathbb L_0^{-1}\|<1$ in the norms of \eqref{extendedL0prime}.

\end{proof}

\subsection{Tame Estimates}
 We now verify Hypothesis (III) by showing that $\mathbb L_u^{-1}$ is tame, which is equivalent to tame elliptic estimates. The tame estimates for $\mathbb L_u$ follow from the tame estimates for the individual diagonal entries. 

 For convenience, we shift the regularity grading to coincide with finite shift in \eqref{extendedL0prime}. Thus set 

 \begin{align}
 \widetilde {\mathbb X}^{\mu,s}_D&:=r^{\mu+2} H^{s,2}_{b,e,D}(X_0^\tau; NX_0)\nonumber \\
 \widetilde {\mathbb D}^s(S^1_\tau;\mathcal B_0)&:= H^{s+3}(S^1_\tau;\R^N)\oplus H^{s+2}(S^1_\tau; \R^N) \label{tameshifts}\\
 \widetilde { \mathbb O}^s&:= H^{s+2-\epsilon}(S^1_\tau;\mathcal C_0) \nonumber\\
  \widetilde {\mathbb Y}^{\mu,s}_{\text{Rg}}&:= \text{\bf Rg}_u \cap r^\mu H^s_b(X_0^\tau; NX_0). \nonumber 
 \end{align}

 \noindent There equivalence with the un-tilded versions is a smooth tame equivalence of Fr\'echet spaces given by the finite regularity shifts in each factor. We write $\|-\|_s$ for the tilded norms in each factor.

 \begin{lm}  \label{tameL1}For $\delta>0$ sufficiently small and $s_0\in \N$ sufficiently large, and for  $u=(\xi,\alpha,\Phi) \in \mathcal U_1 \times \mathcal U_2$ with $\|u\|_{s_0}< \delta$, $L_{u}$ is injective and obeys the following tame estimates: there are constants $C_s$ such that

 \smallskip
\be
\|\ph\|_{s}\leq C_s \ \Big( \| L_{u} (\ph)\|_{s} +\|u\|_{s+s_0}\|L_{u} (\ph)\|_{0}\Big), \label{tameLPhixi}
\ee

\smallskip
\noindent hold for $s\geq 0.$ For $s=0$, the second term on the right may be omitted.   

\end{lm}

\begin{proof}
Lemma \ref{Jacobioperatoredge} and Corollary \ref{higherorderelliptic} together imply the $s=0$ estimate

\[ \|\ph\|_{0} \leq C_0 \|L_{u}\ph\|_0. \]

\noindent which in turn implies injectivity. 

Proceeding now by induction, assume that (\refeq{tameLPhixi}) holds for $s$, and we will show the same holds for $s+1$. Differentiating for a multi-index of rank $|I|=s+1$, one has
\[ L_u(\nabla_b^I \ph )=\nabla_b^I (L_u  \ph)  \  -  \   \sum_{\substack{I_1+I_2=I \\  |I_1|\geq 1}} (\nabla_b^{I_1}L_\gamma) (\nabla_b^{I_2}\ph)  \  - \   \sum_{\substack{I_1+I_2=I \\  |I_1|\geq 1}}(\nabla_b^{I_1}\frak d_u) (\nabla_b^{I_2}\ph).\]

\noindent For each term, using $\frak a_{u}$ to denote the coefficient functions of $\frak d_u$ from Lemma \ref{Jacobioperatoredge}, then for $k=|I_2|$ 
\begin{align}
\big\| (\nabla_b^{I_1}\frak d_u) (\nabla_b^{I_2}\ph)\big\|_{0} &\leq  C\big\|\frak a_{u}\big\|_{r^{\mu+1}H^{k+1+n}}\big\|\ph\|_{s-k} \nonumber \\
& \leq C_s \big\| u\big\|_{k +1 +n }\big\|\ph\|_{s-k}, \label{tamepickup}
\intertext{where we have used the multiplication of Lemma  \ref{multiplications}(C), and the bounds of Lemma \ref{Jacobioperatoredge} respectively. Interpolation (Lemma  \ref{multiplications}(A))followed by Young's inequality with $p=\tfrac{s}{k}$ and $q=\tfrac{s}{s-k}$ then show that }
& \leq C_s \|u\|_{n+1}^{1-\tfrac{k}{s}}\cdot \|u\|_{s+1+n}^{\tfrac{k}{s}} \cdot  \|\ph\|_{0}^{\tfrac{k}{s}}\cdot \|\ph\|_{s}^{1-\tfrac{k}{s}} \nonumber \\
&\leq C_s \Big( \|u\|_{s+1+n} \|\ph\|_0 \ + \  \|u\|_{1+n} \|\ph\|_s \Big) \nonumber. 
\end{align}
\noindent Applying the $L_0$ estimate yields 
\begin{align}
\|\nabla_b^I \ph\|_{0} &\leq C_0 \Big \|\nabla_b^I (L_u  \ph)  \  -  \   \sum_{\substack{I_1+I_2=I \\  |I_1|\geq 1}} (\nabla_b^{I_1}L_\gamma) (\nabla_b^{I_2}\ph)  \  - \   \sum_{\substack{I_1+I_2=I \\  |I_1|\geq 1}}(\nabla_b^{I_1}\frak d_u) (\nabla_b^{I_2}\ph) \Big \|_0 \nonumber \\ 
\|\nabla_b^I \ph\|_{0} &\leq C_0 \Big \|\nabla_b^I (L_u  \ph)  \Big \| \  +  \ C_0 \Big \|   \sum_{\substack{I_1+I_2=I \\  |I_1|\geq 1}} (\nabla_b^{I_1}L_\gamma) (\nabla_b^{I_2}\ph)  \  - \   \sum_{\substack{I_1+I_2=I \\  |I_1|\geq 1}}(\nabla_b^{I_1}\frak d_u) (\nabla_b^{I_2}\ph) \Big \|_0. \label{withC0interp}  \\ 
\intertext{\noindent Summing over multi-indices $I$ and substituting the above interpolation yields}
\|\ph\|_{s+1}& \leq C_{s+1} \ \Big( \| L_u (\ph)\|_{s+1} +\|u\|_{s+1+n}\| \ph\|_{0} \ + \ \|u\|_{1+n}\| \ph\|_{s} \Big), \label{tameinduction} 
\end{align}
\noindent{and substituting the induction hypothesis and the $s=0$ estimate complete the proof, for $s_0 \geq n+1$.}
\end{proof}

\bigskip 

Next, we have the similar tame estimates for $T_u.$ Regularity on the codomain here means regularity in $ H^s(S^1_\tau;\mathcal C_0).$ By Proposition \ref{ReiszBasis}, $\text{ob}_u$ and its inverse induce an equivalence of norms; they do not, {\it however}, intertwine the derivative, i.e. $\del_\tau \text{ob}_u(\zeta(\tau))\neq \text{ob}_u(\del_\tau \zeta(\tau)).$ By definition, the derivatives $\del_\tau$ on $S^1_\tau$ yield powers of $|\ell|,$ thus the norm on the codomain is the weighted sequence of inner products

\be \|T_u(\eta,\beta)\|_{H^s}:= \sum_{j,\ell}\Big | \big \br  \ (\text{d}\mathbb Q)_u(\eta,\beta,0) \ , \ \Psi_{j\ell}^u \big \kt_{r^\mu L^2(X_0^\tau)} \Big |^2 \cdot |\ell|^{2s}. \label{derivdiff}\ee

\noindent The proof relates these to the ambient $\del_\tau$ derivatives on $X_0^\tau$ using Proposition \ref{Psidecomp}.

\begin{lm} \label{tameT2} For $\delta>0$ sufficiently small and $s_0\in \N$ sufficiently large, and for  $u=(\xi,\alpha,\Phi) \in \mathcal U_1 \times \mathcal U_2$ with $\|u\|_{s_0}< \delta$, $T_u$ obeys the following tame estimates: there are constants $C_s$ such that

\smallskip
\be
\|(\zeta,\beta)\|_{s}\leq C_s \ \Big( \| T_{u} (\eta,\beta)\|_{H^{s+2-\epsilon}(S^1_\tau)} +\|u\|_{s+s_0}\|T_{u} (\eta,\beta)\|_{H^{2+\epsilon}(S^1_\tau)}\Big) \label{Ttameinv},
\ee

\smallskip

\noindent hold for $s\geq 0.$ For $s=0$, the second term on the right may be omitted.   

\end{lm}

\begin{proof} Using Proposition   \ref{ReiszBasis}(B) and the notation \eqref{psiedef}, we may write
\begin{align*}
T_u(\eta,\beta)&=\text{ob}_{u}^{-1} \ \circ \  \Pi_u \ \circ \  (\text{d}\mathbb Q)_u(\eta,\beta,0)\\ 
&= \text{ob}_\gamma^{-1}\circ \Pi_\gamma \circ (\text{d}\mathbb Q)_u(\eta,\beta,0) \ + \ \Pi_\psi(\text{d}\mathbb Q)_u(\eta,\beta,0) \ + \ \Pi_{\frak e}\circ (\text{d}\mathbb Q)_u(\eta,\beta,0)\\
&:=\text{ob}_\gamma^{-1}\Pi_\gamma \circ (\text{d}\mathbb Q)_u(\eta,\beta,0) \ + \ \frak t_1(\eta,\beta) \ + \ \frak t_2(\eta,\beta),
\end{align*}

\noindent where the latter two terms are the definition of $\frak t_1,\frak t_2.$ Note the first term is not simply the product deformation operator; rather, it is the non-product mean curvature expression from Proposition \ref{generalmeancurvature} projected to the (near) product obstruction $\text{\bf Ob}_\gamma.$

 We write $Q_u(\eta,\beta)=(\text{d}\mathbb Q)_u(\eta,\beta,0)$ for shorthand for the expression from Proposition \ref{generalmeancurvature}, and use $\nabla_\tau=-i\del_\tau$ to avoid tracking factors of $i$. Differentiating shows 

 \begin{align}
T_{u}(\nabla^{s+1}_\tau \eta, \nabla^{s+1}_\tau \beta )&= \nabla^{s+1}_\tau (T_{u}(\eta,\beta)) \ - \  \sum_{0\leq k\leq s}(\nabla_\tau^{k+1} T_{u})(\nabla_\tau^{s-k}\eta, \nabla^{s-k}_\tau \beta) \label{Ttamederivs}\\
 & \ - \  [\nabla^{s+1}_\tau, \frak t_1](\eta,\beta) \ - \  [\nabla^{s+1}_\tau, \frak t_2](\eta,\beta)  \nonumber 
\end{align}
\noindent wherein, 
\begin{align}
(\nabla_\tau^{k+1} T_{u})(\nabla_\tau^{s-k}\eta, \nabla_\tau^{s-k}\beta)&:=  \sum_{j\ell}\Big \br (\nabla_\tau^{k+1}Q_u)(\nabla^{s-k}_\tau\eta, \nabla^{s-k}_\tau\beta) \ , \  \Psi^u_{j\ell}\Big \kt  \label{ttameA}\\[0.3cm]
[\nabla^{s+1}_\tau, \frak t_1](\eta,\beta)&= \sum_{j\ell} \Big\br Q_u(\nabla^{s+1}_\tau\eta, \nabla^{s+1}_\tau\beta) \ , \   (M^{s+1}-1) \psi_{j\ell}\Big\kt  \label{ttameB} \\[0.3cm]
& \ + \ \sum_{j\ell} \Big \br (\nabla_\tau^{k+1}Q_u)(\nabla^{s-k}_\tau\eta, \nabla^{s-k}_\tau\beta) \ , \   (M^{s+1}-1) \psi_{j\ell}\Big \kt \nonumber  \\[0.3cm]
[\nabla^{s+1}_\tau, \frak t_2](\eta,\beta)&:=\sum_{j\ell} \Big \br Q_u(\eta,\beta) \ , \  (\ell^{s+1}-\nabla^{s+1}_\tau ) \frak e_{j\ell}\Big \kt \label{ttameC}. 
\end{align}
\noindent In the above, $M$ is the uniformly (in $\ell$) bounded Fourier multiplier  $\psi_{j\ell}= \frac{\nabla_\tau}{|\ell|} M\psi_{j\ell},$ due to the mode restriction in Proposition \eqref{Psidecomp}(A). The second two are not the commutators on the nose, but instead arise from the difference from the failure of $\text{ob}_u$ to intertwine derivatives, as explained above (and are thus morally analogous to commutators). Indeed, to arrive at the above beginning from \eqref{derivdiff}, we have

\begin{align*}
\nabla^{s+1}_\tau T_u(\eta,\beta)&= \sum_{j,\ell} \br Q_u(\eta,\beta,0) \ , \ \Psi^u_{j\ell} \kt \cdot \ell^{s+1} \\ 
&= \sum_{j,\ell} \left[ \br  Q_u(\eta,\beta) \ , \ \nabla_\tau^{s+1} \Psi^\circ_{j\ell} \kt \ + \  \br   Q_u(\eta,\beta) \ , \ \nabla_\tau^{s+1}\psi_{j\ell} \kt    \ +  \  \br    Q_u(\eta,\beta) \ , \ \nabla_\tau^{s+1}\frak e_{j\ell} \kt \right]\\ 
&  \ + \ \sum_{j,\ell} \br   Q_u(\eta,\beta) \ , \ \nabla_\tau^{s+1}(M^{s+1}-1)\psi_{j\ell} \kt \ + \ [\nabla_\tau^{s+1},\frak t_2](\eta,\beta)\\[.3cm]
&=  \sum_{j,\ell} \br  \nabla_\tau^{s}\Big[Q_u(\eta,\beta)\Big] \ , \ \Psi^u_{j\ell} \kt  \ + \  [\nabla_\tau^{s+1},\frak t_2](\eta,\beta) \ + \ [\nabla_\tau^{s+1},\frak t_1](\eta,\beta),
\end{align*}

\noindent where the last line follows from integration by parts. The product rule on the first term then yields \eqref{ttameA} --\eqref{ttameB}, after subtracting the first line of \eqref{Ttamederivs}. 

If the ``commutator'' terms $[\nabla^{s+1}_\tau t_1] ,  [\nabla^{s+1}_\tau t_2]$ happened to vanish, the proof would proceed precisely as in Lemma \ref{tameL1}. Indeed, the tameness of the coefficients in the expression for $Q_u$ in Proposition  \ref{generalmeancurvature} shows that differentiating the coefficients preserves the leading effective order. Thus Proposition \ref{UGBstable} again shows that the derivatives $\nabla Q_u$ are operators of the same leading effective order, i.e.

\[\Big\| (\nabla_\tau^{k+1} T_{u})(\nabla_\tau^{s-k}\eta, \nabla_\tau^{s-k}\beta) \Big \|_{H^{2-\epsilon}(S^1_\tau;\mathcal C_0)} \leq C_s \|u\|_{k+1+n} \|(\eta,\beta)\|_{s-k}. \]

\noindent This is the analogous estimate to \eqref{tamepickup}, and the proof is completed in the same way as Lemma \ref{tameL1}. It therefore suffices to show that

\be  \Big \|  [\nabla_\tau^{s+1}, \frak t_1] \ + \  [\nabla_\tau^{s+1}, \frak t_2] \Big \|_{H^{2-\epsilon}(S^1_\tau;\mathcal C_0)} \leq C_0' \delta \|\eta,\beta\|_{s+1} \ + \  C_s \sum_{k=0}^s \Big( \|u\|_{k+1+n} \|(\eta,\beta)\|_{s-k}  \Big).\label{tametoprove}\ee

\noindent Indeed, the second sum here is over the same norms as \eqref{tamepickup}. Absorbing the first term into the left in  \eqref{withC0interp} (at the cost of increasing $C_0$ by $(1-C'_0\delta)^{-1}$)  and can be absorbed in that step, which yields the equivalent of \eqref{tameinduction} for $T_u.$ Induction again completes the proof; thus we prove \eqref{tametoprove}. 

We first bound $[\nabla_\tau^{s+1}, \frak t_1].$ Since $M^{s+1}$ is a uniformly (in $|\ell|$) bounded Fourier multiplier, $(M^{s+1}-1)\psi_{j\ell}$ satisfy the same exponential decay property as $\psi_{j\ell},$ up to a universal $s$-dependent constant. Proposition \ref{UGBstable} (specifically Remark \ref{UGBstablerem})  therefore shows that the two lines of \eqref{ttameB} respectively obey:

\begin{align}
\|[\nabla_\tau^{s+1}, \frak t_1]\|_{H^{2-\epsilon}}&\leq  \Big \|\Pi_{(M^{s+1}-1)\psi} \left( Q_u(\nabla_\tau^{s+1}\eta, \nabla_\tau^{s+1}\beta)\right) \Big \|_{H^{2-\epsilon}}  \nonumber \\ &  \ + \  \sum_{k=0}^s\Big \|\Pi_{(M^{s+1}-1)\psi} \left( (\nabla^{k+1}_\tau Q_u)(\nabla_\tau^{s-k}\eta, \nabla_\tau^{s-k}\beta)\right) \Big \|_{H^{2-\epsilon}} \nonumber \\[0.3cm] 
&\leq C_0\delta \|(\eta,\beta)\|_{s+1} \ + \ C_s \sum_{k=0}^s\|u\|_{k+1 + n} \|(\eta,\beta)\|_{s-k} \label{firstcommutator}
\end{align}

\noindent which is \eqref{tametoprove} for $[ \nabla_\tau^{s+1},\frak t_1].$ 

The second term is simpler, and is essentialy Cauchy-Schwartz using Proposition \ref{Psidecomp}(B) Indeed, writing $(\nabla_\tau^{s+1}-\ell^{s+1})=(\nabla_\tau-\ell)(\nabla_\tau^s + \nabla_\tau^{s-1}\ell + ...+ \ell^s),$ the triangle inequality and integrating by parts yields:

\begin{align}
\Big \|[\nabla^{s+1}_\tau, \frak t_2](\eta,\beta)\Big \|^2_{H^{2-\epsilon}}&:\leq \sum_{j\ell} \sum_{k=0}^s\Big | \br \nabla_\tau^{s-k}\Big [Q_u(\eta,\beta)\Big ] \ , \   
(\nabla_\tau-\ell)  |\ell|^k \frak e_{j\ell}\kt \Big |^2 |\ell|^{4-2\epsilon} \nonumber \\ 
&  \leq \sum_{k=0}^s\left(\sum_{r=0}^{s-k} \|u\|^2_{s-k-r}\|(\eta,\beta)\|^2_r\right)\left(\sum_{j\ell}  \|(\nabla_\tau-\ell)|\ell|^k \frak e_{j\ell}\|^2 |\ell|^{4-2\epsilon}\right)  \nonumber    \\ 
&  \leq \sum_{k=0}^s\left( \|u\|^2_{s-k}\|(\eta,\beta)\|^2_0 \ + \ \|u\|_{s_0}^2 \|(\eta,\beta)\|^2_{s-k}\right)\left(\sum_{j\ell}  \|(\nabla_\tau-\ell)|\ell|^k \frak e_{j\ell}\|^2 |\ell|^{4-2\epsilon}\right)  \nonumber    \\ 
\intertext{where we have applied Cauchy-Schwartz, then interpolation again. Recalling $|\ell| \leq c |\nabla_\tau -\ell| $ on the Fourier support of $\frak e_{j\ell},$ the bounds of Proposition \ref{Psidecomp} (using $L=k+4,$ say) imply }
&\leq C_s \sum_{k=0}^s\left( \|u\|^2_{s-k}\|u\|^2_{k+4}\|(\eta,\beta)\|^2_0 \ + \ \|u\|^2_{k+4} \|(\eta,\beta)\|^2_{s-k}\right), \label{secondcommutator}
\end{align}

\noindent since $\|u\|_{s_0}<1.$ Interpolation $\|u\|_{{s-k}}\|u\|_{k+4}\leq C_{s,k} \|u\|_{s}\|u\|_{s_0}$ on the first term then yields \eqref{tametoprove} for $[ \nabla_\tau^{s+1},\frak t_2].$ Together \eqref{firstcommutator} and \eqref{secondcommutator} establish \eqref{tametoprove}, completing the proof. In fact, as in Remark \ref{UGBstablerem}, $[\nabla_\tau^{s+1}, \frak t_1], [\nabla_\tau^{s+1}\frak t_2]$ are actually bounded into $H^{2-\epsilon+0.9},$ thus compact. 
\end{proof}

Hypothesis (III) is now a routine consequence of Lemmas \ref{tameL1}, \ref{tameT2}. 

\begin{lm} \label{hypothesisII} $\mathbb Q$ obeys Hypothesis (II) of Theorem \ref{nashmoser}, i.e. after possibly reducing of $\mathcal U_0\times \mathcal U_1,$ there are constants $C_s$ such that the family of linearizations in Lemma \ref{hyp2verification} obey the tame estimates

\be
\|(\eta,\beta,\ph)\|_{s}\leq C_s \ \Big( \|{\mathbb L}_{u} (\eta,\beta,\ph)\|_{s} +\|u\|_{s+s_0}\|{\mathbb L}_{u} (\eta,\beta,\ph)\|_{s_0}\Big),
\ee

\noindent  for $s\geq 0.$
\label{hyp3verification}
\end{lm}

\begin{proof} Lemma \ref{hyp2verification} immediately implies that  that 
\be \|(\eta,\beta,\ph)\|_{0}\leq C_0 \|\overline{\mathbb L}_{u} (\eta,\beta,\ph)\|_{0} \ee

\noindent holds, which implies {\it a fortiori}, the result for $s=0.$

Write ${\mathbb L}_{u}(\eta,\beta,\ph)=(f_0,f_1)$ in the decomposition \eqref{extendedL0prime}. The tame estimates from Lemmas \ref{tameL1} and \ref{tameT2} and the lower block-triangular structure \eqref{extendedL0prime} show 
\begin{align}
\|(\eta,\beta)\|_{s} &\leq C_s \Big(\|f_0\|_s \ + \  \|u\|_{s+ s_0} \|f_0\|_{0}\Big) \label{tameline1}\\
\|\ph\|_{s} &\leq C_s \Big(\|f_1- B_u(\eta,\beta) \|_s \ + \  \|u\|_{s+ s_0} \|f_1 -B_u(\eta,\beta)\|_{0}\Big) \label{tameline2}.
\end{align}
\noindent Recall that Proposition \ref{generalmeancurvature} shows that $(\text{d}\mathbb Q)_u(\eta,\beta,0)$ depends on up to $(3,2)$ derivatives of $(\eta,\beta)$ respectively; thus the shifts in \eqref{tameshifts} mean that $\|B_u(\eta,\beta)\|_s \leq C_s \|(\eta,\beta)\|_s.$ The triangle inequality and substituting the bottom estimate into the top therefore shows

\begin{align}
\|(\eta,\beta,\ph)\|_s &\leq C_s \Bigg( \|(f_0,f_1)\|_s \ + \ \|u\|_{s+ s_0 }\|(f_0, f_1)\|_0  \ + \  \|u\|_{s+ s_0 }\|(\eta,\beta)\|_0  \Bigg).
\end{align}
\noindent{Substituting the $s=0$ estimate, which in the updated notation reads $\|(\eta,\beta,\ph)\|_{0}\leq C_0 \|(f_0, f_1)\|_0$,  yields}
\begin{align}
\|(\eta,\beta,\ph)\|_s &\leq C_s \Bigg( \|(f_0,f_1)\|_s \ + \ \|u\|_{s+ s_0 }\|(f_0, f_1)\|_0    \Bigg),
\end{align}
\noindent as desired. 
\end{proof}

\subsection{Proof of Theorem \ref{maina}}

With Hypotheses (I)--(III) verified by Lemmas \ref{hyp1verification},\ref{hyp2verification}, and \ref{hyp3verification} respectively, Theorem \ref{maina} is a direct application of the Nash-Moser Implicit Function Theorem \ref{nashmoser}. 
\begin{proof}[Proof of Theorem \ref{maina}] Let 

\[\mathcal U_0\subseteq \mathbb D^{\infty}(S^1;\mathcal B_0) \hspace{1.5cm} \mathcal U_{1,D}\subseteq \mathbb X^{\mu,\infty}_D(X_0;NX_0) \hspace{1.5cm} \mathcal V_2\subseteq \mathbb P^s(\del X_0; NX_0)\]

\medskip
\noindent denote the balls of radius $\delta$ measured in the $s_0$ norm, for $\delta$ sufficiently small and $s_0$ sufficiently large so that the hypotheses of Lemmas \ref{hyp1verification},\ref{hyp2verification}, and \ref{hyp3verification} hold. (The finite regularity shift introduced in \eqref{tameshifts} and the coordinate change to arclength coordinates can be absorbed by increasing $s_0$ by $3$). 

Observe that the map $\Phi: \mathcal U_0\times  \mathcal U_{1,D}\times \mathcal V_2$

\be (\xi,\alpha, \ph, \varrho) \mapsto \Phi=\ph + \text{ext}((\xi_0,\alpha_0) + \varrho) \label{bigPhiagain}\ee

\noindent defined in \eqref{tameMCformula} is smooth and tame; thus smooth tame dependence on $\Phi$ implies smooth dependence on $(\xi,\alpha,\ph, \varrho)$. Lemmas \ref{hyp1verification},\ref{hyp2verification}, and \ref{hyp3verification} ensure Hypotheses (I)--(III) of Theorem \ref{nashmoser} hold respectively, where the notation coincides, up to decorations, and $\mathbb X^\infty=\mathbb D^\infty \times \mathbb X^{\mu,\infty}_D.$ We conclude that there are three smooth tame implicitly defined functions

\begin{align}
\xi: \mathcal V_2 &\to C^\infty(S^1 ; \R^N) \label{standardspaces1} \\ 
\alpha: \mathcal V_2 &\to C^\infty(S^1 ; \R^K) \\ 
\ph: \mathcal V_2 &\to r^{1+\epsilon}C_b^\infty(X_0 ; NX_0) \label{standardspaces3} 
\end{align}
\noindent so that 

\[\mathbb Q\Big(\xi(\varrho), \alpha(\varrho), \ph(\varrho), \varrho\Big)=0.\]

The definition \ref{tameMC} of $\mathbb Q$ implies that the corresponding submanifolds 

\[ \Big (  X_{\xi,\alpha,\Phi} \ , \ g_0 \big |_{X_{\xi,\alpha,\Phi}}\Big ) \subseteq (S^1 \times \R^N , g_0)\]

\noindent from Definition  \ref{edgesubmanifolds} are minimal in the Euclidean product metric, where $\Phi$ is as in \eqref{bigPhiagain}. By construction, these have a smooth family of conical singularities along $\Gamma_\xi=(t,\xi(t))\subseteq S^1\times \R^N,$ and at $r=1$ in geodesic normal coordinates around $\Gamma_\xi,$ have boundary values defined by $(\xi_0,\alpha_0)+\varrho$ in exponential normal coordinates.  
\end{proof}

\begin{rem} The boundary, more precisely, is 
\[F_{\xi,\alpha}\left[\text{exp}_{N\Sigma}\big((\xi_0,\alpha_0) + \varrho\big)\right]\subseteq F_{\xi,\alpha}[\{r=1\}]\]

\noindent where $F_{\xi,\alpha}$ are the diffeomorphisms defined in Section \ref{diffssection1}. It is equally well possible to find submanifolds such that $\del X_{\xi,\alpha,\Phi}=\del X_0$ by introducing a radial cut-off un the definition of $F_{\xi,\alpha}.$ This comes at the cost of an additional compact term in the expression for the deformation operator $T_0$ in Theorem \ref{boundarydeformation}, arising from the derivatives of the cut-off, which may result in $T_0$ having a finite-dimensional kernel. Introducing a second, finite-dimensional space $\mathcal B'\subseteq L^2(\Sigma; N\Sigma)$ whose boundary extensions cancel the corresponding kernel, the results of Theorem \ref{maina} then hold with prescribed boundary value modulo $\mathcal B_0+\mathcal B'.$
\end{rem}

%% file: sections/sec_analyticspaces.tex
\section{Tangentially Analytic Function Spaces}
\label{analyticspacessection}
Having completed the proof of Theorem \ref{maina}, we now prove in these final four sections the stronger regularity statement, 
Theorem \ref{mainb}. Its proof is more or less formally identical to that of Theorem \ref{maina}, the difference being that we 
carry out the same steps to apply the Nash-Moser theorem in analytic function spaces carefully designed for the situation at hand.

In the current section, we define spaces whose elements are functions analytic in the $S^1$ direction and establish the analogues of 
the multiplication and composition theorems from Section \ref{SobolevSpaces} for these. These proofs require more care now, and rely 
on the analysis of holomorphic extensions to Grauert Tubes in the cotangent bundle. We introduce these analytic spaces in stages, 
first on $S^1$,  then on $S^1 \times \Sigma$, and finally on $S^1 \times C_1(\Sigma)$. 

\label{FunctionSpacesSection}
\subsection{Analytic Spaces I: Spectral Multipliers}
\label{SectionSpectralSpaces}
To begin, we introduce some notation and the function spaces on $S^1$ and $S^1\times \Sigma$. Let $Z$ be a real-analytic manifold, and 
$V\to Z$ a vector bundle with connection $\nabla$. Set $\Delta_Z = \nabla^*\nabla$.  Associated to any strictly positive weight function 
$w: Z\to \R^+$ is the weighted Sobolev space on $Z$
\[
w^{-1} H^s(Z;V):=\bigg\{ u \in \Gamma(Z;V) \ \Big |\  \int_Z |\nabla^\alpha u|^2 w^{2} dV_Z<\infty, \ \ \forall |\alpha|\leq s\bigg\}.
\]
This applies equally well if $Z$ has conical or edge singularities, for the Sobolev spaces from Definition \ref{bSobolev} and Definition \ref{hybridspacesdef} 
using $\nabla_b, \nabla_e$, respectively. In particular, e.g., $r^\mu H^s_b=wH^s_b$ when $w=r^\mu.$ 

More generally, we may also define function spaces when the weight $\frak W$ is an invertible pseudodifferential (or slightly more general) operator.
%which, rather than being defined by a multiplication operator $w\cdot$, are defined by pseudodifferential operators. 
Thus if $\frak W(u)=\sum_{k} \frak W_k u_k$ is defined relative to the spectral decomposition of $\Delta_Z$ by the spectral multiplier 
$\{\frak W_k\}_{k\geq 1}$, we let  
\[  
\frak W^{-1} H^s(Z;V):=\bigg\{ u \in \Gamma(Z;V) \ \Big |\  \int_Z |\nabla^\alpha  \frak W(u)|^2 dV_Z<\infty, \ \ \forall |\alpha|\leq s\bigg\}.
\]
(Of course, $H^s(Z;V)$ itself is also defined this way as $(1 + \Delta_Z)^{-s/2}L^2(Z;V)$.) 

{\it A priori}, the space defined with integrands $\frak W(\nabla^\alpha u)$ is not necessarily the same as the one defined with $\nabla^\alpha\frak W(u)$.
In most of the cases considered below, sufficient smoothness of the multiplier implies equivalence between these spaces. 

The multipliers defining the analytic function spaces on $Z$ depend on a pair of regularity orders $(s,\sigma)$, where $s$ is the standard 
{\it Sobolev regularity} and $\sigma$ is a measure of {\it analytic regularity}. Consider the spectral multiplier  $\smash{\{\frak W\}_{k}=e^{\sigma |\mu_k|}}$,
i.e., 
\be  
\frak W(u):=\text{Exp}(\sigma\sqrt{\Delta_Z})u,\label{frakWsigma}
\ee
in terms of the spectral data $(\phi_k, \mu_k)$ of the (positive-definite) Laplacian $\Delta_Z.$ 
\begin{defn}  \label{analyticsobolev}  The {\bf analytic Sobolev spaces} of regularity $(s,\sigma,2)$ are the Hilbert spaces 
\be 
e^{-\sigma\sqrt{\Delta_Z}}H^{s}(Z;V):=\Big\{ u \in L^2(Z;V)  \ \ \Big | \ \  \|u\|_{s,\sigma}^2:=\big \| (1+\Delta_Z)^{s/2}e^{\sigma 
\sqrt{|\Delta_Z|}}u \big \|^2_{L^2} < \infty\Big\},
\ee
and corresponding inner products. When $s=0$, we denote these spaces by $e^{-\sigma\sqrt{\Delta_Z}}L^2.$ 

Equivalently, in terms of the eigendata $(\phi_k, \mu_k)$ for $\Delta_Z$, if $u=\sum_k u_k \phi_k,$ then
\be 
 \|u\|^2_{s,\sigma} :=
\sum_{k \in \N} (1+\mu_k)^{s} e^{2\sigma\sqrt{\mu_k}} |u_k|^2. \label{analyticZspace}
\ee
Suggested by this formulation, we sometimes refer to $s$ as the {\it polynomial} and $\sigma$ as the {\it exponential} regularity scale.
\end{defn}

\subsection{Analytic Spaces II: Paley-Wiener Characterization}
There is an alternate and more geometric description of the analytic Sobolev spaces defined above using holomorphic extensions of functions to the cotangent bundle. 
This relies on a generalization of the Paley-Wiener Theorem due to Boutet de Monvel \cite{boutet1979convergence}, see also \cite{Guillemin1997PaleyWiener,stenzel2014analytic}. Recall from \cite{PrincetonLecturesII} that the classical Paley-Wiener Theorem states 
that for $f\in L^2(S^1;\C)$, the following are equivalent: 
\be 
\{e^{\sigma|\ell|} f_\ell \} \in \ell^2(\Z;\C) \ \ \Longleftrightarrow \ \ \left\{\begin{matrix} \text{f admits an $L^2$ holomorphic}  \\ \text{extension to $S^1 \times [-\sigma, \sigma]\subseteq T^*S^1$}\end{matrix}\right\},
\ee
where $T^*S^1=\C/(z\mapsto z+2\pi)$ is the cotangent bundle with its standard complex structure. 

We now explain Boutet de Monvel's generalization. Let $Z$ be a compact real-analytic manifold. 
%Recall that in a complex vector space $(V,J)$ where $J^2=-1$, a {\bf totally real subspace} $W\subseteq V$ is one such that $JW\cap W=0$. $W$ 
%is also said to be maximal if $\dim_\R W=\dim_\C V$. 
\begin{defn}
A {\bf complexification} $({\mathcal{  Z}}, \iota)$ of $Z$ is a pair, where $\widetilde {\mathcal Z}$ is a complex manifold, and $\iota: Z\to 
\widetilde {\mathcal Z}$ is a real analytic embedding such that the image of $d\iota: T_zZ \to T_{\iota(z)}{\mathcal Z}$ is a maximal, totally 
real subspace for each $z\in \mathcal Z$. (A subspace $W$ of a complex vector space $(V,J)$ is called {\it totally real} if $W \cap JW = \{0\}$, and
maximality means that $\dim_{\RR} W = \dim_{\CC} V$.) 
\end{defn}
A Stein domain is a triple $(X, J, \ph)$ where $(X,J)$ is an open complex manifold (so in particular, $J$ is an integrable almost complex structure), and $\ph: X\to \R$ 
is a strictly plurisubharmonic exhaustion of $X$ which is Morse. In the following, $\GrT_{r_0}(Z)$ denotes the neighborhood of radius $r_0$ of the zero section 
in $TZ$. 
\begin{thm} [\cite{guillemin1991grauert,guillemin1992grauert}]\label{thm2.2} If $Z$ is closed, real-analytic manifold, then there exists an $r_0>0$ such that $\GrT_{r_0}( Z)\subseteq TZ$ 
is a Stein domain $(\GrT_{r_0}(Z), J, \ph)$ such that the inclusion $Z\hookrightarrow \GrT_{r_0}(Z)$ mapping $Z$ onto the $0$-section is a complexification of $Z$. The neighborhood $\GrT_{r_0}(Z) \subset TZ$ with its associated Stein structure is called the {\bf Grauert tube} of $Z$ of radius $r_0$.
\label{GrauertTube}
\end{thm}

The work of Boutet de Monvel, Guillemin, and Stenzel gives the following generalized Paley-Wiener theorem: 
\begin{thm}[\cite{boutet1979convergence,Guillemin1997PaleyWiener,stenzel2014analytic}] \label{GuillemanBoutet} Let $Z$ be a real analytic manifold, and $\GrT_{r_0}(Z)$ 
a Grauert Tube of radius $r_0$. Then for $\sigma \leq r_0$,   
\[
e^{\sigma \sqrt{|\Delta_Z}|}f \in L^2 \ \   \ \ \Longleftrightarrow \ \ \left\{\begin{matrix} \text{f admits an $L^2$ holomorphic}  \\ 
\text{extension $\widetilde f:\GrT_\sigma(Z)\to \C$}\end{matrix}\right\}.
\]
Moreover, for each $s\geq 0$, the extension induces an isomorphism 
\[  
e^{\sigma \sqrt{|\Delta_Z}|} H^{s}(Z)  \overset{\simeq}\lre \mathcal O(\del \GrT_{\sigma}) \cap H^{s'}(\del \GrT_\sigma)
\] 
where $\calO$ is the Hardy Space of boundary values of holomorphic functions in $L^2(\GrT_{r_0})$, and $s'=s+\frac{\text{dim}(Z)-1}{4}$. 
After possibly decreasing $r_0$, the same holds for sections $f\in L^2(Z, V)$ when $V\to Z$ is a real-analytic vector bundle. 
\end{thm}
\begin{proof} The references cited here establish the theorem for $\C$-valued functions, so it suffices to explain the extension to sections of a vector bundle. 
After reducing $r_0$, we may assume that $V\otimes_\R \C \to Z$ extends to a holomorphic bundle $\widetilde V\to \GrT_\sigma(Z)$; indeed, 
since $Z$ is compact we may extend all transition functions in $\text{GL}(k,\R)\subseteq \text{GL}(k,\C)$ defined on a neighborhoods $\{U_\alpha\}$ via 
the Taylor series to $\GrT_{r_1}(U_\alpha)$ for some $r_1>0$. Analytic continuation ensures that the cocycle condition continues to hold. 
The proof in \cite{stenzel2014analytic} using heat kernels extends in an obvious way to sections of a vector bundle. 
\end{proof}

In general, the two conditions $\ph \in e^{\sigma \sqrt{|\Delta_Z}|} H^{s}(Z;V)$ and $\nabla^\alpha \ph \in e^{\sigma\sqrt{\Delta_Z}}L^{2}(Z;V)$ for $|\alpha|\leq s$ do not appear to immediately agree since
these require that
\be 
\|\nabla^\alpha e^{\sigma\sqrt{\Delta_Z}}\ph\|_{L^2}< \infty \quad \mbox{and}\quad   
\| e^{\sigma\sqrt{\Delta_Z}}\nabla^\alpha\ph\|_{L^2} <\infty  \label{twoversionHssigma}
\ee
for all $|\alpha| \leq s$, respectively.   When $Z=S^1$, $[\nabla^\alpha, \Delta_Z]=0$ and the equivalence is obvious. The result is fortunately true
in general. 
\begin{lm} \label{exchangingorderHssigma} The two criteria \eqref{twoversionHssigma} are equivalent. In other words, for $\sigma>0$ there exists a constant $C=C(s,Z) > 0$ uniform over $\sigma$ in a compact interval such that
\[
\frac{1}{C} \|\ph \|_{s,\sigma}\leq   \| e^{\sigma\sqrt{\Delta_Z}}\ph\|_{H^{s}(Z)} \leq C \|\ph \|_{s,\sigma}
\]
for all $\ph$ such that either quantity is finite. Moreover both conditions \eqref{twoversionHssigma} are equivalent to the spectral definition  \eqref{analyticZspace}. 
\end{lm} 
\noindent The proof relies on the unique continuation property of the holomorphic extensions, but is omitted here since only the case for $Z=S^1$ will be needed in this paper.

Returning to our specific setting, we will also require the following analyticity properties of the minimal link. Let $\Sigma\subseteq (S^{N-1},g_0)$ be a smooth, closed, oriented minimal submanifold as before, with induced metric $g_\Sigma$. 
Let $\rho_\Sigma$ be sufficiently small so that the image of $n(\Sigma):=\exp_\Sigma(\Sigma \times B_{\rho_\Sigma})$ is an open neighborhood of $S^{N-1}$.
In local coordinates $(x_1,...,x_{n-1})$ on $\Sigma$, 
\be 
g_0|_{n(\Sigma)}=g_\Sigma + d\nu^2 + h(x,\nu)\label{gsigmaerror}
\ee
where $\nu=(\nu_1,...,\nu_{N-n})$ are Fermi coordinates in the normal directions and $h(x,\nu)=O(\nu)$.  
 
\begin{lm}[Morrey \cite{morrey1966multiple}(Sec. 6.6), cf. \cite{leung1973reflection}] 
The following hold: 
\begin{enumerate}
\item[(A)] $(\Sigma,g_\Sigma)$ and $(n(\Sigma), g_0)$ are real-analytic Riemannian manifolds, 
\item[(B)] $h(x,\nu): n(\Sigma)\to \R^{(N-1)^2}$ is an analytic function,
\item[(C)] The eigensections $\phi_j$ of the Jacobi operator $L_\Sigma$ are analytic sections of $N \Sigma$. 
\end{enumerate}
\label{Sigmaanalytic}
\end{lm}

\subsection{Anisotropic Multiplication and Composition Theorems}
For shorthand, let $D_t=\sqrt{\Delta_t}.$ Thus for $Z=S^1,$ the spaces from Definition \ref{analyticsobolev} are now written $\smash{e^{-\sigma D_t}H^s(S^1;V)}.$ The {\it anisotropic analytic Sobolev spaces} 
on $S^1\times \Sigma$ with regularity $(s,\sigma)$ are then defined by 
\[
e^{-\sigma D_t} H^{s}(S^1\times \Sigma;V):=\Big\{ u \in L^2(S^1\times \Sigma;V)  \ \ \Big | \ \  \|u\|_{s,\sigma}^2:= \big \| (1+\Delta_\Sigma + \Delta_{t})^{s/2}
e^{\sigma D_t}u \big \|_{L^2} < \infty\Big\}. 
\]
Thus functions $u\in e^{-\sigma D_t} H^{s}(S^1\times \Sigma)$  have up to $s$ derivatives in $L^2,$ all of which additionally extend holomorphically to $\GrT_\sigma(S^1)$ in the first factor.

%Elements in these spaces have standard Sobolev regularity in the $\Sigma$ directions, and analytic regularity of order $\sigma$ in the $S^1$-direction. 
In this subsection, we establish multiplication, interpolation, and composition theorems for these anisotropic analytic Sobolev spaces. 
This is a precursor to the corresponding results for the global tangentially analytic spaces on $X_0$, which are ultimately needed 
to accommodate the non-linearities in the mean-curvature equation, as in the proof of Theorem \ref{maina}. 

\begin{lm} \label{1dmultiplications} 
These spaces have the following properties:   %, and hold uniformly for $\sigma$ in a compact interval. 
\begin{enumerate}
\item[(A)] Let $V_i$ for $i=1,2,3$ be vector bundles that admit extensions to $\GrT_{\sigma}(S^1) \times \Sigma$ which are holomorphic in the first factor. 
Then any fiberwise bilinear bundle map $V_1\times V_2\to V_3$ that admits an extension to 
$\GrT_{\sigma}(S^1) \times \Sigma$ holomorphic in the first factor extends to a bounded operator
\[  
e^{-\sigma D_t} H^{s}(S^1 \times \Sigma;V_1) \times e^{-\sigma D_t}H^{s}(S^1\times \Sigma;V_2)\lre e^{-\sigma D_t} H^{s}(S^1\times \Sigma;V_3)
\]
provided $s > \frac{\text{dim}(S^1\times \Sigma)}{2}$.

\smallskip

\item[(B)] Given $0\leq s_1 < s < s_2$, there is a constant $C=C({s,s_1,s_2})$, uniformly for $\sigma$ lying in a bounded interval, such that the interpolation inequalities 
\[ 
\|\eta\|_{s,\sigma} \leq C \|\eta\|^\alpha_{s_1,\sigma} \|\eta\|_{s_2,\sigma}^{1-\alpha}
\]
hold, where $\alpha=\tfrac{s_2-s}{s_2-s_1}$.
\end{enumerate}
\noindent In particular, the hypotheses of (A) hold for bilinear maps defined by contraction of tensors with the ambient Riemmanian metric and its inverse. 
\end{lm}
\begin{proof}   The (squared) norm on $ e^{-\sigma D_t} H^{s}(S^1\times \Sigma; V)$ can be written as
\be 
\|u\|^2_{s,\sigma}=\sum_{0\leq |\alpha| \leq s} \int_\Sigma \|\nabla^\alpha_\Sigma  e^{\sigma D_t}u \|^2_{H^{s-\alpha}(S^1)} \ dV_\Sigma, \label{unorm525A}
\ee
where each $\nabla^\alpha_\Sigma$ denotes a derivative in a direction tangential to $\Sigma$. Let 
\[
\GrTparam=\{(x,z) \ | \ x\in \Sigma \ , \ z\in \GrT_\sigma(S^1)\}
\]
denote the family of Grauert tubes of $S^1$ with $x \in \Sigma$ as a parameter. Denote by $H^s H^{\beta}_{\del \GrT}(\del \GrTparam)$ the anisotropic Sobolev 
spaces on $\GrTparam$ with $s$ derivatives along $\del \GrTparam$ and $\beta$ {\it additional} derivatives in directions tangential to $\del \GrT_{\sigma}(S^1)$. 
We claim that 
\be
||u||_{s,\sigma} < \infty 
% (\refeq{unorm525A}) \text{ is finite} \ \   
\ \ \Leftrightarrow \ \ \left\{\begin{matrix} \text{u admits an $L^2$ fiberwise holomorphic extension}  \\ \text{$\widetilde u \in 
\mathcal O(\mathcal G)\cap H^s H^{(k-1)/4}_{\del \GrT}(\del \GrTparam) \ \  $}\end{matrix}\right\},\label{highernormequiv1}
\ee
{where $k=\text{dim}(S^1) = 1$}, and $\mathcal O(\del \mathcal G)\subseteq H^1(\del \mathcal G)$ denotes the ``slicewise Hardy space''  consisting of $H^1$ 
functions that are boundary values of holomorphic functions on $\GrT_\sigma(S^1)$ for each $x \in \Sigma$. Smooth $\Sigma$-parametrized families of 
$e^{-\sigma D_t}H^s(S^1)$-sections $u(x,t)$ are dense in the space on the left of \eqref{highernormequiv1}, while their holomorphic extensions $\widetilde u(\tau,z)$ 
are dense in the space on the right; it therefore suffices to show that \eqref{highernormequiv1} is a bounded equivalence of norms on such sections.  If $u(x,t)$ depends smoothly on $x$, 
the holomorphic extensions 
\[
\nabla^a\widetilde u=\widetilde{\nabla^a u}
\]
coincide by unique continuation.  But then, by Theorem \ref{GuillemanBoutet}, $\|\nabla^a  e^{\sigma D_t}u\|_{H^{s-a}(S^1)} \sim 
\| \nabla^a \widetilde u\|_{H^{s-a+(n-1)/4}(\del \mathcal G)}$ for each $a=0,...,s$, which establishes \eqref{highernormequiv}.

Now $k=\text{dim}(S^1)=1$, hence $H^s H^{(k-1)/4}_{\del \GrT}(\del \GrTparam)=H^s(\del \GrTparam)$.  When $s>\frac{dim(\del \GrTparam)}{2}=\frac{n}{2}$, $H^s$ is an algebra. 
Since the fiberwise bilinear pairing is assumed to be holomorphic, and the product of holomorphic functions functions is holomorphic (and commutes with taking boundary values), 
multiplication extends to the right side of \eqref{highernormequiv1} as a bounded map. This proves (A). Standard interpolation inequalities on $\del \GrTparam$ yield (B).
\end{proof}

The following is the analytic generalization of the composition estimate Theorem \ref{composition}. It is once again needed to bound terms of the form $h(x,\Phi(t,r,x))$ appearing in the mean curvature (cf. the proof of Lemma  \eqref{pullbackgPhixi}).  

\begin{lm} \label{compositionslice}  Given $s\geq n= \dim(S^1\times \Sigma)$ and $\sigma>0,$ then there is a $\delta>0$ such that if $\xi \in 
e^{-\sigma D_t} H^{s}(S^1\times \Sigma; N\Sigma)$ with $\|\xi\|_{n,\sigma}< \delta$, then there are constants $C_s$ uniform over $\sigma$ in a compact interval such that:
\begin{itemize}
\item[(A)]  $h(x,  \xi)\in e^{-\sigma D_t} H^{s}(S^1\times \Sigma; \mathrm{Sym}^2(T^*n(\Sigma)))$,
\smallskip
\item[(B)] $\|h(x, \xi) \|_{s,\sigma} \leq C_s \|\xi\|_{s,\sigma}$.
\end{itemize}
\end{lm}

\begin{proof} For $\sigma_\Sigma$ sufficiently small, Theorem \ref{thm2.2} gives that $\Sigma$ admits a Grauert tube $\GrT_{\sigma_\Sigma}(\Sigma)$.
Wen $\rho_\Sigma$ is small enough, the tubular neighborhood $n(\Sigma)\subseteq S^{N-1}$ of radius $\rho_\Sigma$ 
likewise admits a Grauert tube. Despite the fact that Theorem \ref{thm2.2} does not apply directly (becuase $n(\Sigma)$ is not compact), we can realize its Grauert tube as 
an open subset of the Grauert tube $\GrT_{\sigma_\Sigma}(S^{N-1})$ of the round sphere: 
%After possibly reducing the radius $\sigma'$, we may therefore assume this Grauert tube extends to one 
\[
\mathcal U_{\rho_\Sigma, \sigma_\Sigma}:=\{(x,\nu, y, \omega) \ | \ |\nu|< \rho_\Sigma \ , \ |y|^2 + |\omega|^2 < \sigma_\Sigma^2\}\subseteq T(n(\Sigma)).
\]

\noindent Since the ambient round metric, $g_\Sigma$ and $d\nu^2$ extend holomorphically to $\mathcal U_{\rho_\Sigma, \sigma_\Sigma}$, 
\eqref{gsigmaerror} implies that $h$ does as well. Reducing $\sigma_\Sigma$ slightly, we may therefore assume that $h$ extends 
holomorphically to $\mathcal U_{\rho_\Sigma, \sigma_\Sigma}$ with smooth boundary values along $\del \mathcal U_{\rho_\Sigma, \sigma_\Sigma}$. 
After one more reduction of $\sigma_\Sigma$, standard Cauchy estimates %applied on $\del \GrT_{\sigma_\Sigma}$  show that there are a constants $C_m$ independent of $k$ such that
\be 
\| (\nabla^\C_\text{Vert})^k h\|_{H^{n+m}(n(\Sigma))} \leq \frac{C_m k!}{\sigma_\Sigma^{k}},  \label{Cauchyhk}
\ee
hold, where $\nabla_{\text{Vert}}^k h$ denotes the $k^{th}$ derivative in the vertical (fiber) directions on $n(\Sigma)$. 

Expanding in Taylor series, we write 
\be 
h(x,\nu + \xi)=h(x,\nu) + \sum_{k=1}^\infty \frac{1}{k!} \left(\nabla_{\text{Vert}}^k h\big |_{(x,\nu)} \right)\underbrace{(\xi,\ldots, \xi)}_{k \text{ arguments}}\label{TaylorExpansion}
\ee
where $\nabla_{\text{Vert}}^k h$ denotes the $k^{th}$ derivative in the vertical (fiber) directions on $n(\Sigma),$ viewed as a multi-linear function of the arguments. Since $h$ is constant in the $S^1$-direction, ( \refeq{Cauchyhk}) shows the same bounds hold for $\|(\nabla^C_\text{Vert})^k h\|_{s,\sigma}$ along $\nu=0.$ Using the multiplication from  Lemma \ref{1dmultiplications}, let $C_n$ be the constant 
in the multiplicative estimate for $s=n$, and choose $\delta< \sigma_\Sigma/2C_n$. Since $h(x,0)=0$, %by (\refeq{gsigmaerror}),
we have
\[
\begin{split}
\|h(x,  \xi) \|_{n,\sigma} &\leq \sum_{k=1}^\infty  \frac{1}{k!}\Big \| \left(\nabla_{\text{Vert}}^k h\big |_{(x,0)} \right)(\xi,\ldots ,\xi)\Big \|_{n,\sigma} \\
&\leq  \sum_{k=1}\frac{C_n^{k-1}}{k!} \|\nabla_{\text{Vert}}^k h\|_{n,\sigma} \|\xi\|_{n,\sigma}^{k} \\
&\leq C \|\xi\|_{n,\sigma} \sum_{k=1}^\infty \frac{C_n^{k-1} \|\xi\|_{n,\sigma}^{k-1}}{\sigma_\Sigma^{k-1}} \leq C \|\xi\|_{n,\sigma},
\end{split}
\]
where we use \eqref{Cauchyhk} with $m=0$. Convergence follows from the choice of $\delta$. This shows (B) when $s=n$, and  (A) is then immediate. 

For $s>n$, differentiating \eqref{TaylorExpansion} yields 
\[
\nabla^m \left(h(x, \xi) \right) =\sum_{k=1}^\infty \frac{1}{k!} \left(\nabla_{\text{Vert}}^{k+a_0} h\big |_{(x,0)} \right){( \nabla^{a_1}\xi,\ldots , \nabla^{a_k}\xi)}
\]
where $\sum_{j=0}^k a_k=m$. By the interpolation inequality from Lemma \ref{1dmultiplications}(B), 
\[
\begin{split}
\Big \| \left(\nabla_{\text{Vert}}^{k+a_0} h\big |_{(x,0)} \right)(\nabla^{a_1}\xi,\ldots ,\nabla^{a_k}\xi)\Big \|_{n,\sigma}& \leq C_mC_n^{k-1} \|\nabla_{\text{Vert}}^{k} h\big |_{(x,0)}\|_{n+m,\sigma} \cdot \|\nabla^{a_1}\xi\|_{n,\sigma}\cdot \ldots \cdot \|\nabla^{a_k}\xi\|_{n,\sigma}\\
&\leq C_m (C_n C_I)^{k-1}\|\xi\|_{n,\sigma}^{k-1} \|\xi\|_{n+m,\sigma} 
 \end{split}
\]
where $C_I$ is the interpolation constant (which can be taken uniform in the spectral norm).
%, and up to the addition of an additional constant $C_m$ from the equivalence with the spectral norm). 
The argument then proceeds as for $s=n$ after further reducing $\delta.$ 
\end{proof}

\medskip 

\subsection{Spaces with Radially Dependent Analyticity} 
\label{multcompsection}
We now define global spaces of analytic functions on $X_0=S^1 \times C_1(\Sigma)$ which have $b$-Sobolev regularity in the radial direction and 
tangential anisotropic analytic regularity {\it depending on the radius}.  There are versions adapted to both boundary and edge derivatives.
We extend the relevant multiplication, interpolation, and composition theorems from the previous section. This radial dependence allows the spaces to interpolate between analytic tangential regularity at $r=0$ to standard smooth regularity at the outer $r=1$ boundary. 

These global function spaces are defined using weights that are families of pseudodifferential operators parameterized by the radial variables 
(cf. \eqref{weakasympotitic} and the accompanying discussion). 

\begin{notation}\label{notation5.2}
Let $w_k:[0,1] \to \R^{+}$ be a family of smooth positive functions indexed by $k \in \N.$ For a manifold $Z$, denote by $\underline w$ the operator 
\begin{equation}
\begin{split}
 \underline w: L^2(Z ;\R )&\to  L^2([0,1]\times Z; \R) \\  
f=\sum_k a_k \phi_k &\mapsto \underline w[f]  :=  \sum_{k \in \N} w_k(r) a_k  \phi_k, 
\end{split}
\label{underlinef}
\end{equation}
\noindent where $\phi_k$ are again the eigenfunctions of $\Delta_Z.$ We may regard $\underline w$ as a family of spectral multipliers on $L^2(Z; \R)$, parametrized by $r \in [0,1]$.
\end{notation}

We now define a specific families of spectral multipliers on $Z=S^1$ as in Notation \ref{notation5.2}. Let $\sigma(r): [0,1]\to \R$ be a fixed positive smooth and strictly 
decreasing cutoff function. With $D_t=\sqrt{\Delta_t}$ as before, set
\begin{align}
\expmultpos &:=\text{Exp}(\sigma(r) \sqrt{\Delta_t})\label{spectralmultpliersa}
\end{align}
This is an $r$-parameterized families of operators as in Notation (\ref{underlinef}) corresponding to families of spectral multiplier 
$w_\ell(r)=\text{Exp}(|\ell| \sigma(r))$ where $\ell \in \Z$ is the Fourier mode in the $S^1$ direction. 
These are also each invertible, so in particular we can also make sense of the operator
\[
\expmultneg 
\]
Notice that the first multiplier defines a bundle of (anisotropic) analytic Sobolev spaces $\expmultneg H^{s}(S^1\times \Sigma; NX_0)$ 
over the interval $r\in [0,1]$. 

\begin{defn}\label{calHebdef}
Fixing weights $\mu,\nu\in \R$, analytic regularity $\sigma(r)$ and hybrid $(e,b)$ Sobolev regularity $(k,s)$, 
\begin{align} 
\|\ph\|_{\mu, \sigma, s}&=\Big( \int_{X_0} \sum_{|\alpha| \leq s}|  \expmultpos(\nabla_b^\alpha\ph)|^2 
\ r^{-2\mu} \ dV \Big )^{1/2}, \label{calHebnormsb}   \\ 
\|\ph\|_{\mu, \sigma, (s,k)}&=\Big( \int_{X_0}  \sum_{|\beta| \leq k} \sum_{|\alpha| \leq s } | \expmultpos(\nabla_e^\beta\nabla_b^\alpha\ph)|^2 
% \ + \ \ldots \ + \  |   \underline {\frak W}(\nabla_b^s\ph)|^2  \ + \\  & \qquad \qquad \qquad \qquad \ldots \ + 
%\ |   \underline {\frak W}(\ph)|^2   
\ r^{-2\mu} \, dV \Big )^{1/2}, \label{calHebnormse}
\end{align}
where $dV=r^{n-1} dr dt dv_\Sigma$ and $\nabla_b, \nabla_e$ are covariant $b$- and $e$-derivatives, as in Eq.~\eqref{bederivs}.   
The second family of norms coincides with the first when $k=0$. %; in this case the fifth subscript is omitted. 

The $b$ and edge Sobolev spaces with radially-dependent analyticity or, more concisely, the {\bf  analytic $b$ and edge Sobolev spaces},
are the spaces of sections for which these norms are finite:
\begin{align}
r^\mu   \expmultneg H^{s}_b (X_0; NX_0)&:=  \Bigg \{ \ph \in L^2(X_0; NX_0) \  \ \Big |  \ \ \|\ph\|_{\mu, \sigma, s}<\infty\Bigg \} \label{cal Hebspaces} \\
r^\mu   \expmultneg H^{s,k}_{b,e} (X_0; NX_0)&:=  \Bigg \{ \ph \in L^2(X_0; NX_0)  \ \ \Big |  \ \ \|\ph\|_{\mu, \sigma, (s,k)}<\infty \Bigg\} 
\end{align}
\noindent When $(k,s)=(k,0)$, we write simply $r^{\mu}   \expmultneg H^k_e$, and if $(k,s)=(0,0)$, we write 
$r^\mu   \expmultneg L^{2}.$ Additionally, the subspace obeying Dirichlet boundary conditions is denoted by 
$r^\mu   \expmultneg H^{s,k}_{b,e,\Dir},$ and likewise for various other boundary conditions as they arise.  

Notice that the superscript suppresses the $r$-dependence of $\sigma(r)$. We always take $\mu>0$, so there should be no confusion with the similar notation for the anisotropic spaces \eqref{highernormequiv1} 
(even when the domain is not written explicitly). 
\end{defn}

\smallskip

We also consider analytic versions of $C^k$-spaces
\be 
r^\mu  \expmultneg C^{k}_b (X_0; NX_0):=  \Bigg \{ \ph \in L^2(X_0; NX_0) \  \ \Big |  \ \ \|\ph\|_{r^{\mu} \expmultneg C^{k}_b}<\infty\Bigg \}. 
\label{analyticCk}
\ee
 using the norms
\be 
\|\ph\|_{r^{\mu} \expmultpos C^{k}_b} :=  \sum_{|\alpha| \leq k} \sup_{y \in X_0}  \big | \left(r^{-\mu}  \expmultneg
\nabla_b^\alpha \ph\right)(y) \big | 
\label{holderCbsigmanorm} 
\ee
Fortunately we only need these, but not generalizations involving H\"older seminorms.

\begin{lm} \label{C0embedding}For $\mu \in \R$, and $(k,s) \in \N^2$, then % spaces in Definition \ref{calHebdef} obey the following. 
\begin{enumerate}
\item[(A)] there are constants $c=c(\mu,\sigma,s,k),$ and $C=C(\mu,\sigma,s,k)$ such %that there is an equivalence of norms 
\[ 
c\|\ph\|_{\mu,\sigma, (s,k)=} \ \leq \ \|    \expmultpos(\ph)\|_{r^\mu H^{s,k}_{b,e}} \ \leq \ C \|\ph\|_{\mu,\sigma, (s,k)} \, ;
\]
%\noindent with the standard boundary and edge Sobolev spaces. 
\item[(B)] when $\mu>n/2$, there is a dimensional constant $a_n$ such that 
\[
r^{\mu} \expmultneg H^{s}_b(X_0; NX_0) \hookrightarrow r^{\mu-n/2} \expmultneg C^{s-a_n}_b(X_0; NX_0)
\]
is a continuous embedding.
\end{enumerate} 
\end{lm}
\begin{proof}
(A) Omitting the superscript $b$ from the notation for the moment, write $r\nabla_r$ for the radial derivative. For $s=0,$ the conclusion is simply Definition \ref{calHebdef}.In the $\ell^{th}$ tangential Fourier mode, we have 
\be  
\big[(r\nabla_r),   \expmultpos\big](\ph)=
r|\ell|  (\nabla_r \sigma)    e^{\sigma(r) |\ell|} (\ph) \label{commutatorR}
\ee
and this gives
\bea
 |r\nabla_r \expmultpos(\ph)|^2  dt \leq C\Big( |e^{\sigma(r)|\ell|}(r\nabla_r\ph) |^2 + |  e^{\sigma(r)|\ell|}(r\nabla_t\ph)|^2   \Big).
\eea
Because $r$ and $\nabla_t$ both commute with the weight $\expmultpos$, the second inequality in (A) follows when $s=1$.
Larger values of $s$ are handled by induction. A rearrangement of \eqref{commutatorR} yields the reverse inequality.

(B) follows from a Banach-space valued Morrey embedding \cite[Thm. 6.2]{arendt2016mapping}. First, observe that for the standard $b$-spaces, 
one has the Morrey embedding 
\be 
r^\mu H^s_b(C_1(\Sigma); E) \hookrightarrow r^{\mu-n/2}C^{s-\lceil n/2 \rceil }_b(C_1(\Sigma); E)\label{CSigmaembedding}
\ee
Indeed, with the coordinate change $r=e^\tau$, $r\del_r$ becomes $\del_\tau$, and $r^{-2\mu + n-1}dr = e^{(-2\mu+n)\tau}d\tau$. 
With these, the space on the left here is $e^{(\mu - n/2)\tau} H^s(-\infty; 0] \times \Sigma \ ; \ E)$, which is then identified with the standard 
Sobolev space on $(-\infty,0]\times \Sigma$ via multiplication by the weight $w=e^{(\mu-n/2)\tau}$.  At that point, \eqref{CSigmaembedding} 
follows from the standard Morrey Embedding.

The $\|-\|_{\mu,\sigma, s+k}$-norm dominates the mixed-regularity norm with $s$ tangential derivatives along $C(\Sigma)$ and up to $k$ 
additional derivatives in the $S^1$ direction, hence
\[ 
r^{\mu} \expmultneg H^{n+k }_b(X_0; NX_0) \ \  \hookrightarrow \ \  r^{\mu}H^{n}_b\Big (C_1(\Sigma) \ ; \ 
 \underline e^{-\sigma(r)D_t} H^{k}(S^1)\Big ). 
\]
Here the space on the right consists of $r^\mu H^n$ sections of the bundle over $C(\Sigma)$ whose fibers are the Banach spaces associated
to the norm $\|   \underline e^{\sigma(r)D_t} \ph\|_{H^s}$ on $S^1$.  Lemma \ref{exchangingorderHssigma} implies that multiplication by the weights gives a local trivialization of this bundle modeled on the 
product  $C(\Sigma) \times H^{s}(S^1; NX_0)$ which is norm-preserving up to a uniform constant for $\sigma_0$ in a compact interval. Changing to the $\tau$ coordinates again, 
\cite[Thm. 6.2]{arendt2016mapping} gives the embedding
\[
r^{\mu}H^{n}_b\Big (C_1(\Sigma) \ ; \   \underline e^{-\sigma(r)D_t}H^{k}(S^1)\Big ) \hookrightarrow  
r^{\mu-n/2}C^{0}_b\Big (C_1(\Sigma)  \ ; \  \underline e^{-\sigma(r)D_t} H^{k}(S^1)\Big).
\]
  
\noindent Finally, multiplication by the weights along with the standard Morrey embedding on $S^1$ shows that 
$  e^{-\sigma(r)D_t} H^{1}(S^1; NX_0)\hookrightarrow   e^{-\sigma(r)D_t}C^{0}(S^1;NX_0)$ 
is continuous, with norm uniformly bounded in $r$. Composing with this embedding on the fibers of the Banach bundle, we obtain 
the proof when $s-f(n)=0$, since the  norm on $\smash{r^{\mu-n/2}C^{0}_b\Big (C_1(\Sigma) ;  e^{-\sigma(r)D_t}C^0)}$ 
is equivalent to \eqref{holderCbsigmanorm}.   Taking derivatives gives the result for larger values of $s.$
\end{proof}

We then have the following extensions of the interpolation and multiplication theorems. 

\begin{prop} \label{multiplications} The following hold for $\mu,s,\sigma(r)$ as above.
\begin{enumerate}
\item[(A)]   If $0\leq s_1 < s < s_2$, there are constants $C=C(\mu,\sigma,s,s_1, s_2)$  and interpolation inequalities 
\[ 
\|\ph\|_{\mu,\nu, \sigma, s} \leq C \|\ph\|^\alpha_{\mu,\nu, \sigma, s_1} \|\ph\|_{\mu,\nu, \sigma, s_2}^{1-\alpha},
\]
where $\alpha=\tfrac{s_2-s}{s_2-s_1}$. 

\item[(B)] When $s> \text{dim}(X_0)$ and $\mu>n/2$, there are constants $C=C(\mu,\sigma, s)$ such that for fiberwise bilinear bundle maps obeying the hypotheses of Lemma \ref{1dmultiplications}(A), pointwise multiplication extends to a bounded operator on 
$r^{\mu} \expmultneg H^{s}_b$, i.e. 
\[ 
\|\ph \psi \|_{\mu, \sigma, s}  \ \leq \ C \|\ph  \|_{\mu, \sigma, s}  \| \psi \|_{\mu, \sigma, s}  
\]
\item[(C)] In the same setting, except where $\mu<n/2$ but $\mu+1>n/2$, a fiberwise bilinear bundle map extends to a bounded operator
\[ 
\|\ph \psi \|_{\mu, \sigma, s}  \ \leq \ C \|\ph  \|_{\mu+1, \sigma, s+a_n}  \| \psi \|_{\mu, \sigma, s}  
\]
\noindent for $\ph \in r^{\mu+1} \expmultneg H^{s+a_n}$ and $\ph \in r^{\mu} \expmultneg H^{s}$ for a dimensional constant $a_n.$
\end{enumerate}
\noindent Once again, (B) and (C) apply, in particular, to fiberwise bilinear maps on tensors defined by contraction with the ambient metric and its inverse. 
\end{prop}

\begin{proof} (A) follows directly from the isomorphism in Lemma \ref{C0embedding}(A) and interpolation inequalities on the standard $H^{s}_b$ spaces. 

(B) With $r=e^{\tau}$, the norm \eqref{calHebnormsb} is equivalent to 
\begin{align} 
\| \ph \|^2_{\mu,0, \sigma, s} &= \sum_{0\leq |\alpha| \leq s} \int_{-\infty}^0 \int_\Sigma \|  e^{\sigma(\tau)D_t}\nabla^{\alpha} \ph\|^2_{H^{s-|\alpha|}(S^1)}  
e^{-2\tau+n} d\tau dV_\Sigma  \nonumber \\
&\sim   \sum_{0\leq |\alpha| \leq s} \int_{-\infty}^0  \int_\Sigma \|  e^{\sigma(\tau)D_t}\nabla^{\alpha}_\tau  u\|^2_{H^{s-|\alpha|}(S^1)} d\tau \label{unorm525}
\end{align}
where $\alpha$ ranges over multi-indices on $\R\times \Sigma$ and $u=e^{(-\mu+n/2)\tau}\ph$.  Multiplication by $w=e^{(\mu-n/2)\tau}$ 
obeys $w^2 \leq C w$ precisely when $\mu>n/2$. In this range, boundedness of multiplication for weight $\mu$ thus follows from 
boundedness of multiplication for the re-weighted norms \eqref{unorm525}.

The proof now proceeds as in Lemma \ref{1dmultiplications}: let 
\[
\GrTparam=\{(\tau, x,z)   \ |  \ (\tau,x) \in \R^{\leq 0}\times \Sigma \ , \ \ z\in \GrT_{\sigma(\tau)}\}
\]
denote the parametrized Grauert tube, with radii $\sigma(\tau)$. By the same argument as in \eqref{highernormequiv1},
\be 
(\refeq{unorm525}) \text{ is finite} \ \   \ \ \Leftrightarrow \ \ \left\{\begin{matrix} \text{u admits an $L^2$ slicewise holomorphic extension}  \\ 
\text{$\widetilde u \in \mathcal O(\mathcal G)\cap H^s H^{(k-1)/4}_{\del \GrT}(\del \GrTparam) \ \  $}\end{matrix}\right\},\label{highernormequiv}
\ee
with $k=\text{dim}(S^1)=1$. By virtue of the boundedness of $\sigma$, $\del \GrTparam$ has uniformly bounded geometry, and 
boundedness of multiplication follows from the standard Sobolev multiplication theorem on $H^s(\del\GrTparam).$  
This completes (B).

 The proof of (C) is identical, now using that $e^{-(\mu+1) + n/2}e^{-\mu + n/2}\leq e^{-\mu + n/2}$ for these weights, and the corresponding module multiplication on $\del \GrTparam.$
\end{proof}

We then obtain an extension of Lemma \ref{compositionslice}. 
\begin{prop} \label{compositionfull}Let $s\geq n+1=\mathrm{dim}(X_0)$. %, and $\sigma(r):[0,1]\to \R^{\geq 0}$ be continuous. 
There are constants $C_s=C(\mu,\nu,\sigma,s)$ %depending only on $\Sigma$ and $s$ 
and a $\delta>0$ such that if $\ph \in r^\mu   \expmultneg H_b^{s}(X_0; NX_0)$ has $\|\ph\|_{\mu,\nu,\sigma,n}< \delta$, 
then 
\begin{itemize}
\item[(A)]  $h(x,  \ph(t,r,x))\in r^\mu    \expmultneg H^{s}(X_0; \text{Sym}^2(T^*n(\Sigma)))$;
\item[(B)] $\|h(x, \ph(t,r,x)) \|_{\mu,\nu, \sigma, s,} \leq C_s \|\ph\|_{\mu,\nu, \sigma, s}$
\end{itemize}
\end{prop}
\begin{proof} Because $h$ depends on only $(x, \nu)$, the expression \eqref{TaylorExpansion} in the proof of Lemma \ref{compositionslice} applies
here.  For the rest of the proof, we use the multiplication and interpolation inequalities.   This proof is thus formally the same as 
that for Lemma \ref{compositionslice}, but invoking 
%the multiplication and interpolation of 
Proposition \ref{multiplications} in place of Lemma \ref{1dmultiplications}. The uniformity follows from the same uniformity statement in these. 
\end{proof}

\subsection{Analytic Weights and Conventions}
\label{conventionspart2}
Henceforth, we fix the following choices of data. 

\begin{enumerate}
\item The choices of $\mu$ and $\Sigma$ are retained from Section \ref{weightconventions}, so that items (1)--(4) enumerated there hold, now for the spaces $r^{\mu+1}\expmultneg H^s_b$  (resp. $r^{\epsilon}\expmultneg C^s_b),$ using the corresponding statement from subsection \ref{multcompsection}. 
\item We fix $\sigma<1,$ and 
\begin{align}
\sigma_r(r)&:= \sigma (1-r).
\end{align}
\end{enumerate}
For $\sigma>0$, we use the shorthand $\mathcal D^{\sigma,s}=e^{-\sigma D_t} H^s$ to denote the analytic spaces on $S^1$, so that spaces and norms are now  denoted
\[
\Big ( r^\mu \expmultneg H^s_b(X_0;V) \ , \  \|-\|_{\mu,\sigma_r, s}\Big ) \hspace{2cm}\Big( \mathcal D^{\sigma,s}(S^1;V)\ , \  \|-\|_{\mathcal D^{\sigma,s}}\Big).
\]
This notation is simply meant  to more easily distinguish the function spaces on $S^1$ versus those on the entire domain. 
When there are also edge derivatives, the norm on $X_0$ is instead denoted $\|-\|_{\mu,\sigma_r, (s,k)}.$

%% file: sections/sec_ellipticanalytic.tex
%%
%% sec_linear_edge_theory.tex
%%

\section{Elliptic Edge Theory on Analytic Spaces}
\label{LinearanalyticSection}
The linearized Jacobi operator $L_0$ restricts to a bounded operator on the `analytic' spaces defined above. In fact, this restriction remain injective since these analytic spaces are subspaces of the standard Sobolev spaces. However, we must still prove that $L_0$ has closed range on these more restricted spaces.  To accomplish this, we extend, in this section, the elliptic edge theory in Sections \ref{PreliminariesSection} and \ref{section4} to the spaces defined in Section \ref{analyticspacessection}.

\subsection{Mapping Properties on Tangentially Analytic Spaces}
It is immediate from Definition \ref{calHebdef} that 
\be 
L_0: r^{\mu+2}\expmultnegEll H^{s,2}_{b,e,D}(X_0; NX_0)\lre r^\mu \expmultnegEll H^s_b(X_0; NX_0)\label{L0analytic} 
\ee
is bounded. In this first subsection, we prove the following analytic analogue of the elliptic estimate of Corollary \ref{higherorderelliptic}(B). 
\begin{prop} \label{injellipticestL0} 
 Let $\mu,\sigma_r$ be as in Subsection \ref{conventionspart2}, with $\sigma<1.$ Then there is a constant $C=C(\mu,s)$ such that for  $\ph\in r^{\mu+2}  \expmultnegEll H^{s,2}_{b,e,D}(X_0; NX_0)$, 
the operator \eqref{L0analytic} is injective with closed range, and
\be 
\|\ph\|_{\mu+2,\sigma, (s,2)}\leq \frac{C(\mu,s)}{(1-\sigma)}\|L_0 \ph\|_{\mu, \sigma, s}\label{analyticinj}
\ee
holds. 
\end{prop}

Note that \eqref{analyticinj} requires {\it a priori} that $\ph$ lie in the domain space. This means that Proposition \ref{injellipticestL0} 
{\it cannot} be used to bootstrap. This is typical for (left) semi-Fredholm elliptic theory, even on standard Sobolev spaces.  

\begin{proof}  The injectivity of this operator is immediate since we already know injectivity on the larger function space
without the exponential weight. Decomposing into eigenmodes, this estimate when $s=0$ reduces to showing that 
\[
||\phi_{j\ell}||_{\mu+2, \sigma, 2} \leq C || L_{j\ell} \phi_{j\ell}||_{\mu, \sigma, 0},
\]
where $L_{j\ell}$ is the operator in \eqref{BesselODE}  and the constant $C$ is independent of $j, \ell$.  Let us now conjugate this
this operator by $e^{\sigma (1-r) |\ell|} r^{(n-1)/2}$. This replaces $L_{j\ell}$ by 
\[
L'_{j\ell} = e^{\sigma (1-r) |\ell|} r^{(n-1)/2} L_{j\ell} e^{-\sigma (1-r) |\ell|} r^{(1-n)/2} : r^{\mu+2} H^2_{b,d}( [0,1], dr)  \longrightarrow r^\mu L^2([0,1], dr).
\]
A short calculation shows that
\begin{equation}
\begin{split}
L'_{j\ell} & = \frac{d^2\,}{dr^2} + 2 \sigma |\ell| \frac{d\,}{dr} + \left( \frac{n-1}{2} - \frac{(n-1)^2}{4} - \lambda_j \right) \frac{1}{r^2} - (1 - \sigma^2) |\ell|^2 \nonumber \\
& = \frac{d^2\,}{dr^2} + 2\sigma |\ell| \frac{d\,}{dr} - \left( \frac{A_j}{r^2} + (1-\sigma^2)|\ell|^2\right) \label{conjugated1},
\end{split}
\end{equation}
where 
\[
A_j := \frac{ (n-1)^2}{4} - \frac{n-1}{2} + \lambda_j = \frac12 (n-1)(n-3) + \lambda_j \geq 0,
\]
since $\lambda_j \geq 0$ by the spectral rigidity hypothesis. In fact, $A_j = 0$ if and only if $n=3$ and $\lambda_j = 0$.  

Using this, we now compute that
\[
\int_0^1 (-L'_{j\ell} u) u\, dr = \int_0^1 \left\| \frac{du}{dr}\right\|^2 + \left(\frac{A_j}{r^2} + (1-\sigma^2)\right)|\ell|^2 |u|^2\, dr \geq 2 (1-\sigma) \int_0^1 |u|^2\, dr
\]
for all $|\ell| \geq 1$. This uses that $u$ vanishes at $r=0$ and $r=1$, and also that $\int_0^1 u u'\, dr = 0$.   We get from this that
\[
2 (1-\sigma) ||u||_{L^2}^2 \leq ||L'_{j\ell} u||_{L^2} ||u||_{L^2} \ \  \Longrightarrow  \ \ 2(1-\sigma) ||u||_{L^2} \leq ||L'_{j\ell} u||_{L^2} \leq
||L'_{j\ell} u||_{r^\mu L^2}.
\]
Since this is an ODE, a simple integration shows that for any $u \in L^2$ with $L'_{j\ell} u \in r^{\mu}L^2$, where $\mu = n/2 - 1 + \epsilon$ as usual,
then $u = c_0 + c_1 r + \ldots + c_k r^k + \tilde{u}$ where $k$ is the largest integer less than $\mu+2$, and the $L^2$ norm of $u$ is
comparable to $\sum |c_j| + ||\tilde{u}||_{r^{\mu+2} L^2}$.   If we know a priori that $u \in r^{\mu+2} L^2$, we conclude that 
\[
||u||_{r^{\mu+2}L^2} \leq \frac{C}{(1-\sigma)} ||L'_{j\ell} u||_{r^\mu L^2},
\]
and finally, using that $||u||_{r^{\mu+2}H^2_b} \leq C( ||L'_{j\ell} u||_{r^\mu L^2} + ||u||_{r^\mu L^2})$ we obtain the estimate on this eigenspace. This concludes the proof when $\ell \neq 0$ and $s = 0$.
The estimate when $\ell = 0$ is elementary.   The proof when $s \neq 0$ follows by commuting $b$-derivatives through $L_0$.
\end{proof}

\medskip 

\begin{rem}  This proof breaks down precisely when $\sigma=1$. In this borderline case, the coefficient of the positive $|\ell|^2$-term in conjugated operator $L'_{j\ell}$ in  \eqref{conjugated1} vanishes, which makes the operator effectively hypoelliptic, with vanishing ellipticity in the directions tangential to $S^1.$ The elliptic estimate now fails, and we can only obtain control of one tangential derivative. This is the borderline behavior referenced in the discussion following Conjecture \ref{conj1}. 
\end{rem}

\subsection{The Formal Dual Operator} 

By Proposition \ref{injellipticestL0}, $L_0$ is injective and has closed range on the analytic spaces \eqref{L0analytic}. 
Just as in Section \ref{section4}, the cokernel is infinite dimensional cokernel and may be identified with the space of solution of the dual problem 
\be  
\Psi\in \text{range}\left(L_0\big|_{r^\mu \expmultnegEll H^{2}_e} \right) \hspace{.5cm}\Longleftrightarrow \hspace{.5cm} L_0^\star\Psi=0\label{analyticdual}
\ee
where $L_0^\star$ denotes an appropriate adjoint operator.   

This duality has the same structure as in Section \ref{dualspacessection}, but now the $r^\mu L^2$-pairing is replaced by the $r^\mu \expmultnegEll L^{2}$-pairing. 
In general, if $\frak W$ is any invertible pseudodifferential operator, we continue to write $\frak W H^{s}_e=\{u  \ | \ \frak W^{-1}u \in H^s_e \}.$ We thus have realizations of the two different dualities:
\smallskip 
\begin{align}
(r^\mu  \expmultnegEll L^{2})^\star  &:= \Big \{ \ v  \ \big | \ \ \exists C_v \ \text{s.t.} \ \br u, v\kt_{r^\mu \expmultnegEll L^{2}} 
\leq C_v \|u\|_{\mu,\sigma}  \ \ \forall u \in r^{\mu}\expmultnegEll L^{2} \Big \} \simeq r^\mu \expmultnegEll L^{2} \nonumber \\ 
(r^\mu \expmultnegEll L^{2})^\dag & := \Big \{ \  \psi   \ \big | \ \ \exists C_\psi \ \text{s.t.} \ \br u, \psi \kt_{L^2} \leq C_\psi \|u\|_{\mu,\sigma}  \ \ 
\forall u \in r^{\mu}\expmultnegEll L^{2} \Big \}= r^{-\mu}\expmultposEll L^2. \label{twodualsanalytic}
\end{align}

\noindent The first of these is the standard Riesz pairing of a Hilbert space with its dual. Likewise, for higher regularity, 
\begin{align}
(r^{\mu+2} \expmultnegEll H^{2}_{e,\Dir})^\star& = r^{\mu-2}\expmultnegEll H^{-2}_e \label{dualstardef1}
\end{align}
\noindent And higher regularity versions $r^{\mu-2}\expmultnegEll H^{s,-2}_{b,e}$ are defined analogously, requiring $b$-derivatives $\nabla_b^\alpha$ for $|\alpha|\leq s$ to lie in the above spaces. Once again, each version of duality has its associated adjoint, where now

\be 
L_0^\star =\left( r^{2\mu} \underline e^{-2\sigma_r D_t} \right) \circ L_0 \circ  \left(r^{-2\mu}  \underline e^{+2\sigma_r D_t}\right). \label{weakL0dual}
\ee

\noindent Expanding out, this is none other than the conjugated operator used in the proof of Proposition \ref{injellipticestL0}.  

The `analytic' obstruction and its complement are now given by 

\begin{align}
\text{\bf Ob}_0^{\sigma,0}:= \text{range}\left(L_0|_{r^{\mu+2}  \expmultnegEll H^{s,2}_{b,e,\Dir}}\right)^\perp \\
\text{\bf Rg}_0^{\sigma,0}:= \text{range}\left(L_0|_{r^{\mu+2}  \expmultnegEll H^{s,2}_{b,e,\Dir}}\right),
\end{align}

\noindent where orthogonality means that in $r^\mu \expmultnegEll L^2.$ The following is then formally identical to Lemma \ref{secondorderinvertibility}. 

\begin{lm} \label{secondorderinvertibilityanalytic}
Let $\mu, \sigma_r$ be as in Subsection \ref{conventionspart2}, so that Proposition \ref{injellipticestL0} applies. Then we have the following: 
\begin{enumerate}
\item[(A)] The operator
$$
L_{0}^\star L_{0}: r^{\mu+2}  \expmultnegEll H_{e,D}^{2} \lre r^{\mu-2}  \expmultnegEll H_{e}^{-2}
$$
is a bounded linear isomorphism with inverse denote $\mathcal P_0.$ 
\item[(B)] The composition   
\begin{center}
\tikzset{node distance=4.2cm, auto}
\begin{tikzpicture}
\node(A){$r^{\mu+2}  \expmultnegEll H^{2}_{e}$};
\node(C)[right of=A]{$ \begin{matrix}\text{\bf Ob}_0^{\sigma,0} \\  \oplus \\\text{\bf Rg}_0^{\sigma,0}  \end{matrix}$};
\node(C')[right of=C, yshift=0cm]{$r^{\mu-2}  \expmultnegEll H^{-2}_{e} $};
\draw[->] (A) to node {$L_{0} $} (C);
\draw[->] (C) to node {$L^\star_{0} $} (C');
\draw[->] (C) to node {$ L_{0}^\star $} (C');
\draw[<-,>=stealth] (A) edge[bend right=20]node[below]{$\mathcal P_{0} $} (C');
\end{tikzpicture}\end{center}
factors through the $\mathcal R^{\sigma,0}_0$ summand. 
\item[(C)] The $r^\mu   \expmultnegEll L^{2}$-orthogonal projections to $\text{\bf Ob}_0^{\sigma,0}, \text{\bf Rg}_0^{\sigma,0}$ are given by  
\[
\Pi_0=\text{Id}-L_{0}\mathcal P_{0} L_{0}^\star \hspace{2cm} \Pi^\perp_0=L_{0}\mathcal P_{0} L_{0}^\star,
\]
\noindent and these respect regularity in the same sense as Lemma \ref{secondorderinvertibility}.
\end{enumerate}
\end{lm}

\subsection{The Analytic Obstruction}
\label{analyticobstructionsection}
Corollary \ref{firstob} readily generalizes to characterize the analytic obstruction as well. The key point in this is the following: although the analytic weight $\sigma_r(r)$ varies radially, the fact that the obstruction modes concentrate exponentially at $r=0$ means the that obstruction {\it inherits the norm with the $r=0$ analytic weight}. 

More precisely, we have the following. We introduce the shorthand 
\be 
\mathrm{e}_{j\ell}^\sigma:= e^{-\sigma|\ell|} e^{i\ell t} \phi_j,
\label{ejlomega}
\ee
which gives an orthonormal basis for $\mathcal D^{\sigma,0}(S^1; \mathcal C_0)$; we call this the analytic spectral basis. 
Recall that as in Corollary \ref{firstob}, the obstruction space in $r^{\mu}L^2$ is spanned by $\Psi_{j\ell}=r^{2\mu}\Psi_{j\ell}^\dag$, see
Definition \ref{obbasisdef}. That is, the obstruction is the image of the basis of $\ker(L_0 | _{r^{-\mu}L^2})$ under the change of weight 
isomorphism identifying the two duals in \eqref{twodualsanalytic}. 
\begin{defn} \label{analyticobstruction}
The {\bf analytic obstruction basis} is given by 
\begin{align}
\Psi_{j\ell}^\sigma=  (1-\sigma)^{\mu+1/2} \  \cdot \  r^{2\mu}e^{2\sigma |\ell| r  }\Psi_{j\ell}^\dag (r) \ \cdot  \mathrm{e}^{\sigma}_{j\ell}  \label{Psijlsigma}
\end{align}
and the {\bf analytic obstruction map} is the associated parameterization 
\be 
\text{ob}_0^\sigma:  \  \mathcal D^{\sigma,0}(S^1; \mathcal C_0)\to\text{\bf Ob}_0^{\sigma,0}  \qquad
\text{ob}_0^\sigma( \eta)= \sum_{j=1}^{J} \sum_{\ell \in \Z} \br \eta(t) \, , \, \mathrm{e}^{\sigma}_{j\ell}\kt_{\mathcal D^{\sigma,0}(S^1;\mathcal C_0)}  
\, \cdot \,  \Psi^\sigma_{jl}. 
\label{obanalyticdef}
\ee
\end{defn}

The following is the analogue of Corollary  \ref{firstob} in this setting. 
\begin{lm}
\label{analyticoblm} 
With $\{\Psi_{j\ell}^\sigma\}_{j\ell \in \Z^J}$ and $\text{ob}_0^\sigma$ as in Definition \ref{analyticobstruction}, 
\begin{itemize}
\item[(A)] $\text{ob}^\sigma_0$ is a bounded linear isometry. 
\item[(B)] the inverse is given by the sequence of inner products 
\[ 
(\text{ob}_0^\sigma)^{-1}(\Psi)= \sum_{j=1}^{J} \sum_{\ell \in \Z} \br \Psi, \Psi^\sigma_{j\ell}\kt_{r^\mu  \expmultnegEll   L^{2}}  \ \cdot \   e_{j\ell}^\sigma 
\] 
for $\Psi \in \text{\bf Ob}_0^{\sigma,0} $, and 
\item[(C)] $\text{ob}_0^\sigma$ respects regularity in the sense that it restricts to a bounded linear isomorphism
\[ 
\text{ob}_0^\sigma: \mathcal D^{s,\sigma}(S^1; \mathcal C_0) \lre \text{\bf Ob}_0^{s,\sigma} .
\]
\end{itemize}
\end{lm}
\begin{proof}
For each $j,\ell$,  $\Psi_{j\ell}^\sigma$ is the unique normalized solutions of $L_0^\star \Psi_{j\ell}^\sigma=0$ in its Fourier mode and the given function space. 
We thus calculate that the $\smash{r^{\mu}\expmultnegEll L^{2}}$-norm is finite if and only if 
\[
\eta(t)=\sum_{j,\ell} \eta_{j\ell} \cdot \mathrm{e}_{j\ell}^\sigma
\]
has finite $\mathcal D^{\sigma,0}(S^1; \mathcal C_0)$-norm, i.e. $\{\eta_{j\ell}\}\in \ell^2(\Z ^{J(\mu)})$. 

Focus first on the term involving the modified Bessel function of the second kind: 
\begin{align*}
\Big \| \text{ob}_0^\sigma(\eta)\Big \|^2_{r^{\mu} \expmultnegEll L^{2}}&= \Big \| \sum_{j,\ell} \eta_\ell  \Psi_{j\ell}^\sigma\Big \|^2_{r^{\mu} \expmultnegEll L^{2}}= \sum_{j,\ell}\int_{X_0} |  \eta_{j\ell}\Psi_{j\ell}^\sigma|^2  e^{2|\ell|\sigma(1-r)}r^{-2\mu} r^{n-1}dr dt dv_\Sigma \\
&=2\pi  (1-\sigma)^{2\mu+1}  \sum_{j,\ell} |\eta_{j\ell}|^2 \int_0^1 | \Psi_{j\ell}^\dag(r)|^2  e^{2\sigma|\ell| r} r^{2\mu+n-1} dr
\end{align*}
using \eqref{Psijlsigma}. Recalling from Lemma  \ref{obnormalization} that $\Psi_{j\ell}^\dag=D_j |\ell|^{\mu+1}r^{\tfrac{2-n}{2}}K_{\nu_j}(|\ell|r)$ 
up to a term decaying like $O(e^{-|\ell|})$, this becomes
\begin{align*}
&= 2\pi (1-\sigma)^{2\mu+1}  \sum_{j,\ell}|D_j|^2 |\ell|^{2\mu+2} |\eta_{j\ell}|^2   \int_0^1  r^{2\mu+1} |K_{\nu_j}(|\ell|r)|^2 e^{2\sigma |\ell|r} dr.\\
\intertext{So setting $\rho=|\ell| r$, }
&= 2\pi (1-\sigma)^{2\mu+1} \sum_{j,\ell}|D_j|^2  |\eta_{j\ell}|^2  \int_0^{|\ell|}  \rho^{2\mu+1} |K_{\nu_j}(\rho)|^2 e^{2\sigma \rho} d\rho\\
& \simeq 2\pi  (1-\sigma)^{2\mu+1} \sum_{j,\ell}|D_j|^2 |\eta_{j\ell}|^2  \int_0^{\infty}  \rho^{2\mu}  e^{-2(1-\sigma) \rho }d\rho \\
& \simeq  (1-\sigma)^{2\mu+1} (1-\sigma)^{-2\mu-1} \|\{\eta_{j\ell}\}\|^2_{\ell^2}\simeq  \|\eta\|^2_{\mathcal D^{\sigma,0}(S^1; \mathcal C_0)},
\end{align*}
where $\simeq$ means a bound by a multiple above and below, with constant independent of $\ell$ as $|\ell|\to \infty$. The last step uses 
that $K_{\nu_j}\sim \rho^{-1/2}e^{-\rho}$ as $\rho\to \infty$. The term involiving the modified Bessel function of the first kind from 
Lemma \ref{obnormalization} leads to a lower order error. This completes Items (A) and (B). 

For Item (C), since $\nabla_t=\ell$ and $r\del_r = \rho \del_\rho$, each $b$-derivative contributes only an additional power of $|\ell|$ 
to the norm. One needs here the classical formula 
\be 
\rho\del_\rho K_{\nu_j}(\rho)=\nu K_{\nu_j}(|\ell| r) -\rho K_{\nu_j+1}(\rho) \label{Kderiv}
\ee
and similarly for $I_{\nu_j}(\rho)$.
\end{proof}

In a similar fashion, the ``unexpectedly'' good boundedness from Lemma \ref{projectionsmoothing} generalizes to this setting. The proof is identical to that lemma, since an additional exponential factor of $e^{(1-\sigma)\rho}$ for $\sigma<1$ does not disrupt the finiteness of the integrals there. Using the notation of that lemma and of \eqref{effectiveordersum}, where $A^i \in \expmultnegEll C^{s_i+2}_b$ now, we have the following. 

\begin{cor}  \label{projectionsmoothingb}Suppose that $\Psi_f$ is a term of effective order $m$, where $f(t)\in \mathcal D^{\sigma,s}.$ Then $  \text{ob}_0^{-1}\circ \Pi_0(\Psi) \in \mathcal D^{\sigma,s-m}(S^1;\mathcal C_0),$ and  

\[\| \text{ob}_0^{-1}\circ\Pi_0(\Psi)\|_{\mathcal D^{\sigma,s-m}(S^1; \mathcal C_0)} \leq C \sum_{i=1}^M\|\psi_i\|_{\mathcal D^{\sigma,s_i}}.\]

\noindent In particular, there is a gain in tangential regularity of order $\alpha+ \tfrac{n}{2}-\mu.$ Here, the constant can depend on up to the $\expmultnegEll C^{s_i+2}_b$-norm of $A^i.$
\end{cor}

%% file: sections/sec_arclengthgauge.tex
\section{The Arclength Gauge} 
\label{arclengthgaugesection}
As explained in the beginning of Section \ref{analyticspacessection}, Theorem \ref{mainb} follows by repeating the Implicit Function Theorem 
argument used to prove Theorem \ref{maina}, but now in the analytic spaces we have defined. The one salient difference is that it is far more
delicate to control the perturbations in the function spaces in Definition \ref{calHebdef} than before. In particular, different analytic metrics 
generally lead to {\it inequivalent} norms on these spaces. This section ameliorates this issue by defining a refined domain for the mean curvature.
This involves a tame Fr\'echet manifold defined by a particular choice of coordinate gauge. 

\subsection{Tame Fr\'echet Thickenings and Submanifolds}
To explain the problem more precisely, consider what happens if we were to replace the spaces $r^\mu H^s_b$ by $r^\mu \expmultnegEll H^s_b,$ 
and then attempted to repeat the proof of Theorem \ref{maina} in Sections \ref{deformationssection}--\ref{tameestimate}. The argument in this
new setting breaks down because Corollary \ref{taunorms} fails in these analytic spaces. A perturbation $\xi$ of the singular set has arclength 
$2\pi(1+\gamma_\xi)$, with $\gamma_\xi$ as in \eqref{txidef}.  Since the spectrum of $D_t=\sqrt{\Delta_t}$ depends on this arclength, 
such a perturbation changes the measure of exponential regularity by $\sigma\mapsto (1+\gamma_\xi)^{-1} \sigma$. 
To carry out the results in Section \ref{obstructionperturbations}, one must work with arclength coordinates on the $S^1$ factor.  
We therefore define a space of submanifolds with edge singularities determined by vector fields lying in $ r^\mu \expmultnegEll H^s_b$,
% from Definition \ref{calHebdef}, 
but using {\it weights and derivatives in arclength coordinates}.  We must then replace the equivalence of spaces in Corollary \ref{taunorms} by
bundles of spaces parametrized in terms of this arclength.

This definition is {\it a priori} slightly circular. Namely, we require that $(\xi,\alpha,\Phi)$ lies in a function spaces determined by the arclength 
coordinate relative to {\it its own} pullback metric. We remedy this by defining the domain implicitly as a ``slice'' in a thickening of the function 
spaces where the arclength may vary. 

Recall the exponential map 
\[
\text{Exp}: \mathcal U_0 \times \mathcal U_1 \to \Big\{ (X_{\xi,\alpha,\Phi}, g_0) \ | \ (\xi,\alpha,\Phi)\in \mathcal U_0\times \mathcal U_1\Big\}
\]
from Definition \ref{edgesubmanifolds}, which parameterizes submanifolds with edge singularities. Given any Banach (respectively, tame Fr\'echet) 
spaces $\mathcal H_0,\mathcal H_1\subseteq C^\infty,$ the collection of submanifolds in the image of $\text{Exp}\left((\mathcal H_0 \times 
\mathcal H_1) \cap (\mathcal U_0\times \mathcal U_1)\right)$ is again a Banach (respectively, tame Fr\'echet) manifold, on which the restriction 
of the universal mean curvature (see Definition \ref{smoothtameMC}) is defined. More generally, this also holds if $\mathcal H_0\times \mathcal H_1$ 
is itself a Banach or tame Fr\'echet submanifold of $C^\infty.$ 

We now define the `thickening'. %which is the ambient space the arclength gauge is ``cut out'' of. 
For small $\gamma_0>0$ which will be specified shortly, and any $\gamma \in (-\gamma_0, \gamma_0),$ write $\tau_\gamma$ for the arclength 
coordinate on the circle $S^1_\tau$ with arclength $2\pi(1+\gamma).$ Denote by $D_\tau=\sqrt{\Delta_\tau}$ the corresponding spectral multiplier. 
Fixing $s_0$ and $\delta>0$ as before, set 
\begin{align}
\mathcal U_0^\gamma&=\Big\{ (\xi,\alpha) \in \mathcal D^{\sigma,\infty}(S^1_\tau; \mathcal B_0) \ \Big |  \ \|\xi\|_{s_0,\sigma}< \delta \Big\}\\
\mathcal U_1^\gamma&=\Big\{ \Phi \in r^{\mu+2}\expmultnegtau H^{\infty,2}_{b,e}(X_0^\tau) \ \Big |  \ \|\Phi\|_{\mu+2,\sigma, (s_0,2)}< \delta \Big\}, 
\label{calUgamma}
\end{align} 
where $X_\tau=S^1_\tau \times C_1(\Sigma).$ As before, there is an associated collection of manifolds with edges determined by these subspaces (with,
for example, $\mathcal H_0=\mathcal D^{\sigma,\infty}$). Now set 
\be 
\mathbb U_\text{thick}^{\mu,\sigma,s_0}:=\Big\{ (\xi,\alpha,\Phi, \gamma) \ \Big |  \  (\xi,\alpha) \in \mathcal U_0^\gamma \ , \ \Phi \in 
\mathcal U_1^\gamma \ , \ \gamma \in (-\gamma_0, \gamma_0) \Big\}.\label{Xthicc} 
\ee
This is the total space of a bundle of open smooth tame Fr\'echet manifolds over $(-\gamma_0,\gamma_0)$. More precisely, multiplication by the 
$\gamma$-dependent family of weights $e^{-\sigma D_\tau}$ and $\expmultnegtau$ on the two factors respectively identifies the fibers with 
open subsets of $C^\infty.$ 
\begin{defn} We define 
\begin{itemize}
\item[(A)]The {\bf Arclength Gauge Condition} is defined by $\mathbb A^{-1}(0)$ for the map
\begin{align}
\mathbb A: \mathbb U^{\mu,\sigma,s_0}_\text{thick} &\lre \R \\ 
\mathbb A(\xi,\alpha,\Phi, \gamma)&:= 2\pi(1+\gamma)-|\Gamma_\xi| \label{bbGgauge}
\end{align}
$\Gamma_\xi$ is the total arclength of $\xi^\perp$, see \eqref{smoothtameMC}.
\item[(B)] The {\bf Arclength Gauge Slice} of regularity $s_0$ is defined by 
\be 
\mathbb M^{\mu, \sigma,s_0}_\text{arc}:=\mathbb A^{-1}(0). 
\ee
 \end{itemize}
 \end{defn}
Note that $\mathbb A(\xi,\alpha,\Phi,\gamma)=0$ is the precise requirement that $(\xi,\alpha) \in \mathcal D^{\sigma,\infty}(S^1_{\tau(\xi)}; \mathcal B_0)$, 
i.e., $\xi$ has analytic regularity measured with respect to its own arclength coordinate. We write $\pi_\mathbb U: \mathbb U^{s_0,\infty}_\text{thick}\to (-\gamma_0, 
\gamma_0)$ for the projection map. 

We have the following: 
\begin{cor} \label{uniformizationgaugeconstruction} The arclength gauge slice
\[
\mathbb M_\text{arc}^{\mu, \sigma,s_0,}\subseteq \mathbb U^{\mu,\sigma,s_0}_\text{thick}
\]
is a smooth tame graph over the fiber $\pi_{\mathbb U}^{-1}(0).$ In particular, projection to the vertical tangent bundle extends to a bounded linear isomorphism $\text{d}p_{V}:T\mathbb M_\text{arc}^{\mu,\sigma,s} \to VT\mathbb U^{\mu, \sigma,s}_\text{thick}$ for all $ s\geq s_0.$ 
\end{cor}
\begin{proof} This follows from a minor application of the Nash-Moser Implicit Function Theorem \ref{nashmoser}. The derivative of \eqref{bbGgauge} is 
clearly an isomorphism on the $\gamma$ factor, thus it suffices to show that $|\Gamma_\xi|$ is a smooth tame map. But this too is obvious because 
it is a map into a finite-dimensional space which depends on only finitely many derivatives of $\xi$ (in the trivialization defined following \eqref{Xthicc}).  

Therefore, 
\[
\mathbb M_\text{arc}^{\mu,\sigma,s}:=\Big\{(\xi,\alpha,\Phi, \gamma(\xi)) \ \big | \  (\xi,\alpha,\Phi) \in \mathcal U_0^\gamma \times \mathcal U_1^\gamma\Big \}.
\]
i.e., $\gamma$ is defined implicitly as a function of the first coordinate. We tacitly write $(\xi,\alpha,\Phi) \in \mathbb M^{\mu,\sigma,s_0}_{\text{arc}}$ to 
denote tuples in the arclength slice that would more properly have four components. The (regularity $s$ completion of the) tangent space is 
\begin{align}
T^s_{(\xi,\alpha,\Phi)} \mathbb M^{\mu,\sigma,s_0}_\text{arc}&:=\ker\Big(2\pi \dot \gamma - d|\Gamma|\Big)\\
& \subseteq \mathcal D^{s,\sigma}(S^1_\tau;\mathcal B_0) \ \oplus  \ r^{\mu+2} 
\expmultnegtau H^{s,2}_{b,e}(X_0^\tau; NX_0)  \ \oplus \  \R \ \simeq \ T^s_{(\xi,\alpha,\Phi,\gamma)}\mathbb U^{\mu,\sigma,s_0}_{\text{thick}}. 
\end{align}
For $s_0$, the norm of the projection to the vertical tangent bundle of  $(\dot \xi, \dot \alpha, \dot \Phi, \dot \gamma)$, i.e., the projection to the first 
three factors) satisfies
\[
\|(\dot \xi, \dot \alpha, \dot \Phi)\|_{s_0} \leq \|(\dot \xi, \dot \alpha, \dot \Phi, \dot \gamma)\|_{s_0} \leq C_{s_0} \|(\dot \xi, \dot \alpha, \dot \Phi)\|_{s_0}.  
\]
The first inequality is trivial, and the second follows because $\displaystyle{2\pi \dot\gamma=d|\Gamma|(\dot \xi)
=\int_0^{2\pi} \frac{2\xi \dot \xi }{\sqrt{1+|\xi'(s)|^2}}ds}.$ 

For $s\geq s_0,$ this extends to a tame equivalence of norms. This yields the desired bounded, hence we obtain
the (linear) isomorphism with the vertical tangent bundle.
\end{proof}

\subsection{Fr\'echet Spaces of Analytic Functions} 
We now define the following: 
\begin{align} 
 \mathbb X^{\mu,\sigma,\infty}&:= \bigcap_{s\geq s_0} r^{\mu+2} \expmultnegtau H^{s,2}_{b,e}(X_0^\tau; NX_0) \hspace{1.5cm}  \label{bbXdef} \\ 
 \medskip
 \mathbb D^{\sigma,\infty}&:=\bigcap_{s\geq s_0} \mathcal D^{\sigma,s+2}(S^1_\tau;\mathcal B_0)  \\ 
 \mathbb P^{\infty}&:= \bigcap_{s\geq 0} H^{s+3/2}(\del X_0; NX_0)^\perp, \label{bbPdef}
\end{align}
Similarly, we write
\[
\mathbb Y^{\mu,\sigma,s}:=\bigsqcup_{(\xi,\alpha,\Phi)} r^\mu  \expmultnegEll H^s_b(X_0^\tau; NX_{\xi,\alpha,\Phi}) 
\qquad \mathbb Y^{\mu,\infty}:=\bigcap_{s\geq s_0} \mathbb Y^{\mu,\sigma,s}.
\]
Then $\mathbb Y^{\mu,\sigma,\infty}\to \mathbb M^{\mu,\sigma,\infty}_\text{arc}$ is a tame Fr\'echet vector bundl which is, trivialized by the projection $P_0$,
just as in Lemma \ref{Ytrivialization}. 
\begin{cor} The following hold: 
\begin{enumerate}
\item[(A)] The universal mean curvature from Definition \ref{smoothtameMC} restricts to a smooth tame section 
\be 
\tikzset{node distance=3.0cm, auto}
\begin{tikzpicture}
\node(D)[yshift=-1.5cm,xshift=.5cm]{$ \mathbb M^{\mu,\sigma,\infty}_\text{arc}.$};
\node(Y)[yshift=-1.5cm][above of=D]{$  \mathbb Y^{\mu,\sigma,\infty} $};
\node(Q)[yshift=-2.3cm][xshift=-.7cm][above of=D]{$   {\mathbb Q} $};
%\node(E)[overlay,yshift=-2.3cm][xshift=5.0cm][above of=D]{$\begin{matrix}{ \mathbb Q}(\xi,\alpha,\ph,\varrho)=\mathcal Q(g_{\xi^\perp,\alpha^\perp,\Phi}) \\ \Phi=\ph + \text{ext}((\xi_0,\alpha_0) + \varrho)\end{matrix}$};
%\node(Y'')[xshift=-2.2cm][yshift=-.1cm][right of=D]{$=$};
%\node(Y')[xshift=0.6cm, yshift=-.2cm][right of=D]{$ \left\{ X_{\Phi,\xi} \   \Big | \  \  \begin{matrix} \text{reg}(X_{\Phi,\xi})=\text{graph}(\Phi)  \\ \text{sing}(X_{\Phi,\xi})=\text{graph}(\xi) ) \end{matrix}\right\} $};
%\node(D')[xshift=-.8cm, yshift=0.9cm]{$\Theta $};
\draw[->](Y) to node {$$} (D);
\draw[->][swap] (D) edge[bend left=45]node[below] {$$} (Y);
%\draw[->] (D) to node {$\widetilde {\slashed {\mathbb D}}$} (E);
\end{tikzpicture}
\label{vectorbundles.}
\ee
\item[(B)] The linearized universal mean curvature at $u=(\xi,\alpha,\Phi) \in \mathbb M^{\mu,\sigma,\infty}_\text{arc}$ (acting on vector fields $\ph$ with 
vanishing Dirichlet boundary values at $r=1$) extends to a bounded operator on the tangent spaces 
\[
\mathbb L_u:=(\text{d}\mathbb Q)_{u}: \mathcal D^{s,\sigma}(S^1_\tau; \mathcal B_0) \oplus  \mathbb X^{\mu,\sigma,s}_{D}(X_0^\tau; NX_0) 
\lre  r^{\mu} \expmultnegtau H^{s}_b(X_0^\tau; NX_0)
\]
where $\tau=\tau(\xi)$ is the arclength coordinate of $\xi.$ 
\item[(C)] The linearization with respect to $\Phi$ is given by 
\be 
L_{\xi,\alpha,\Phi}= L_\gamma + \frak d_{\xi,\alpha,\Phi} \qquad \frak d_{\xi,\alpha,\Phi}=\sum_{|I|\leq 2} \frac{\frak a_{\xi,\alpha,\Phi}^I(t,r,x)}{r^2} \nabla_e^I.
\label{Ldecompsecond}
\ee
Here $\frak d_{\xi,\alpha,\Phi}$ is a second-order edge operator, with $\frak a_{\xi,\alpha,\Phi}^{I}\in r \expmultnegtau  C^{s}_b(X_0; End(NX_0))$,
satisfying $\smash{\|\frak a_{\xi,\alpha,\Phi}^{I}\|_{r \expmultnegtau C^{s_0}_b}\leq C\delta}$ in the $s_0$-norm. 
\end{enumerate}
\end{cor}
\begin{proof} For (A), the proof of Corollary \eqref{tameMC} applies verbatim, now using the multiplication and composition theorems on the 
analytic spaces from Propositions \ref{multiplications} and \ref{compositionfull}. (B) is simply the definition of this partial linearization as in 
Definition \ref{extendedJacobi}, using the identification with the vertical tangent space from Corollary \ref{uniformizationgaugeconstruction}, 
and the proof of (C) is identical to the one for Lemma \ref{Jacobioperatoredge}. 
\end{proof}

Somewhat informally, the upshot of this discussion is that the arclength gauge ensures that $(\xi,\alpha,\Phi)$ have analytic regularity with respect to 
the arclength coordinate $\tau=\tau_\xi.$ The mean curvature \eqref{trueMC} is then a non-linear combination of terms having analytic regularity with 
respect to this precise $\tau_\xi$. The multiplication and compositions theorems for this particular analytic regularity then ensure that it takes values 
in the $\mathbb Y$ space with the same analytic regularity. The linearization, by definition, freezes the analytic regularity at $\tau_\xi$, 
so $\mathbb L_u$ only involves multiplications, such as between the basepoint $\xi$ and the first-order deformation $\eta$. Both of these 
have analytic regularity $\tau_\xi.$   

We note finally that analyticity in the arclength coordinate $\tau$ and analyticity in the ambient global coordinate $t$ are related as follows:
\begin{lm} \label{relatinganalyticity}
There is constant $C_0$ such if $\xi \in \mathcal D^{\sigma,\infty}(S^1_\tau; \R^N)$ has norm $\|\xi\|_{s_0,\sigma}<\delta$ then $\xi \in \mathcal D^{\sigma',\infty}
(S^1_t;\R^N)$, where $0 < \sigma' <  \sigma - C\delta^2$ provided $\delta<<1.$
\end{lm}
\begin{proof} Let $\mathcal G_{\sigma'}(S^1_t),\mathcal G_{\sigma}(S^1_\tau), $ denote the Grauert tubes defined by the flat metric and the pullback 
metric $(1+|d\xi'(t)|^2)dt^2 = (1+\gamma)^2 d\tau^2$ respectively, with radii $\sigma',\sigma>0.$ For $\Gamma_\xi=\{(t,\xi(t))\}$ the graph, let 
$p: S^1_\tau \to S^1_t$ be the map defined by the projection to the first coordinate in the ambient space. We claim that $p$ admits a holomorphic 
extension to $\mathcal G_{\sigma}(S^1_\tau).$ Set 
\[
h(\tau):= \d{t}{\tau}=\sqrt{(1+\gamma)^2 - |d\xi(\tau)|^2}.
\]
Using Lemma \ref{1dmultiplications}, we see that $h(\tau)$ admits a holomorphic extension $H$ to $\mathcal G_{\sigma}(S^1_\tau),$ using the 
obvious branch of the square root. Moreover, on the zero section, $H(\tau)=\tfrac{dt}{d\tau}.$ Now let 
\[
P(z):= \int_0^z H(z) dz  \ \ \mod 2\pi, 
\]
which likewise admits a holomorphic extension to $\mathcal G_\sigma(S^1_\tau).$ Since $P(\tau)=\int_0^\tau h(\tau)d\tau=t$ on the zero section, this
is the desired holomorphic extension. Moreover, the assumption on the norm implies $P(z)=1 + O(\delta^2)$, hence $P$ is local diffeomorphism. Since 
$|P(z)-P(w)|\geq (1-C\delta^2) |z-w|,$ it is also injective. It follows that $P: \mathcal G_\sigma(S^1_\tau)\to \mathcal G(S^1_t)$ is a biholomorphism
onto its image, which contains the tube of radius $\sigma'.$ Since the inverse is also holomorphic, if $\Xi$ is the  holomorphic extension of $\xi(\tau)$, the 
composition 
\begin{center}
\tikzset{node distance=2.2cm, auto}
\begin{tikzpicture}
\node(A){$\mathcal G_{\sigma'}(S^1_{t})$};
\node(C)[right of=A,xshift=.7cm]{$\mathcal G_{\sigma}(S^1_\tau)$};
\node(C')[right of=C, yshift=0cm]{$ \C $};
\draw[->] (A) to node {$ P^{-1}$} (C);
\draw[->] (C) to node {$\Xi $} (C');
\end{tikzpicture}\end{center}
is the desired holomorphic extension of $\xi$ on the Grauert tube in $t$-coordinates.   
\end{proof}

\begin{rem} This construction of the arclength gauge is a simple instance of a general principle. The Fr\'echet domain can be altered to be any bundle of 
spaces with weights depending on the basepoint; in general, $\gamma$ is replaced by any necessary weight data, and $\mathbb A$ by the condition 
defining the weights at particular basepoint, which is always a graph over the ``vertical'' fiber as in the arclength case, provided the defining condition 
is smooth and tame. A previous version of this article included a stronger ``Uniformization Gauge'' which defines analyticity with respect to the pair of 
uniformized coordinates $(\rho,\tau)$ on the family of cylinders $S^1\times [0,1]\subseteq X_0.$ This trick may be useful in other settings, e.g., cases 
where the indicial set of the linearization depends on the basepoint. 
\end{rem}

\subsection{Smooth Boundary Traces and Extensions}
As explained in Section \ref{analyticspacessection}, the spaces with radially-dependent analyticity interpolate between analytic spaces at $r=0$ and 
smooth spaces at $r=1.$ This is captured by the commutative diagram
\be 
\tikzset{node distance=4.0cm, auto}
\begin{tikzpicture}
\node(C) {$  r^{\mu+2}H^s_b(X_0^\tau)$};
\node(C')[right of=C, yshift=0cm]{$ H^s(\del X_0) $};
\node(D)[below of=C, yshift=1.7cm]{$r^{\mu+2}\expmultnegtau H^s_b(X_0^\tau)$};
\node(D')[below of=C', yshift=1.7cm]{$ H^s(\del X_0)$};
\draw[->] (C) to node {$\text{tr}$} (C');
\draw[->] (D) to node {$\text{tr}$} (D');
\draw[->] [swap](C) to node {$ \frak W^{-1}(r)$} (D);
\draw[<-] [swap](D') to node {$\text{Id}$} (C');
%\draw[->] (C) to node {$ L_\xi^\star $} (C');
%\draw[->,>=stealth] (A) edge[bend right=20]node[below]{$L_0\circ \underline P_0 $} (C');
\end{tikzpicture}
\label{extdiagram}
\ee
where $\frak W(r)=\underline e^{\sigma_r D_\tau}$ denote thes spectral weight as before. Note that by construction $\frak W^{-1}(1)=\text{Id}$ is 
the weight at the $r=1$ boundary. %More formally, we have the following: 
\begin{lm} There is an extension map 
\[
\text{ext}^\sigma: H^s(\del X_0)\to r^{\mu+2}\expmultnegtau H^s_b(X_0^\tau)
\]
such that the following hold: 
\begin{enumerate}
\item[(A)] $\text{tr}\circ \text{ext}^\sigma=\text{Id}_{\del X_0},$
\item[(B)] $\text{ext}^\sigma=\frak W^{-1}(r)\circ \text{ext}$ 
\end{enumerate}
\noindent where $\text{ext}$ is the extension from \eqref{Sobolevext}.
\label{analyticextension}
\end{lm}
\begin{proof} (B) may be taken as the definition, after which boundedness and (A) are immediate from the commutativity of \eqref{extdiagram}. 
\end{proof}

%% file: sections/sec_analyticbootstrapping.tex
\section{Tame Analytic Estimates}
\label{analytictame}
The proof of Theorem \ref{mainb} is now formally identical to that of Theorem \ref{maina}. We spell out only a few of the details. 

Just as in \eqref{extendedL0}, the linearization $\mathbb L_0$ at $u=(0,0,0)$ is a bounded operator
\be
\mathbb L_0(\eta,\beta,\ph)=\begin{pmatrix} \Pi_0 \circ \text{d}\mathbb Q( \cdot , \cdot, 0  ) & & 0  \\  \\ \Pi_0^\perp \circ \text{d}\mathbb Q(\cdot,\cdot,0)  & &  L_0 \end{pmatrix} \begin{pmatrix} (\eta,\beta) \\  \\ \ph \end{pmatrix} : \begin{matrix}  \mathcal D^{\sigma,2}(S^1 ;\mathcal B_0)\\ \oplus \\ r^{\mu+2} \expmultnegtau H^{2}_{e,D}(X_0) \end{matrix} \ \ \lre \ \ \begin{matrix} \text{\bf Ob}_0^{\sigma,0} \\ \oplus   \\ \text{\bf Rg}_0^{\sigma,0}\end{matrix}.
\label{extendedL0}
\ee
The formula for the deformation operator in Theorem \ref{boundarydeformation} now gives the
\begin{cor} \label{analyticcor1}For $\sigma \in (0,1)$, the deformation operator restricts to an isomorphism 
\[
T_0: \mathcal D^{\sigma, 3}(S^1; \R^N)\oplus \mathcal D^{\sigma,2}(S^1;\R^K)\lre \mathcal D^{\sigma, 2-\epsilon}(S^1;\mathcal C_0).
\]
where $\mu+1=\tfrac{n}{2}+\epsilon$ as before. This inverse has norm $||T_0^{-1}||\leq C (1-\sigma)^{-\mu-1/2}.$ 
\end{cor}

\medskip 

We conclude that
\medskip
\be
{\mathbb L}_0(\eta,\beta,\ph)=\begin{pmatrix} T_0  \\  \\ \Pi_0^\perp \circ (\text{d}\mathbb Q)_0(\_,\_,0)  & &  L_0 \end{pmatrix} 
\begin{pmatrix} (\eta,\beta) \\  \\ \ph \end{pmatrix} : \begin{matrix}   (\mathcal D^{\sigma,3} \oplus \mathcal D^{\sigma,2})\\ 
\oplus \\ r^{\mu+2}\expmultneg H^{2}_{e,D} \end{matrix} \ \ \lre \ \ \begin{matrix}   \mathcal D^{\sigma,2-\epsilon}(S^1;\mathcal C_0)\\ \oplus   \\ \text{\bf Rg}_0^{\sigma,0}\end{matrix} 
\label{extendedL0}
\ee

\smallskip
\noindent which is the analogue of Corollary \ref{invertibleLbb0}. Next, a careful inspection of Section \ref{obstructionperturbations} shows that the proofs there 
all stem from the elliptic estimate \eqref{dualellipticestimate}, which is dual to the estimate of Corollary \ref{higherorderelliptic}, see Lemma \ref{secondorderinvertibility}.
Repeating the same arguments verbatim, using Proposition \ref{injellipticestL0} (and the resulting estimates for $L_\gamma + \frak d_u$ and its dual), 
for $\sigma\in (0,1)$, we obtain the
\begin{cor}  \label{analyticcor2}
Fixing $\sigma \in (0,1)$, there is an analytic Riesz basis $ \Psi_{j\ell}^{\sigma,u}$, as well as an obstruction map $\text{ob}_u^\sigma$, such that 
\begin{enumerate} 
\item[(A)] These satisfy the analogues of the conclusions of Proposition \ref{ReiszBasis},
\item[(B)] The difference $(\Psi_{j\ell}^{u,\sigma}-\Psi_{j\ell}^{\circ,\sigma})$ satisfies the analogue of the bound in Proposition  \ref{Psidecomp},
\item[(C)] The composition $(\text{ob}_u^\sigma)^{-1}$ has the same ``unexpectedly'' good boundedness as in Proposition \ref{UGBstable} (and Corollary \ref{projectionsmoothingb}) 
\end{enumerate}
Moreover, all bounds diverge like $O((1-\sigma)^{-1})$ as $\sigma \to 1.$
\end{cor}
The sections $\Psi_{j\ell}^{\circ,\sigma}$ here are the analytic obstruction basis at the product space from Definition \ref{analyticobstruction}. 
The analogous statements all refer to the analytic spaces with weights defined by the arclength coordinate $\tau=\tau_\xi$ at $u=(\xi,\alpha,\Phi).$

We now prove Theorem \ref{mainb}.
\begin{proof}[Proof of Theorem \ref{mainb}] Using Corollaries \ref{analyticcor1} and \ref{analyticcor2}, the interpolation inequalities from 
Proposition \ref{multiplications} and the elliptic estimate of Proposition \ref{injellipticestL0} in place of the analogous statements for standard 
Sobolev spaces, the tame estimates in Section \ref{tameestimate} follow from the arguments therein. By the Nash-Moser Implicit Function 
Theorem, we conclude that for an open neighborhood $\mathcal V_2\subseteq C^\infty(\del X_0; NX_0)^\perp,$ there are three smooth tame 
implicitly defined maps
\begin{align}
\xi_\sigma: \mathcal V_2 &\to \mathcal D^{\sigma, \infty}(S^1_\tau;\R^N) \label{analyticspaces1}\\ 
\alpha_\sigma: \mathcal V_2 &\to \mathcal D^{\sigma,\infty}(S^1_\tau; \R^K) \\ 
\ph_\sigma: \mathcal V_2 &\to r^{1+\epsilon} \expmultnegtau C_b^\infty(X_0 ; NX_0) 
\label{analyticspaces3}
\end{align}
such that
\[
\mathbb Q\Big(\xi^\sigma(\varrho), \alpha^\sigma(\varrho), \ph^\sigma(\varrho), \varrho\Big)=0.
\]

The corresponding submanifolds $( X_{\xi_\sigma,\alpha_\sigma,\Phi_\sigma} \ , \ g_0)$ from Definition  \ref{edgesubmanifolds} are minimal in the 
Euclidean product metric, where $\Phi$ \eqref{bigPhiagain}, now using the analytic extension map from Lemma \ref{analyticextension}. 
The polynomial divergence of all estimates in Corollaries \ref{analyticcor1} and \ref{analyticcor2} show that $\mathcal V_2$ may be taken 
to have radius $\delta<C_0(1-\sigma)^p$ for some $p\geq 2.$ By Lemma \ref{relatinganalyticity}, the singular set of
$\text{sing}(X_{\xi_\sigma,\alpha_\sigma,\Phi_\sigma})=\Gamma_{\xi^\perp_\sigma}$ is the graph of a function $\xi_\sigma^\perp(t)\in \mathcal D^{\sigma',\infty}(S^1_t;\R^N)$ 
with analytic regularity $\sigma'< \sigma - C\delta^2.$ In particular, when $\sigma$ is sufficiently close to $1,$ and $\delta$ sufficiently small, 
we may assume that $1-2(1-\sigma) < \sigma'< \sigma<1.$

By the uniqueness statement in Theorem \ref{nashmoser}, this $\xi_\sigma^\perp$ must be the same section $\xi^\perp$ defining the singular set 
in Theorem \ref{maina}. Indeed,  since \eqref{analyticspaces1}--\eqref{analyticspaces3} are (continuously) included in their smooth analogues from
\eqref{standardspaces1}--\eqref{standardspaces3}, and the boundary traces coincide by \eqref{extdiagram}, the tuples $(\xi,\alpha,\Phi,\varrho)$ and
$(\xi_\sigma,\alpha_\sigma,\Phi_\sigma,\varrho)$ both lie in $\mathbb Q^{-1}(0)$ and have the same boundary value parameter, they must be equal. 
Equivalently, the following diagram commutes, where the bottom row denotes the spaces with analytic regularity: 
\be 
\tikzset{node distance=5.0cm, auto}
\begin{tikzpicture}
\node(C) {$  (C^\infty(S^1)\times  rC^\infty_b \times C^\infty(\del X_0) )\cap \mathbb Q^{-1}(0)$};
\node(C')[right of=C,xshift=2cm, yshift=0cm]{$  C^\infty_0(S^1; \R^N) $};
\node(D)[below of=C, yshift=1.7cm]{$(\mathcal D^{\sigma,\infty}\times  r \expmultnegtau C^\infty_b  \times C^\infty(\del X_0))\cap \mathbb Q^{-1}(0) $};
\node(D')[below of=C',yshift=1.7cm]{$ \mathcal D^{\sigma,\infty}_0(S^1; \R^N)$};
\node(B)[left of=C, xshift=-.5cm,yshift=-1.65cm]{$ \mathcal V_2 $};
\node(B')[above of=C, yshift=-1.7cm]{$ C^\infty(X_0; NX_0)^\perp $};
\draw[->] (C) to node {$ \pi_\perp$} (C');
\draw[->] (D) to node {$ \pi^\sigma_\perp$} (D');
\draw[right hook->] (D) to (C);
\draw[right hook->] (D') to (C');
\draw[right hook->] (B) to node[swap] {$(\xi,\alpha,\ph,\text{Id})$} (C);
\draw[right hook->] (B) to node[swap] {$(\xi_\sigma,\alpha_\sigma,\ph_\sigma,\text{Id})$}  (D);
\draw[right hook->, densely dotted] (B) to node {$\varrho\mapsto \varrho$} (B');
\draw[->] (C) to node {$ \text{tr}\big |_{r=1}$} (B');
%\draw[->,>=stealth] (A) edge[bend right=20]node[below]{$L_0\circ \underline P_0 $} (C');
\end{tikzpicture}
\label{extdiagram}
\ee
Here $\pi_\perp,\pi_\perp^\sigma$ denote the projections to the non-constant Fourier modes of the first component of tuples $(\xi,\alpha,\ph,\varrho).$ 
Chasing the value $\varrho \in \mathcal V_2$ to the top right yields $\xi^\perp=\xi_\sigma^\perp.$ We conclude that $\text{sing}(X_{\xi,\alpha,\Phi})
=\text{graph}(\xi^\perp)$ is analytic, with analytic regularity $\sigma'.$  Relabeling $\sigma=\sigma'$, we have now completed the proof. 
\end{proof}